\documentclass[a4paper,11pt,leqno]{article}

\usepackage[latin1]{inputenc}
\usepackage[T1]{fontenc} 
\usepackage[english]{babel}
\usepackage{verbatim}
\usepackage{calligra}
\usepackage{enumerate}
\usepackage{dsfont}
\usepackage{amsfonts}
\usepackage{amsmath}
\usepackage{amsthm}
\usepackage{amssymb}
\usepackage{mathtools}
\usepackage{mathrsfs}
\usepackage{color}
\usepackage{graphicx}
\usepackage{bm}

\usepackage{tikz}
\usepackage[retainorgcmds]{IEEEtrantools}

\usepackage{graphicx}		
\usepackage[hypcap=false]{caption}

\DeclareMathOperator*{\tend}{\longrightarrow}

\DeclareMathOperator*{\D}{\rm{div}}

\theoremstyle{definition}
\newtheorem{defi}{Definition}[section]

\newtheorem{rmk}[defi]{Remark}

\theoremstyle{plane}
\newtheorem{thm}[defi]{Theorem}
\newtheorem{prop}[defi]{Proposition}
\newtheorem{cor}[defi]{Corollary}
\newtheorem{lemma}[defi]{Lemma}

\newcommand{\tbf}{\textbf}

\newcommand{\tsl}{\textsl}

\newcommand{\mbb}{\mathbb}

\newcommand{\mc}{\mathcal}
\newcommand{\mf}{\mathfrak}
\newcommand{\mds}{\mathds}

\newcommand{\veps}{\varepsilon}
\newcommand{\eps}{\veps}
\newcommand{\what}{\widehat}
\newcommand{\wtilde}{\widetilde}
\newcommand{\vphi}{\varphi}
\newcommand{\oline}{\overline}

\newcommand{\ra}{\rightarrow}

\newcommand{\g}{\gamma}
\renewcommand{\k}{\kappa}
\newcommand{\s}{\sigma}
\renewcommand{\t}{\tau}
\newcommand{\z}{\zeta}
\newcommand{\lam}{\lambda}
\newcommand{\de}{\delta}

\newcommand{\lan}{\langle}
\newcommand{\ran}{\rangle}

\newcommand{\R}{\mathbb{R}}

\newcommand{\C}{\mathbb{C}}
\newcommand{\N}{\mathbb{N}}
\newcommand{\Z}{\mathbb{Z}}
\newcommand{\T}{\mathbb{T}}
\renewcommand{\P}{\mathbb{P}}

\renewcommand{\L}{\mc{L}}
\newcommand{\E}{\mc{E}}
\renewcommand{\H}{\mc{H}}

\renewcommand{\div}{{\rm div}\,}

\newcommand{\curl}{{\rm curl}\,}

\newcommand{\Id}{{\rm Id}\,}
\newcommand{\Supp}{{\rm Supp}\,}

\newcommand{\dx}{ \, {\rm d} x}
\newcommand{\dt}{ \, {\rm d} t}
\newcommand{\B}{B^s_{\infty, r}}
\newcommand{\al}{\alpha}
\newcommand{\bt}{\beta}

\allowdisplaybreaks

\def\d{\partial}
\def\div{{\rm div}\,}

\begin{document}

\newcommand{\dimitri}[1]{\textcolor{red}{[***DC: #1 ***]}}
\newcommand{\fra}[1]{\textcolor{blue}{[***FF: #1 ***]}}

\title{\textsc{\Large{\textbf{Els\"asser formulation of the ideal MHD \\ and
improved lifespan in two space dimensions}}}}

\author{\normalsize\textsl{Dimitri Cobb}$\,^1\qquad$ and $\qquad$
\textsl{Francesco Fanelli}$\,^{2}$ \vspace{.5cm} \\
\footnotesize{$\,^{1,} \,^2\;$ \textsc{Universit\'e de Lyon, Universit\'e Claude Bernard Lyon 1}}  \vspace{.1cm} \\
{\footnotesize \it Institut Camille Jordan -- UMR 5208}  \vspace{.1cm}\\
{\footnotesize 43 blvd. du 11 novembre 1918, F-69622 Villeurbanne cedex, FRANCE} \vspace{.2cm} \\
\footnotesize{$\,^{1}\;$\ttfamily{cobb@math.univ-lyon1.fr}}, $\;$
\footnotesize{$\,^{2}\;$\ttfamily{fanelli@math.univ-lyon1.fr}}
\vspace{.2cm}
}

\date\today

\maketitle

\subsubsection*{Abstract}
{\footnotesize In the present paper, we show an improved lower bound for the lifespan of the solutions to the ideal MHD equations in the case of space dimension $d=2$.
In particular, for small initial magnetic fields $b_0$ of size (say) $\veps>0$, the lifespan $T_\veps>0$ of the corresponding solution goes to $+\infty$ in the limit $\veps\ra0^+$.

Such a result does not follow from standard quasi-linear hyperbolic theory.
For proving it, 
three are the crucial ingredients: first of all, to work in endpoint Besov spaces $B^s_{\infty,r}$, under the condition $s>1$ and
$r\in[1,+\infty]$ or $s=r=1$; moreover, to use the Els\"asser formulation of the ideal MHD, recasted in its vorticity formulation; finally, to take advantage of the special structure
of the non-linear terms.

We also rigorously establish the equivalence between the original formulation of the ideal MHD and its Els\"asser formulation for a large class of weak solutions. The construction
of explicit counterexamples shows the sharpness of our assumptions. Related non-uniqueness issues are discussed as well.

}

\paragraph*{\small 2010 Mathematics Subject Classification:}{\footnotesize 35Q35 
(primary);
76W05, 
76B03, 
35L60 
(secondary).}

\paragraph*{\small Keywords: }{\footnotesize ideal MHD; Els\"asser variables; vorticity; critical regularity; improved lifespan in $2$-D.}

{\footnotesize
\tableofcontents
}


\section{Introduction} \label{s:intro}

In this paper, we are concerned with the well-posedness issue for the following system of PDEs, which we set in $\R_+\times\R^d$, with $d\geq2$ (the case of the flat torus $\T^d$ may be covered
with minor modifications):
\begin{equation}\label{i_eq:MHD-I}
\begin{cases}
\partial_t R + \D \big( R u \big) = 0\\[1ex]
\partial_t u + \D(u \otimes u - b \otimes b) + R \mathfrak{C}u + \nabla \left( \Pi + \dfrac{1}{2} |b|^2 \right) = 0\\[1ex]
\partial_t b + \D(u \otimes b - b \otimes u) = 0\\[1ex]
\D(u) = \div(b) = 0\,.
\end{cases}
\end{equation}
The previous system shares evident similarities with the \emph{ideal MHD system}, which we will recall in a while and which can be recovered from \eqref{i_eq:MHD-I} by setting $R\equiv0$.
Here and in what follows, ``ideal'' always means non-viscous and non-resistive.

Equations \eqref{i_eq:MHD-I} have been rigorously derived in \cite{Cobb-F_Rig} in the case of space dimension $d=2$. 
In certain physical regimes, they give an accurate description of the dynamics of slightly non-homogeneous ideal magnetofluids,
whose motion is subject to large forces depending linearly on the density variations and on the velocity field. 
See Subsection \ref{ss:i_qh} for more details.

The vector fields $u=u(t,x)\in\R^d$ and $b=b(t,x)\in\R^d$ stand for, respectively, the velocity field of the fluid and the self-generated magnetic field acting on the flow.
The scalar function $\Pi=\Pi(t,x)\in\R$ is the usual hydrodynamic fluid pressure, while 
$R=R(t,x)\in\R$ represents the density oscillation function
around a constant reference state. With this interpretation in mind, we will, throughout this text, refer to \eqref{i_eq:MHD-I} as the
\emph{quasi-homogeneous ideal MHD system}.

Similarly to a density, the function $R$ is simply advected by the flow. On the other hand, $R$ representing departures from the constant equilibrium, it is not relevant to impose any
sign condition on it. The coupling with the MHD part of the system takes place in the momentum equation, through the term $R\mf Cu$. Here, the tensor $\mf C\in \mc M_d(\R)$
is a $d \times d$ real-valued matrix, which is assumed to have constant coefficients.

Finally, we define the MHD pressure $\pi$ by
\begin{equation} \label{eq:MHD-p}
\pi\,:=\, \Pi\, +\, \frac{1}{2}\, |b|^2\,.
\end{equation}
The MHD pressure is the sum of the hydrodynamic pressure $\Pi$ and the magnetic pressure $|b|^2 /2$, which is itself the (opposite of the) trace of the magnetic stress tensor
$\sigma_m\,:=\,b \otimes b\, -\, \frac{1}{2}\, |b|^2\, {\rm Id}$. The MHD pressure function $\pi$ acts as a Lagrange multiplier related to the incompressibility constraint $\D (u) = 0$.

\subsection{On the ideal magnetohydrodynamics} \label{ss:i_idealMHD}


When $R\equiv0$, one recovers from \eqref{i_eq:MHD-I} the classical (homogeneous) ideal MHD system
\begin{equation}\label{i_eq:ideal}
\begin{cases}
\partial_t u + \D(u \otimes u - b \otimes b) + \nabla \left( \Pi + \dfrac{1}{2} |b|^2 \right) = 0\\[1ex]
\partial_t b + \D(u \otimes b - b \otimes u) = 0\\[1ex]
\D(u) = \div(b) = 0\,.
\end{cases}
\end{equation}
This system has been broadly studied in the past thirty years. If reviewing the whole literature on the subject goes beyond the scopes of the present paper, let us however give some
insights on the structure of the equations and well-known results.

\subsubsection{Previous results} \label{sss:i_wp}
From the mathematical viewpoint, equations \eqref{i_eq:ideal} are a quasi-linear symmetric hyperbolic system of first order. In particular, solutions to that system (formally) conserve
the total energy, defined as the sum of the kinetic and magnetic energies: for smooth enough solutions, one has
\begin{equation} \label{i_eq:id-en}
\frac{\rm d}{\dt} \Big( \|u\|^2_{L^2} + \|b\|^2_{L^2} \Big)\,=\,0\,.
\end{equation}
Notice that this identity does not hold true anymore if we replace the $L^2$ norms with the $L^p$ norms of the solutions, with $p\neq 2$.
Owing to the symmetric hyperbolic structure of the equations, one is naturally led to consider their well-posedness in energy spaces, and more precisely in Sobolev spaces $H^s$, with $s>s_0:=1+d/2$.
This was done in \cite{Sch} in bounded domains of $\R^d$, supplementing system \eqref{i_eq:ideal} by suitable boundary conditions.

On the other hand, a more careful inspection of equations \eqref{i_eq:ideal} reveals that this system is basically a coupling of transport equations.
As a matter of fact, introducing the so-called \emph{Els\"asser variables}
(sometimes called ``characteristic variables'' and denoted by $z^\pm$), defined by the relations
$$
\alpha\,:=\,u+b\qquad\qquad\mbox{ and }\qquad\qquad \bt\,:=\,u-b\,,
$$
system \eqref{i_eq:ideal} can be recasted in the following form:
\begin{equation}\label{i_eq:Els}
\begin{cases}
\partial_t \alpha + (\beta \cdot \nabla) \alpha + \nabla \pi_1 = 0\\[1ex]
\partial_t \beta + (\alpha \cdot \nabla) \beta + \nabla \pi_2  = 0\\[1ex]
\D(\alpha) = \D(\beta) = 0\,,
\end{cases}
\end{equation}
for two possibly different ``pressure functions'' $\pi_1$ and $\pi_2$. Both scalar functions $\pi_1$ and $\pi_2$ can be seen as Lagrange mutlipliers enforcing the two independent
divergence-free conditions $\D(\al) = \D(\bt) = 0$.
Notice that the homologous condition $\D(b)=0$ for the original magnetic field $b$ is naturally satisfied by solutions of \eqref{i_eq:ideal}, 
provided it is satisfied at the initial time; therefore, no Lagrange multiplier is needed in the equation for $b$ (equivalently, system \eqref{i_eq:ideal}
is not overdetermined). 

The Els\"asser formulation \eqref{i_eq:Els}, well-known to physicists, was already at the basis of the approach of \cite{Sch}, and since then
has been largely exploited in mathematical works.
For instance, the transport structure underlying equations \eqref{i_eq:Els} also makes it possible to propagate
$L^p$ norms of the solution, whenever both $u$ and $b$ are Lipschitz continuous. Starting from this observation, Secchi in \cite{Secchi} extended the result of \cite{Sch} to the class of
Sobolev spaces $W^{k,p}$, under the condition $k>k_0(p):=1+d/p$ and for $p\in\,]1,+\infty[\,$.
In \cite{C-K-S}  Els\"asser variables were used to investigate questions related to energy and magnetic helicity conservation (in
connection with Taylor's conjecture; see also \cite{BBV} for recent developments and further references) and blow-up criteria
for solutions to \eqref{i_eq:ideal}.
The same approach, based on the reformulation \eqref{i_eq:Els}, was also used (somehow in an implicit way) in \cite{MY} to study well-posedness of the ideal MHD system in
critical Besov spaces $B^{1+d/p}_{p,1}(\R^d)$, for any $p\in[1,+\infty]$, and in \cite{C-M-Z} for analogous results in the framework of Triebel-Lizorkin spaces.
We refer to \cite{C-C-M} and \cite{Hmidi} for further results and additional references on the classical ideal MHD equations.
We also refer to \cite{Bellan} for more physical insights about MHD-type systems, their (formal) derivation from the Vlasov equations and their regimes of validity.

\subsubsection{Motivations} \label{sss:i_motiv}
Before going ahead and considering the quasi-homogeneous counterpart \eqref{i_eq:MHD-I} of the ideal MHD equations, let us make two important remarks.

The first observation concerns the \emph{equivalence} between the Els\"asser formulation \eqref{i_eq:Els} and the original formulation \eqref{i_eq:ideal} of the ideal MHD.
Passing from \eqref{i_eq:ideal} to \eqref{i_eq:Els} is based only on algebraic manipulations of the equations, and yields in addition $\nabla\pi_2=\nabla\pi_1=\nabla\pi$, where $\pi$ is the MHD pressure defined in \eqref{eq:MHD-p}. On the contrary, proving that a solution of \eqref{i_eq:Els}, no matter how regular, also solves \eqref{i_eq:ideal} is not completely
clear, since one has to make sure that no ``pressure term'' appears in the magnetic field equation, or, in other words, that $\nabla\pi_1=\nabla\pi_2$.

On the other hand, as the equivalence problem seems to lie with the pressure terms, one could try to get rid of any gradient term in the equations, either the original system or the Els\"asser
one, by applying the Leray projector $\P$ onto the space of divergence-free vector fields.
One is thus led to another formulation of the ideal MHD system. This time, the equivalence between the two projected systems is clear
by purely algebraic arguments. On the other hand, the equivalence of each of these systems with the corresponding non-projected one is not totally obvious. As one may guess, in order to establish
the equivalence, one has to face a problem which is actually quite similar to the one described above.

In this paper, we will explore those issues in great detail, see Section \ref{s:Elsasser} below. In particular, we will exhibit sufficient conditions for the equivalences to hold,
for a fairly large class of weak solutions to the ideal MHD system. 
We will give an overview of our results in this direction in Subsection \ref{ss:overview}; for the time being, we only reveal that 
the previous study has some impact as well on the question of (lack of) uniqueness of solutions to systems \eqref{i_eq:ideal} and \eqref{i_eq:Els}.

\medskip

The second remark we want to make here is that all the previously mentioned well-posedness results for the ideal MHD system are only \emph{local in time}. As a matter of fact, the global
well-posedness issue for equations \eqref{i_eq:ideal} still remains an outstanding open problem in the mathematical theory of fluid mechanics, even in the two dimensional framework.
Interestingly enough, global in time results are missing even in presence of partial dissipation in the system, and more precisely when the equations are viscous but non-resistive
(see \tsl{e.g.} papers \cite{F-MC-R-R_JFA}, \cite{C-MC-R-R} and \cite{F-MC-R-R_ARMA} about this problem). 
It is worth mentioning that, in the reverse case of inviscid but resistive equations, the global existence of solutions has been established in \cite{K}.

Coming back to the ideal situation \eqref{i_eq:ideal}, the results do not improve even when the space dimension is $d=2$. This case plays a special role, because it is well-known
that the incompressible Euler equations, obtained from \eqref{i_eq:ideal} by setting $b\equiv0$, are globally well-posed in the planar case.
If one looks for explicit lower bounds for the existence time of solutions, quasi-linear hyperbolic theory gives that 
the lifespan $T>0$ of a solution $(u,b)$ can be bounded from below by
$$
T\,\geq\,\frac{C}{\left\|\big(u_0\,,\,b_0\big)\right\|_{H^s}}\,,
$$
where $C>0$ is a ``universal'' constant and $s>s_0=1+d/2$. In fact, a similar bound was exhibited in \cite{MY}, with the Sobolev norm $H^s$ replaced by the critical Besov norm $B^{1+d/p}_{p,1}$.
Notice however that, in view of the previous estimate, $T$ may remain bounded even in the regime of small magnetic fields,
specifically when $\|b_0\|_{H^s}\sim\veps$ (with $\veps>0$ small), although system \eqref{i_eq:ideal} formally converges to the Euler equations in that regime. Therefore, the previous
lower bound on $T$ is not very satisfactory in dimension $d=2$.

The study of both issues, namely the equivalence between the various formulations of the ideal MHD system and the lifespan of the solutions in two space dimensions, represents the main motivation
behind this work. For the sake of generality, we will investigate those questions in the context of the quasi-homogeneous system \eqref{i_eq:MHD-I},
but the results will apply straight away also to the classical ideal system \eqref{i_eq:ideal}.

\subsection{The quasi-homogeneous counterpart} \label{ss:i_qh}

In the case of space dimension $d=2$, the quasi-homogeneous equations \eqref{i_eq:MHD-I} have been rigorously derived in \cite{Cobb-F_Rig}
from an incompressible, viscous and resistive MHD system with variable density
and Coriolis force, in the regime of low Rossby number (\tsl{i.e.} fast rotation) and vanishing viscosity and resistivity coefficients, for small density perturbations around a constant state.

Let us briefly (and informally) explain how this works. If we denote by $\rho$ the true density of the fluid, and if the Rossby number is (say) of order $\veps>0$,
the Coriolis force, acting in the momentum equation, can be written in the form
$$
\C(\rho,u)\,:=\,\frac{1}{\veps}\,\rho\,u^\perp\,=\,\frac{1}{\veps}\,\rho\,\big(-u^2,u^1\big)\,.
$$
If the density is supposed to be a small variation of the constant state (say) $1$, we can write $\rho=1+\veps R$. On the one hand, the fluid being incompressible,
$R$ is simply advected by the flow, yielding the first equation in system \eqref{i_eq:MHD-I}. On the other hand, the Coriolis force splits into the sum of two terms:
$$
\C(\rho,u)\,=\,\frac{1}{\veps}\,u^\perp\,+\,R\,u^\perp\,.
$$
Due to incompressibility again, the first term on the right-hand side is a perfect gradient, so it can be absorbed in the pressure term. On the other hand, the second term is of order $1$, 
so it persists in the limit for $\veps$ going to $0$.
This gives rise to the coupling term $R\,\mf Cu$ in the momentum equation, with $\mf C$ of the form
\begin{equation}\label{eq:MatrixC}
\mathfrak{C} = \left(
\begin{array}{cc}
0 & -1\\
1 & 0
\end{array}
\right)\,.
\end{equation}
Taking the limit in all the other terms in the momentum and magnetic field equations is easy, using the \tsl{ansatz} $\rho=1+\veps R$, since the only singular
terms in the equations are the Coriolis force and the pressure gradient, which however disappears from the weak formulation of the system, owing to the divergence-free condition
on the test functions. Finally, the vanishing viscosity and resistivity limits lead to the ideal system \eqref{i_eq:MHD-I}.

We refer to \cite{Cobb-F_Rig} for details on the derivation of the quasi-homogeneous ideal MHD system (see also \cite{Cobb-F} for the derivation of the dissipative counterpart).
To conclude, we point out that here, for the sake of generality, we do not restrict our attention to the tensor $\mf C$ given by \eqref{eq:MatrixC}, and consider
rather a general form for it.


\medbreak
To the best of our knowledge, system \eqref{i_eq:MHD-I} is new in the literature. 
In paper \cite{Cobb-F_Rig} mentioned above, we have studied its well-posedness in Besov spaces $B^{s}_{p,r}(\R^d)$ which are embedded in the space of globally Lipschitz functions $W^{1,\infty}(\R^d)$,
in the case of \emph{finite} values of the integrability index $p\in\,]1,+\infty[\,$. Recall that the continuous embedding $B^s_{p,r}\hookrightarrow W^{1,\infty}$
holds true whenever
\begin{equation}\label{i_eq:Lip2}
s>1+\frac{d}{p}\quad \mbox{ and }\quad r\in [1,+\infty]\,,\qquad\qquad\mbox{ or }\qquad\qquad s=1+\frac{d}{p} \quad \mbox{ and }\quad r=1\,.
\end{equation}
The choice of the functional class is quite natural, after remarking that the structure of the equations is pretty similar to the classical case \eqref{i_eq:ideal},
with the only difference of an additional transport equation for the scalar function $R$ and an additional lower order coupling term $R\,\mf Cu$ in the momentum equation.
Therefore, it is natural to consider well-posedness of system \eqref{i_eq:MHD-I} in functional classes which guarantee the velocity field to be Lipschitz.

In \cite{Cobb-F_Rig}, we have exploited in a fundamental way the Els\"asser formulation of \eqref{i_eq:MHD-I}, which is absolutely analogous to \eqref{i_eq:Els}, with the addition
of the transport equation for $R$ (of course, recasted in terms of $\al$ and $\bt$, see system \eqref{eq:MHDab} below).
In order to get rid of the two ``pressure'' functions $\pi_1$  and $\pi_2$, we applied the Leray projector $\P$ onto the space of divergence-free vector fields and, after a commutator
process, we recasted the system as a coupling of three transport equations by divergence-free vector fields.
At that point, the standard Littlewood-Paley machinery applied with no essential difficulty, allowing to recover local in time well-posedness in spaces $B^s_{p,r}$,
for any $1<p<+\infty$ and under condition \eqref{i_eq:Lip2}. We also established a lower bound for the lifespan of the solutions in that functional framework
(analogous to the one given by hyperbolic theory).

The case $p=+\infty$ was not treated in \cite{Cobb-F_Rig}. There are at least two reasons for that.
The first one is that the Leray projector $\P$ is not well-defined in $L^\infty$-type spaces; therefore, treating the endpoint case $p=+\infty$ requires a different
approach. 
The second reason is pretty much linked with what we have mentiond in Paragraph \ref{sss:i_motiv}:
in a purely $L^\infty$-based framework, 
it is not clear to us that the original formulation
\eqref{i_eq:MHD-I} and the Els\"asser formulation \eqref{i_eq:Els} are equivalent. In the same way, without assuming any integrability for the solutions, the equivalence between each of those
formulations and the corresponding system obtained after application of the operator $\P$ also seems questionable.

We will give more details on how to overcome both issues in the next subsection.

\subsection{Contents of the paper and outline of the main results} \label{ss:overview}


In this subsection, we give an overview of the contents of the paper, and we formulate some rough statements, giving the flavour of our main results. We refer to Section \ref{s:results}
below for the rigorous formulation of our results.

\medbreak

The primary concern of this paper is to better understand the global in time existence issue for solutions of the ideal MHD \eqref{i_eq:ideal}, or its quasi-homogeneous counterpart \eqref{i_eq:MHD-I},
in the two-dimensional case. In this direction, we are able to give only a partial answer, contained in the next statement; we refer to Theorem \ref{th:lifespan} below for a more general
formulation.
\begin{thm} \label{t_i:life}
Let $d=2$. Let $(R_0,u_0,b_0)$ be a set of initial data in the H\"older space $C^{2,\g}$, for some $\g\in\,]0,1[\,$, with $\div(u_0)=\div(b_0)=0$ and such that $u_0$ and $b_0$ are integrable
on $\R^2$.
Assume also that the $C^{2,\g}$ norms of both $R_0$ and $b_0$ are of size $\veps>0$.
Let $(R,u,b)$ be the corresponding solution to system \eqref{i_eq:MHD-I}, and denote by $T=T_\veps>0$ its lifespan.

Then, there exists a constant $C>0$, depending also on the norms of the initial datum, such that one has the following lower bound for $T_\veps$:
$$
T_\veps\,\geq\,C\,\log^5\frac{1}{\veps}\,. 
$$
In the previous inequality, $\log^5$ stands for the fifth iterated logarithm. If $R_0\equiv0$, so that $R\equiv0$, then the $\log^5$ function can be replaced by $\log^3$.

In any case, one has the property $T_\veps\,\longrightarrow\,+\infty$ for $\veps\ra0^+$.
\end{thm}

The previous statement can be seen as an ``asymptotically global'' well-posedness result: for small non-homogeneities $R_0$ and magnetic fields $b_0$, system \eqref{i_eq:MHD-I}
behaves like the $2$-D incompressible Euler, which is globally well-posed, and the lifespan of its solutions tends to become infinite.
In particular, taking $R_0\equiv0$ (so that $R\equiv0$ for all times), we also get the corresponding result for solutions to the classical ideal MHD system \eqref{i_eq:ideal}.

As we are going to see, Theorem \ref{t_i:life} does not rely on a perturbative argument around special equilibria,
like \tsl{e.g.} \cite{Bed-M} and \cite{Elg} for Euler flows and porus media respectively; yet, one has to remark that, contrarily to Theorem \ref{t_i:life}, those results are global.

The basic idea behind Theorem \ref{t_i:life}, already used in \cite{DF} for the non-homogeneous incompressible Euler equations and in \cite{F-L} for a zero-Mach number limit system,
consists in taking advantage of improved estimates for linear transport equations in Besov spaces of regularity index $s=0$. Those estimates (recalled in Theorem \ref{th:AnnInnLinTV} below)
were first discovered by Vishik \cite{Vis} and, with a different proof, by Hmidi and Keraani \cite{HK}. They state that, when $s=0$, the $B^0_{p,r}$ norm of the solution can be bounded
\emph{linearly} with respect to the Lipschitz norm of the velocity field, instead of exponentially (as in classical $B^{s}_{p,r}$ estimates).

Now, keep in mind that we want to work in a framework which ensures the velocity field to be Lipschitz. Therefore, to get closer to a $B^0_{p,r}$ setting, the best we can do
is to take $p=+\infty$, in which case the constraint \eqref{i_eq:Lip2} becomes
\begin{equation}\label{i_eq:Lip}
\noindent s>1\quad \mbox{ and }\quad r\in [1,+\infty]\,,\qquad\qquad\mbox{ or }\qquad\qquad s=r=1\,.
\end{equation}
Observe that taking $p=+\infty$ somewhat destroys the quasi-linear symmetric hyperbolic structure of the equations. Therefore, in order to recover well-posedness in $\B$ spaces, we need to work
with the Els\"asser formulation of system \eqref{i_eq:MHD-I}.
However, as revealed in Paragraph \ref{sss:i_motiv} and recalled at the end of Subsection \ref{ss:i_qh} above, the equivalence between the two systems \eqref{i_eq:MHD-I} and \eqref{i_eq:Els}
becomes unclear in a framework based on $L^\infty$-type conditions. In fact, it is not always possible to apply the Leray projector $\P$ to the equations. Let us make an \tsl{intermezzo}
and comment on these issues for a while.

\medbreak
In Section \ref{s:Elsasser} below, we will explore in great detail the previous questions, in the case of any space dimension $d\geq2$.
In particular, in Theorem \ref{th:symm} we will exhibit sufficient conditions for the equivalence between \eqref{i_eq:MHD-I}
and \eqref{i_eq:Els} to hold for a fairly large class of weak solutions to the two systems (see Definitions \ref{d:MHDIdeal} and \ref{d:MHDab} below).
Our result requires some global integrability on the magnetic field (namely, on the quantity $\al-\bt$), and on the non-linear terms, while no special assumptions on $u$ are needed
at this level. In particular, by the construction of an explicit counterexample, we will prove the \emph{failure} of the equivalence of the two formulations,
in a framework based exclusively on $L^\infty$-type conditions on the solutions, without any additional integrability assumption. 

Similarly, in Theorem \ref{t:equiv-P} we will show sufficient conditions able to guarantee the equivalence between each of the previous systems with the corresponding system, obtained
after application of the Leary projector $\P$. Those conditions consist again in requiring global integrability of all the quantities appearing in the equations; this time,
assumptions have to be made on both $u$ and $b$ (or, equivalently, on $\al$ and $\bt$).
Moreover, we will shape a new counterexample from the previous one, which entails that the equivalence does not hold, in general, in a purely $L^\infty$ framework.

From those considerations, we will also deduce some conclusions about the lack of uniqueness of solutions to systems \eqref{i_eq:MHD-I} and \eqref{i_eq:Els} in a setting based
merely on $L^\infty$-type conditions.
In particular, we will see that uniqueness for the projected systems does not yield (for $L^\infty$ solutions, no matter how regular they are) uniqueness for the corresponding
original systems (namely, the systems where the pressure gradients appear). Notice that all we have said until now 
is not specific of the quasi-homogeneous system \eqref{i_eq:MHD-I}: it applies also to the classical ideal MHD system \eqref{i_eq:ideal} and (for what concerns
the use of the Leray projector $\P$ and consequences) the incompressible Euler equations.

\medbreak
This having been pointed out, let us come back to the question of the time of existence of solutions, and resume with a sketch of the proof to Theorem \ref{t_i:life} above.

In order to bypass the previous difficulties, we impose some integrability conditions on the initial velocity field $u_0$ and the initial magnetic field $b_0$:
specifically, we require $u_0$ and $b_0$ to belong to $L^2(\R^d)$. Notice that such integrability is not needed for $R$, which thus enjoys only $L^\infty$-based regularity properties.
We also point out that the finite energy condition is on the one hand physically relevant, since consistent with the derivation of system \eqref{i_eq:MHD-I} given in \cite{Cobb-F_Rig}, and
on the other hand natural in our context.
Indeed, similarly to \eqref{i_eq:id-en}, we can find a simple energy estimate also for equations \eqref{i_eq:MHD-I}: for regular solutions, a basic energy method yields
\begin{equation*}
\frac{1}{2} \frac{\rm d}{\dt} \int_{\R^d} \left(  |u|^2 + |b|^2 \right)\, \dx\, +\,  \int R\, \mathfrak{C} u \cdot u\, \dx\, =\, 0\,.
\end{equation*}
Now, using H\"older's inequality 
and the fact that the $L^\infty$ norms of $R$ are preserved by pure transport by a divergence-free vector field, namely $\|R(t)\|_{L^\infty} = \|R_0\|_{\infty}$,
we easily find
\begin{equation}\label{eq:BasicEN}
\|u(t)\|_{L^2} + \|b(t)\|_{L^2}\, \leq\, \big( \|u_0\|_{L^2} + \|b_0\|_{L^2} \big)\, e^{c\|R_0\|_{L^\infty} t}\,.
\end{equation}
Note that, if the matrix $\mathfrak{C}$ is skew-symmetric, or equivalently if $\mathfrak{C}y \cdot y = 0$ for all $y \in \mathbb{R}^d$,
then the integral involving $R \mathfrak{C} u\cdot u$ vanishes. This is the case in particular if $d = 2$ and $\mathfrak{C}$ is given by \eqref{eq:MatrixC}.
In that case, one recovers the classical energy conservation \eqref{i_eq:id-en}, namely, after integration in time,
\begin{equation} \label{eq:BasicENantiSymmetric}
\|u(t)\|_{L^2} + \|b(t)\|_{L^2}\, =\, \|u_0\|_{L^2} + \|b_0\|_{L^2}\,.
\end{equation}

Of course, the first step in the proof of Theorem \ref{t_i:life} is an existence and uniqueness result in the previous functional framework, together with
a continuation criterion which guarantees us that we can measure the lifespan of a solution
in the space of lowest regularity index, \tsl{i.e.} for $s=r=1$. Those facts are provided by the next statement (see Theorems \ref{th:BesovWP} and \ref{th:cont-crit} below for the rigorous
formulations).
\begin{thm} \label{t_i:WP}
Let $d\geq2$. Let $(s,r)\in\R\times[1,+\infty]$ satisfy condition \eqref{i_eq:Lip}. Then the quasi-homogeneous ideal MHD system \eqref{i_eq:MHD-I} is well-posed, locally in time,
in the space
$$
\mbb X^s_r\,:=\,\left\{(R,u,b)\in\left(\B(\R^d)\right)^3\;\Big|\quad \div(u)\,=\,\div(b)\,=\,0\qquad\mbox{ and }\qquad u\,,\,b\;\in\,L^2(\R^d)\; \right\}\,.
$$

In addition, let $(R,u,b)$ be a solution of system \eqref{i_eq:MHD-I} on $[0,T[\,\times\R^d$, for some $T>0$, with $\big(R(t),u(t),b(t)\big)\in\mbb X^s_r$ for all $t\in[0,T[\,$.
If $T<+\infty$ and
$$
\int_0^{T} \Big( \big\| \nabla u(t) \big\|_{L^\infty} + \big\| \nabla b(t) \big\|_{L^\infty} \Big) \dt < +\infty\,,
$$
then $(R,u,b)$ can be continued beyond $T$ into a solution of \eqref{i_eq:MHD-I} with the same regularity.
\end{thm}

Once global integrability of solutions is guaranteed, one may think that the method of \cite{Cobb-F_Rig}, based on the application of the Leray projector $\P$ to the equations,
would work and lead to the desired result. Actually, there is still one fundamental reason for avoiding the use of $\P$ and looking for a different approach:
since the couple $(s,r)$ satisfies \eqref{i_eq:Lip}, the lowest regularity space we can reach is $B^{1}_{\infty,1}$, whose regularity is still too high
to apply the estimates of \cite{Vis} and \cite{HK} (which, we recall, require to work in $B^0_{p,r}$ spaces).

At this point, the key remark is that, if $u$ and $b$ belong to $B^{1}_{\infty,1}$, then both the vorticity $\omega\,:=\,\d_1u^2-\d_2u^1$ and the electric current $j\,:=\,\d_1b^2-\d_2b^1$
have the sought regularity, namely they belong to $B^0_{\infty,1}$. Hence, the right way to proceed is to look for $B^{0}_{\infty,1}$ estimates for 
$\omega$ and $j$, or better for their respective counterparts in Els\"asser variables, which we will denote $X\,:=\,\d_1\al^2-\d_2\al^1$ and $Y\,:=\,\d_1\bt^2-\d_2\bt^1$.
On the one hand, notice that the $L^2$ condition on $u$ and $b$, or equivalently on $\al$ and $\bt$, also plays a role at this level, controlling the low frequency part of the solution so as to
reconstruct $\al$ and $\bt$ from their $\curl$, $X$ and $Y$.
On the other hand, we remark that working at the $B^0_{\infty,1}$ regularity level is a source of technical difficulties, since that space is \emph{not} an algebra.
The consequence is that the analysis of the non-linear terms becomes much more involved; however, one succeeds in saving the game by using in a crucial way the special structure of the non-linear
terms.

Before concluding, one last remark is necessary. The argument depicted above would lead to an estimate for the lifespan in terms of $\al$ and $\bt$, and not of $u$ and $b$. In that case,
it would be impossible to get the asymptotically global result of Theorem \ref{t_i:life} in terms of small values of $R_0$ and $b_0$; instead, our argument would yield a similar result,
in \emph{any} space dimension, in terms of small vaues of $R_0$ and $\bt_0$ (see also Remark \ref{r:life-b} below in this respect).
The last key remark coming into play here involves, once again, the special structure of the non-linear terms appearing in the equations for $X$ and $Y$:
the basic observation is that it is possible to bound those non-linear terms \emph{linearly} with respect to the density variation $R$ and the magnetic field $b$.

This fact is absolutely fundamental to get Theorem \ref{t_i:life}. However, invoking the magnetic field $b$ leads us to leave the Els\"asser formulation for a while, and lose its symmetric structure.
The consequence is that the bounds for $b$ involve a \emph{loss of derivatives}, namely one is obliged to use a higher order norm of the velocity field $u$.
This implies that we need to impose a higher regularity assumption on the initial datum (see Theorem \ref{th:lifespan} for the precise statement):
this requirement is possibly technical, but absolutely unavoidable for our argument to work.

In order to handle the loss of derivatives in the previous estimate, we start by bounding $u$ in terms of $\al$ and $\bt$: this
allows us to recover the symmetric framework provided by the Els\"asser formulation \eqref{i_eq:Els}.
Finally, and this is the end of our argument, the fundamental observation is that the higher order norm of the solution stays bounded as soon as the lower regularity norm stays bounded.
Roughly speaking, this is the same idea which stands at the basis of the continuation criterion of Theorem \ref{t_i:WP}: it plays a key role also here, and allows us to
get the desired lower bound for the lifespan $T$.

\subsubsection*{Structure of the paper}
Before concluding this introduction, we give an overview of the paper.

In the next section, we give the rigorous statements of our main results: well-posedness of the quasi-homogeneous ideal MHD system \eqref{i_eq:MHD-I}
in the $\B$ framework, a continuation criterion and the improved lower bound for the lifespan of the solutions in the case of two space dimensions.

In Section \ref{s:tools} we recall some basic material, mainly from Littlewood-Paley theory and Fourier analysis, which is heavily used throughout this work.
In particular, we recall classical transport estimates in Besov spaces, together with the improvement of \cite{Vis} and \cite{HK} in the case of Besov spaces having regularity index
$s=0$.

In Section \ref{s:Elsasser}, we investigate the equivalence between the original formulation of the ideal MHD system, the Els\"asser formulation, and the (original or Els\"asser) formulation
obtained after applying the Leray projector onto the space of divergence-free vector fields. The equivalence of all those formulations is proved under some
global integrability assumptions for $(u,b)$, or equivalently for $(\al,\bt)$, and the non-linear terms. We also construct simple counterexamples, which show that the equivalences
fails in a purely $L^\infty$ setting (without global integrability assumptions).
Considerations about lack of uniqueness of solutions will be made as well.

Section \ref{WP} contains the proof to the local well-posedness result and to the continuation criterion. The proof of the improved lifespan
of solutions in $2$-D is the topic of Section \ref{s:lifespan}, which concludes the paper.

\subsubsection*{Acknowledgements}

{\small
The work of the second author has been partially supported by the LABEX MILYON (ANR-10-LABX-0070) of Universit\'e de Lyon, within the program ``Investissement d'Avenir''
(ANR-11-IDEX-0007),  and by the projects BORDS (ANR-16-CE40-0027-01) and SingFlows (ANR-18-CE40-0027), all operated by the French National Research Agency (ANR).
}

\section{Statement of the main results} \label{s:results}

In this section, we give the rigorous statements of our main results, concerning the well-posedness of problem \eqref{i_eq:MHD-I}.
First of all, we have a local in time existence and uniqueness statement for initial data in the Besov space $\B$, under the Lipschitz condition \eqref{i_eq:Lip} on the indices. As already mentioned,
we ask also for finite energy initial velocity and magnetic fields.

\begin{thm}\label{th:BesovWP}
Let $(s,r)\in\R\times[1, +\infty]$ such that either $s>1$, or $s=r=1$. Let $\big(R_0,u_0, b_0\big)$ be a set of initial data such that $R_0\in\B(\R^d)$ and
$u_0,b_0\,\in L^2(\R^d;\R^d)\cap \B(\R^d;\R^d)$, with $\div(u_0)\,=\,\div(b_0)\,=\,0$.

Then, there exists a time $T > 0$ such that, on $[0,T]\times\R^d$, problem \eqref{i_eq:MHD-I} has a unique solution $(R,u,b)$ with the following properties: if $r < +\infty$,
\begin{itemize}
 \item $R\in C^0\big([0,T];\B(\R^d)\big)\,\cap\,C^1\big([0,T]; B^{s-1}_{\infty,r}(\R^d)\big)$;
 \item both $u$ and $b$ belong to $C^0\big([0,T];\B(\R^d;\R^d)\big)\,\cap\,C^1\big([0,T];L^2(\R^d;\R^d)\cap B^{s-1}_{\infty,r}(\R^d;\R^d)\big)$.
\end{itemize}
In the case $r=+\infty$, one only gets weak continuity in time, and the space $C^0\big([0,T];\B\big)$ has to be replaced by the space $C^0_w\big([0,T];B^s_{\infty,\infty}\big)$.

In addition, the hydrodynamic pressure $\Pi$ enjoys the property $\nabla\Pi\in C^0\big([0,T];L^2(\R^d;\R^d)\cap\B(\R^d;\R^d)\big)$, with the usual modification when $r=+\infty$.

Finally, there exists a constant $C>0$, depending only on the dimension $d$ and on the regularity parameters
$(s,r)$, such that
\begin{equation} \label{est:d-life}
T\, \geq \,\frac{1}{\| R_0 \|_{L^\infty}}\; {\rm argsinh}\! \left( \frac{C\, \| R_0 \|_{L^\infty} }{ \big\|R_0\big\|_{\B}\,+\,\big\|\big(u_0, b_0\big)\big\|_{B^{s}_{\infty, r}\cap L^2}} \right)\,.
\end{equation}
\end{thm}

Before going on, let us make some remarks about the previous statement.
\begin{rmk} \label{r:d-life}
The $\rm argsinh$ function is present in the previous lower bound \eqref{est:d-life} for the lifespan $T$ only because the energy inequality \eqref{eq:BasicEN} is exponential with respect to time.

In the case where $R \equiv 0$, or if the matrix $\mathfrak{C}$ is skew-symmetric, then we have the better energy inequality \eqref{eq:BasicENantiSymmetric}. This yields the lower bound
\begin{equation*}
T\, \geq\, \frac{C}{\big\|R_0\big\|_{\B}\,+\,\big\|\big(u_0, b_0\big)\big\|_{B^{s}_{\infty, r}\cap L^2}}\,,
\end{equation*}
which is the classical lower bound one expects to get for quasi-linear hyperbolic problems.
Note that, for small values of $\|R_0\|_{L^\infty}\,\big(\big\|R_0\big\|_{\B}+\big\|\big(u_0, b_0\big)\big\|_{B^{s}_{\infty, r}\cap L^2}\big)^{-1}$, the lower bound \eqref{est:d-life}
is near to the previous inequality.
\end{rmk}

\begin{rmk}
Uniqueness in Theorem \ref{th:BesovWP} is a consequence of a stability estimate for the Els\"asser variables, see Theorem \ref{th:w-s} below. However, a slight modification
of the argument of the proof would in fact yield a weak-strong uniqueness result: so long as the solution of Theorem \ref{th:BesovWP} exists, this solution
is the unique solution to problem \eqref{i_eq:MHD-I} in the energy space $L^\infty ([0, T] ; L^2 \cap L^\infty \times L^2 \times L^2)$.

See also Remark \ref{r:w-s} below for more comments about this.
\end{rmk}

%

Next, we establish a continuation criterion in the previous functional framework, in the spirit of the classical Beale-Kato-Majda continuation criterion \cite{B-K-M} for solutions
to the incompressible Euler equations.

\begin{thm} \label{th:cont-crit}
Let $T^* > 0$ and let $(R, u, b)$ be a solution to \eqref{i_eq:MHD-I} on $[0,T^*[\,\times\R^d$, enjoying the properties described in Theorem \ref{th:BesovWP} for all $T<T^*$. Assume that 
\begin{equation*}
\int_0^{T^*} \Big( \big\| \nabla u(t) \big\|_{L^\infty} + \big\| \nabla b(t) \big\|_{L^\infty} \Big) \dt < +\infty\,.
\end{equation*}

Then, if $T^*<+\infty$, the triplet $(R, u, b)$ can be continued beyond $T^*$ into a solution of \eqref{i_eq:MHD-I} with the same regularity.
\end{thm}

From the previous statement, we immediately deduce the following result.
\begin{cor} \label{c:reg}
Let $(s,r)\in\R\times[1,+\infty]$ such that either $s>1$, or $s=r=1$. The lifespan of a solution $(R,u,b)$ to system \eqref{i_eq:MHD-I} does not depend on $(s,r)$.
In particular, the lifespan in the space $\B\times\big(L^2\cap\B\big)\times\big(L^2\cap\B\big)$ is the same as the lifespan in
$B^1_{\infty,1}\times\big(L^2\cap B^1_{\infty,1}\big)\times\big(L^2\cap B^1_{\infty,1}\big)$.
\end{cor}

So far, all our results were independent of the dimension. Now, we focus instead on the special case of space dimension $d = 2$:
exploiting the (formal) proximity of the quasi-homogeneous ideal MHD system \eqref{i_eq:MHD-I} with the incompressible Euler equations in the regime of small
$R$ and $b$, we can prove that the solutions may have a very long lifespan, if the initial density variation $R_0$ and initial magnetic field $b_0$ are small.

\begin{thm} \label{th:lifespan}
Consider an initial datum $\big(R_0, u_0, b_0\big)$ such that $R_0\in B^2_{\infty,1}(\R^2)$ and $u_0, b_0 \in L^2(\R^2;\R^2)\cap B^2_{\infty,1}(\R^2;\R^2)$, with $\div(u_0)=\div(b_0)=0$.

Then, the lifespan $T>0$ of the corresponding solution $(R, u, b)$ of the $2$-D quasi-homogeneous ideal MHD problem \eqref{i_eq:MHD-I}, given by Theorem \ref{th:BesovWP} above,
enjoys the following lower bound:
\begin{equation*}
T\,\geq\,\frac{C}{\|R_0\|_{B^2_{\infty,1}}\,+\,\big\| \big(u_0, b_0\big) \big\|_{B^2_{\infty, 1} \cap L^2 }}\;
\bigg[\log\big(1\,+\,C\,\cdot\,\big)\bigg]^{\bigcirc5}\left(\frac{\big\| \big(u_0, b_0\big) \big\|_{B^1_{\infty, 1} \cap L^2} }{\big\| \big(R_0, b_0\big) \big\|_{B^1_{\infty, 1}}}\right)\,,
\end{equation*}
where $C>0$ is a ``universal'' constant, independent of the initial datum, and $\big[\log(1+C\,\cdot\,)\big]^{\bigcirc n}$ is the $n$-th iterated logarithm function $z\mapsto \log(1+Cz)$.
\end{thm}

\begin{rmk} \label{r:lifespan}
As will appear clear in the proof (see Section \ref{s:lifespan}), if $R \equiv 0$ we can replace the fifth iterated logarithm by the better $\big[\log(1+C\,\cdot\,)\big]^{\bigcirc3}$ function.
If instead $R\neq0$ but the matrix $\mf C$ is skew-symmetric, \tsl{i.e.} under the assumption that \eqref{eq:MatrixC} holds (up to a multiplicative constant), then we can replace the fifth iterated logarithm by the
$\big[\log(1+C\,\cdot\,)\big]^{\bigcirc4}$ function.
\end{rmk}

In particular, by taking $R\equiv0$, we see that Theorem \ref{th:lifespan} holds as well for solutions to the classical ideal MHD system \eqref{i_eq:ideal}. Then, solutions
to that system also enjoy the improved lifespan in space dimension $d=2$, and the ``asymptotically global'' well-posedness result of Theorem \ref{t_i:life} also holds true
for \eqref{i_eq:ideal}.

\section{Tools} \label{s:tools}

In this section we recall some tools, mainly from Fourier analysis, which we are going to employ in our study.
Subsections \ref{ss:LP} and \ref{ss:NHPC} are devoted to recall some basic definitions and properties of Littlewood-Paley theory, Besov spaces and paradifferential calculus.
Those notions will find application in Subsection \ref{ss:transport}, which focuses on the study of transport equations in Besov spaces.
Finally, in Subsection \ref{ss:int} we present in full detail a key integrability lemma.

\subsection{Non-homogeneous Littlewood-Paley theory and Besov spaces} \label{ss:LP}

Here we recall the main ideas of Littlewood-Paley theory. We refer to Chapter 2 of \cite{BCD} for details.
For simplicity of exposition, we deal with the $\R^d$ case; however, everything can be adapted to the torus $\T^d$ with minor modifications.

First of all, let us introduce the so-called ``Littlewood-Paley decomposition'', based on a non-homogeneous dyadic partition of unity with
respect to the Fourier variable. 
We fix a smooth radial function $\chi$ supported in the ball $B(0,2)$, equal to $1$ in a neighborhood of $B(0,1)$
and such that $r\mapsto\chi(r\,e)$ is nonincreasing over $\R_+$ for all unitary vectors $e\in\R^d$. Set
$\varphi\left(\xi\right)=\chi\left(\xi\right)-\chi\left(2\xi\right)$ and
$\vphi_j(\xi):=\vphi(2^{-j}\xi)$ for all $j\geq0$.
The dyadic blocks $(\Delta_j)_{j\in\Z}$ are defined by\footnote{Throughout we agree  that  $f(D)$ stands for 
the pseudo-differential operator $u\mapsto\mc{F}^{-1}[f(\xi)\,\what u(\xi)]$.} 
$$
\Delta_j\,:=\,0\quad\mbox{ if }\; j\leq-2,\qquad\Delta_{-1}\,:=\,\chi(D)\qquad\mbox{ and }\qquad
\Delta_j\,:=\,\varphi(2^{-j}D)\quad \mbox{ if }\;  j\geq0\,.
$$
We  also introduce the following low frequency cut-off operator:
\begin{equation} \label{eq:S_j}
S_ju\,:=\,\chi(2^{-j}D)\,=\,\sum_{k\leq j-1}\Delta_{k}\qquad\mbox{ for }\qquad j\geq0\,.
\end{equation}
Note that $S_j$ is a convolution operator. More precisely, if we denote $\mc F(f)\,=\,\what f$ the Fourier transform of a function $f$ and $\mc F^{-1}$
the inverse Fourier transform, after defining
$$
K_0\,:=\,\mc F^{-1}\chi\qquad\qquad\mbox{ and }\qquad\qquad K_j(x)\,:=\,\mathcal{F}^{-1}[\chi (2^{-j}\cdot)] (x) = 2^{jd}K_0(2^j x)\,,
$$
we have, for all $j\in\N$ and all tempered distributions $u\in\mc S'$, that $S_ju\,=\,K_j\,*\,u$.
Thus the $L^1$ norm of $K_j$ is independent of $j\geq0$, hence $S_j$ maps continuously $L^p$ into itself, for any $1 \leq p \leq +\infty$.


The following property holds true: for any $u\in\mc{S}'$, one has the equality~$u=\sum_{j}\Delta_ju$ in the sense of $\mc{S}'$.
Let us also mention the so-called \emph{Bernstein inequalities}, which explain the way derivatives act on spectrally localized functions.
  \begin{lemma} \label{l:bern}
Let  $0<r<R$.   A constant $C$ exists so that, for any nonnegative integer $k$, any couple $(p,q)$ 
in $[1,+\infty]^2$, with  $p\leq q$,  and any function $u\in L^p$,  we  have, for all $\lambda>0$,
$$
\displaylines{
{\Supp}\, \widehat u \subset   B(0,\lambda R)\quad
\Longrightarrow\quad
\|\nabla^k u\|_{L^q}\, \leq\,
 C^{k+1}\,\lambda^{k+d\left(\frac{1}{p}-\frac{1}{q}\right)}\,\|u\|_{L^p}\;;\cr
{\Supp}\, \widehat u \subset \{\xi\in\R^d\,|\, r\lambda\leq|\xi|\leq R\lambda\}
\quad\Longrightarrow\quad C^{-k-1}\,\lambda^k\|u\|_{L^p}\,
\leq\,
\|\nabla^k u\|_{L^p}\,
\leq\,
C^{k+1} \, \lambda^k\|u\|_{L^p}\,.
}$$
\end{lemma}   

By use of Littlewood-Paley decomposition, we can define the class of Besov spaces.
\begin{defi} \label{d:B}
  Let $s\in\R$ and $1\leq p,r\leq+\infty$. The \emph{non-homogeneous Besov space}
$B^{s}_{p,r}\,=\,B^s_{p,r}(\R^d)$ is defined as the subset of tempered distributions $u$ for which
$$
\|u\|_{B^{s}_{p,r}}\,:=\,
\left\|\left(2^{js}\,\|\Delta_ju\|_{L^p}\right)_{j\geq -1}\right\|_{\ell^r}\,<\,+\infty\,.
$$
\end{defi}
Besov spaces are interpolation spaces between Sobolev spaces. In fact, for any $k\in\N$ and~$p\in[1,+\infty]$,
we have the following chain of continuous embeddings:
$$
B^k_{p,1}\hookrightarrow W^{k,p}\hookrightarrow B^k_{p,\infty}\,,
$$
where  $W^{k,p}$ stands for the classical Sobolev space of $L^p$ functions with all the derivatives up to the order $k$ in $L^p$.
When $1<p<+\infty$, we can refine the previous result (this is the non-homogeneous version of Theorems 2.40 and 2.41 in \cite{BCD}): we have
$B^k_{p, \min (p, 2)}\hookrightarrow W^{k,p}\hookrightarrow B^k_{p, \max(p, 2)}$.
In particular, for all $s\in\R$, we deduce the equivalence $B^s_{2,2}\equiv H^s$, with equivalence of norms.

As an immediate consequence of the first Bernstein inequality, one gets the following embedding result.
\begin{prop}\label{p:embed}
The space $B^{s_1}_{p_1,r_1}$ is continuously embedded in the space $B^{s_2}_{p_2,r_2}$ for all indices satisfying $p_1\,\leq\,p_2$ and
$$
s_2\,<\,s_1-d\left(\frac{1}{p_1}-\frac{1}{p_2}\right)\qquad\qquad\mbox{ or }\qquad\qquad
s_2\,=\,s_1-d\left(\frac{1}{p_1}-\frac{1}{p_2}\right)\;\;\mbox{ and }\;\;r_1\,\leq\,r_2\,. 
$$
\end{prop}


In particular, we get the following chain of continuous embeddings: provided that $(s, p, r) \in \mathbb{R} \times [1, +\infty]^2$ satisfies the condition
\begin{equation}\label{eq:AnnLInfty}
s > \frac{d}{p} \qquad\qquad \text{ or } \qquad\qquad s = \frac{d}{p}\quad \text{ and }\quad r = 1\,,
\end{equation}
then we have
\begin{equation*}
B^s_{p,r} \hookrightarrow B^{s - \frac{d}{p}}_{\infty, r} \hookrightarrow B^0_{\infty, 1} \hookrightarrow L^\infty\,.
\end{equation*}

To conclude this part, let us discuss continuity properties of the Leary projector $\P$ on Besov spaces, whose use was fundamental in the approach of \cite{Cobb-F_Rig}.
First of all, we recall that $\mathbb{P}$ is the $L^2$-orthogonal projection on the subspace of divergence-free functions. It is defined by the formula
\begin{equation} \label{eq:P-op}
\mathbb{P}\,=\,m(D)\,:=\, \Id \,+\, \nabla (- \Delta)^{-1} \D\,,
\end{equation}
where $m=m(\xi)$ denotes the symbol of $\P$.
The previous formula has to be interpreted in the sense of Fourier multipliers:
\begin{equation}\label{eq:LerayFM}
\forall\, f \in L^2(\mathbb{R}^d ; \mathbb{R}^d), \qquad\qquad \what{(\mathbb{P}f)_j}(\xi) = \sum_{k=1}^d\left( 1 - \frac{\xi_j \xi_k}{|\xi|^2} \right) \what{f_k}(\xi)\,.
\end{equation}
In general, if $f \in \mathcal{S}'$ is a tempered distribution, the projection $\mathbb{P}f$ is well defined (as a tempered distribution) as long as the product
$m(\xi) \what{f}(\xi)$ is also well defined as a tempered distribution, which is the case if {\tsl{e.g.}} $\what{f}$ is in $L^1_{\rm loc}(\R^d)$.

We also recall that $\P$ is a singular integral operator. By Calder\'on-Zygmund theory, it is therefore a continuous operator from $L^p$ into itself for any $1<p<+\infty$.
Thanks to this property, it is easy to see that $\P$ is also continuous  from $B^s_{p,r}$ into itself, for all $(s, r) \in \mathbb{R} \times [1, +\infty]$, as long as $1 < p < +\infty$.

In the case where $p = 1$ or $p = +\infty$, however, the previous continuity property is no longer true.
Even worse, $\P$ is ill-defined as a Fourier multiplier on $L^\infty$. We refer to the last paragraph of Subsection \ref{ss:integrab} for more comments about this issue.
In this article, where we investigate well-posedness in the $\B$ framework (\tsl{i.e.} for $p=+\infty$), we will resort to the vorticity formulation of the equations in order to avoid the use of $\P$.

\subsection{Non-homogeneous paradifferential calculus}\label{ss:NHPC}

Let us now introduce the paraproduct operator (after J.-M. Bony, see \cite{Bony}). Once again, we refer to Chapter 2 of \cite{BCD} for full details.
Constructing the paraproduct operator relies on the observation that, 
formally, the product  of two tempered distributions $u$ and $v$ may be decomposed into 
\begin{equation*} 
u\,v\;=\;\mathcal{T}_u(v)\,+\,\mathcal{T}_v(u)\,+\,\mathcal{R}(u,v)\,,
\end{equation*}
where we have defined
$$
\mathcal{T}_u(v)\,:=\,\sum_jS_{j-1}u\,\Delta_j v\qquad\qquad\mbox{ and }\qquad\qquad
\mathcal{R}(u,v)\,:=\,\sum_j\sum_{|j'-j|\leq1}\Delta_j u\,\Delta_{j'}v\,.
$$
The above operator $\mc T$ is called ``paraproduct'' whereas
$\mc R$ is called ``remainder''.
The paraproduct and remainder operators have many nice continuity properties. 
The following ones will be of constant use in this paper.
\begin{prop}\label{p:op}
For any $(s,p,r)\in\R\times[1,+\infty]^2$ and $t>0$, the paraproduct operator 
$\mathcal{T}$ maps continuously $L^\infty\times B^s_{p,r}$ in $B^s_{p,r}$ and  $B^{-t}_{\infty,\infty}\times B^s_{p,r}$ in $B^{s-t}_{p,r}$.
Moreover, the following estimates hold:
$$
\|\mathcal{T}_u(v)\|_{B^s_{p,r}}\,\leq\, C\,\|u\|_{L^\infty}\,\|\nabla v\|_{B^{s-1}_{p,r}}\qquad\mbox{ and }\qquad
\|\mathcal{T}_u(v)\|_{B^{s-t}_{p,r}}\,\leq\, C\|u\|_{B^{-t}_{\infty,\infty}}\,\|\nabla v\|_{B^{s-1}_{p,r}}\,.
$$
For any $(s_1,p_1,r_1)$ and $(s_2,p_2,r_2)$ in $\R\times[1,+\infty]^2$ such that 
$s_1+s_2>0$, $1/p:=1/p_1+1/p_2\leq1$ and~$1/r:=1/r_1+1/r_2\leq1$,
the remainder operator $\mathcal{R}$ maps continuously~$B^{s_1}_{p_1,r_1}\times B^{s_2}_{p_2,r_2}$ into~$B^{s_1+s_2}_{p,r}$.
In the case $s_1+s_2=0$, provided $r=1$, the operator $\mathcal{R}$ is continuous from $B^{s_1}_{p_1,r_1}\times B^{s_2}_{p_2,r_2}$ with values
in $B^{0}_{p,\infty}$.
\end{prop}

The consequence of this proposition is that the spaces $\B$ are Banach algebras as long as condition \eqref{eq:AnnLInfty} holds with $s > 0$.
Moreover, in that case, we have the so-called \emph{tame estimates}.

\begin{cor}\label{c:tame}
Let $(s, r)\in\R\times[1,+\infty]$ be such that that $s > 0$. Then, we have
\begin{equation*}
\forall\, f, g \in \B\,, \quad\qquad \| fg \|_{\B}\, \lesssim \,\| f \|_{L^\infty}\, \|g\|_{\B}\, +\, \| f \|_{\B} \,\| g \|_{L^\infty}\,.
\end{equation*}
\end{cor}

\begin{rmk} \label{r:tame}
The space $B^0_{\infty, 1}$ is \emph{not} an algebra. If $f, g \in B^0_{\infty, 1}$, one can use Proposition \ref{p:op} to bound the paraproducts $\mc T_f(g)$ and $\mc T_g(f)$, but not the
remainder $\mathcal{R}(f, g)$. 
\end{rmk}

\subsection{Transport equations and commutator estimates} \label{ss:transport}

In this section, we focus on transport equations in non-homogeneous Besov spaces. We refer to Chapter 3 of \cite{BCD} for a complete presentation of the subject.
We study the initial value problem
\begin{equation}\label{eq:TV}
\begin{cases}
\partial_t f + v \cdot \nabla f = g \\
f_{|t = 0} = f_0\,.
\end{cases}
\end{equation}
We will always assume the velocity field $v=v(t,x)$ to be a Lipschitz divergence-free function, \tsl{i.e.} $\D(v) = 0$.
It is therefore practical to formulate the following definition: the couple $(s, r) \in \mathbb{R} \times [1, +\infty]$ is said to satisfy the Lipschitz condition if 
condition \eqref{i_eq:Lip} holds. This implies the embedding $\B \hookrightarrow W^{1, \infty}$.

The main well-posedness result concerning problem \eqref{eq:TV} in Besov spaces is contained in the following statement, stated in the case
$p=+\infty$ (the only relevant one for our analysis).
We recall here that, when $X$ is Banach, the notation $C^0_w\big([0,T];X\big)$ refers to the space of functions which are continuous in time with values in $X$ endowed with its weak topology.
\begin{thm}\label{th:transport}
Let $(s, r) \in \mathbb{R} \times [1, +\infty]$ satisfy the Lipschitz condition \eqref{i_eq:Lip}. Given some $T>0$, let $g \in L^1_T(\B)$. Assume that $v \in L^1_T(\B)$ and that there exist real numbers
$q > 1$ and $M > 0$ for which $v \in L^q_T(B^{-M}_{\infty, \infty})$. Finally, let $f_0 \in \B$ be an initial datum. Then, the transport equation \eqref{eq:TV} has a unique solution $f$ in:
\begin{itemize}
\item the space $C^0\big([0,T];\B\big)$, if $r < +\infty$;
\item the space $\left( \bigcap_{s'<s} C^0\big([0,T];B^{s'}_{\infty, \infty}\big) \right) \cap C^0_{w}\big([0,T];B^s_{p, \infty}\big)$, if $r = +\infty$.
\end{itemize}
Moreover, this unique solution satisfies the following estimate:
\begin{equation*} 
\| f \|_{L^\infty_T(\B)} \leq \exp \left( C\!\! \int_0^T \| \nabla v \|_{B^{s-1}_{\infty, r}} \right)
\left\{ \| f_0 \|_{\B} + \int_0^T \exp \left( - C\!\! \int_0^t \| \nabla v \|_{B^{s-1}_{\infty, r}} \right) \| g(t) \|_{\B} {\rm d} t  \right\},
\end{equation*}
for some constant $C = C(d, s, r)>0$.
\end{thm}

The proof of the previous statement is based on a classical procedure, which consists in finding a transport equation the dyadic blocks $\Delta_jf$ solve. Several commutator estimates
are then needed. The first one is contained in the next statement (see Lemma 2.100 and Remark 2.101 in \cite{BCD}).
Notice that estimate \eqref{eq:lCommBLinfty} below is not contained in \cite{BCD}, but it easily follows by slight
modifications to the arguments of the proof (see in particular the control of the term $R^3_j$ at pages 113-114 of \cite{BCD}).


\begin{lemma}\label{l:CommBCD}
Assume that $v \in \B$ with $(s, r)$ satisfying the Lipschitz condition  \eqref{i_eq:Lip}.
Denote by $\big[ v \cdot \nabla, \Delta_j \big] f\,=\,(v \cdot \nabla) \Delta_j - \Delta_j (v \cdot \nabla)$ the commutator between the transport operator $v\cdot\nabla$ and the frequency
localisation operator $\Delta_j$. 
Then we have
\begin{equation*}
\forall\, f \in \B\,, \qquad\qquad  2^{js}\left\| \big[ v \cdot \nabla, \Delta_j \big] f  \right\|_{L^\infty} \lesssim c_j \Big( \|\nabla v \|_{L^\infty} \| f \|_{\B} +
\|\nabla v \|_{B^{s-1}_{\infty, r}} \|\nabla f \|_{L^\infty} \Big)\,,
\end{equation*}
and also
\begin{equation}\label{eq:lCommBLinfty}
\forall\, f \in B^{s-1}_{\infty, r}\,, \qquad\qquad
2^{j(s-1)} \left\| \big[ v \cdot \nabla, \Delta_j \big] f  \right\|_{L^\infty} \lesssim c_j \Big( \|\nabla v \|_{L^\infty} \| f \|_{B^{s-1}_{\infty, r}} +
\|\nabla v \|_{B^{s-1}_{\infty, r}} \| f \|_{L^\infty} \Big)\,,
\end{equation}
where the $\big(c_j\big)_{j\geq -1}$ are (possibly distinct) sequences in the unit ball of $\ell^r$. 

\end{lemma}

The second commutator result deals with commutators between paraproduct operators and Fourier multipliers.

\begin{lemma}\label{l:ParaComm}
Let $\k$ be a smooth function on $\mathbb{R}^d$, which is homogeneous of degree $m$ away from a neighborhood of $0$. Then, for a vector field $v$ such that $\nabla v \in L^\infty$, one has:
\begin{equation*}
\forall\, f \in \B\,, \qquad \left\| \big[ \mathcal{T}_v, \k(D) \big] f \right\|_{B^{s-m+1}_{\infty, r}}\, \lesssim\, \|\nabla v\|_{L^\infty} \|f\|_{\B}\,.
\end{equation*}
\end{lemma}


To conclude this part, let us present a refinement of Theorem \ref{th:transport} above, discovered by Vishik \cite{Vis} and, with a different proof, by Hmidi and Keraani \cite{HK}.
It states that, if $\D(v) = 0$ and the Besov regularity index is $s = 0$,
then the estimate in Theorem \ref{th:transport} can be replaced by an inequality which is linear with respect to $\|\nabla v\|_{L^1_T(L^\infty)}$.

\begin{thm}\label{th:AnnInnLinTV}
Assume that $\nabla v \in L^1_T(L^\infty)$ and that $v$ is divergence-free. Let $r \in [1, +\infty]$.
Then there exists a constant $C = C(d)$ such that, for any solution $f$ to problem \eqref{eq:TV} in $C^0\big([0,T];B^0_{\infty,r}\big)$, with the usual modification of $C^0$ into $C^0_w$
if $r=+\infty$, we have
\[ 
\| f \|_{L^\infty_T(B^0_{\infty, r})}\, \leq\, C\, \bigg\{ \| f_0 \|_{B^0_{\infty, r}}\, +\, \| g \|_{L^1_T(B^0_{\infty, r})} \bigg\}\;\left( 1+\int_0^T\| \nabla v(\tau) \|_{L^\infty}{\rm d} \tau \right)\,.
\] 
\end{thm}

\subsection{An integrability lemma} \label{ss:int}

This section is devoted to the proof of Lemma \ref{l:Integrability} below, which plays a crucial role in establishing the equivalence between the original formulation and the Els\"asser
formulation of the quasi-homogeneous ideal MHD system.

This is a standard property, which however is usually stated in a $L^1_{\rm loc}$ setting (this is enough when dealing, for instance, with Navier-Stokes type systems),
see \tsl{e.g.} Lemma 1.1 in Chapter III of \cite{Galdi}. In default of a precise reference, we provide a full proof for reader's convenience.

\begin{lemma}\label{l:Integrability}
Let $T \in \mathcal{S}'(\mathbb{R}^d ; \mathbb{R}^d)$ be a tempered distribution which satisfy 
\begin{equation} \label{hyp:T}
\lan T\,,\,\vphi\ran_{\mc D'\times\mc D}\,=\,0\qquad\qquad\forall\,\vphi\,\in\,\mc D(\Omega)\quad\mbox{ such that }\quad \div(\vphi)=0\,.
\end{equation}

Then there exists another tempered distribution $S \in \mathcal{S}'(\mathbb{R}^d ; \mathbb{R})$ such that $T = \nabla S$.
\end{lemma}

\begin{proof}
To begin with, we notice that, from our assumption, it is easy to deduce (see \tsl{e.g.} the remarks preceding Proposition 1.2 of \cite{CDGG})
that\footnote{From now on, we agree that $\nabla f$ is the transpose matrix of the differential $Df$ of $f$: in other words, $[\nabla f]_{ij} = \partial_i f_j$. Of course, all derivatives here are to be understood in the weak sense.}
\begin{equation} \label{def:vort}
\curl T\,:=\,DT\,-\,\nabla T\,=\,0\qquad\qquad \mbox{ in the sense of } \; \mc S'\,.
\end{equation}
Unless otherwise mentionned, all brackets $\langle \, \cdot \, , \, \cdot \, \rangle$ should be understood in the sense of the $\mc S' \times \mc S$ duality. We work in Fourier variables. 

We must show the existence of a tempered distribution $\what S$ such that $\what{T_j} (\xi) = i \xi_j \what S(\xi)$.
Thanks to the condition $\curl(T) = 0$, this is no problem for frequencies away from $\xi = 0$. Thus, let us fix a smooth $\chi\in C^{\infty}_0(\R^d)$ such that
$\chi(\xi)\equiv1$ in a neighborhood of the origin.

Now, let $\phi\in\mc S(\R^d;\R^d)$ be a Schwartz function, and decompose its Fourier transform into
$$
\what\phi(\xi)\,=\,\chi(\xi)\,\what\phi(\xi)\,+\,\big(1-\chi(\xi)\big)\,\what\phi(\xi)\,.
$$

For the high-frequency part $\big(1-\chi(\xi)\big)\,\what\phi(\xi)$, we can apply the classical Leray-Helmholtz decomposition to write, in Fourier variables, the following identity:
$$
\big(1-\chi(\xi)\big)\what\phi(\xi)\,=\,\xi\big(1-\chi(\xi)\big)\,\frac{\xi\cdot\what\phi(\xi)}{|\xi|^2}\,+\,\psi(\xi)\,,\quad\mbox{ with }\;
\psi(\xi)\,:=\,
\big(1-\chi(\xi)\big)\left(\what\phi(\xi)-\xi\,\frac{\xi\cdot\what\phi(\xi)}{|\xi|^2}\right).
$$
We remark that both terms in the right-hand side of the previous relation belong to $\mc S$. Therefore, after noticing that $\xi\cdot\psi(\xi)=0$ and using hypothesis \eqref{hyp:T} on $T$, we can compute
\begin{align*}
\left\lan T\,,\,\big(1-\chi(D)\big)\,\phi \right\ran\,&=\,\left\lan\what T(\xi)\,,\,\xi\,\big(1-\chi(\xi)\big)\,\frac{\xi\cdot\what\phi(\xi)}{|\xi|^2} \right\ran\;=\;
\left\lan i\,\xi\,\left(\big(1-\chi(\xi)\big)\,\frac{-i\,\xi\cdot\what T(\xi)}{|\xi|^2}\right)\,,\,\what\phi(\xi) \right\ran \\
&=\,\left\lan i\,\xi\,\what S_1(\xi)\,,\,\what\phi(\xi) \right\ran\;=\;\left\lan \nabla S_1\,,\,\phi \right\ran\,.
\end{align*}

Let us now deal with the low-frequency part $\chi(\xi)\,\what\phi(\xi)$.
The obvious problem is that  the term $\xi\cdot \what{T}(\xi)\,|\xi|^{-2}$ may not be defined as a distribution near $\xi = 0$.
To circumvent this obstacle, we ``flatten'' the test function around zero by subtracting a Taylor polynomial: let us introduce, for some integer $N \geq 0$ which we will fix later on,
the $\xi$-function
\begin{equation*}
\mf R[\phi](\xi)\,:=\,
\chi(\xi)\,\left(\what{\phi}(\xi)\,-\,\sum_{|\alpha| \leq N} \frac{\xi^\alpha}{\alpha !} \partial^\alpha \what{\phi}(0)\right)\,.
\end{equation*}
The sum ranges on all multi-indices $\alpha \in \mathbb{N}^d$ such that $|\alpha| = \alpha_1 + \ldots + \alpha_d \leq N$.
Then, for $N$ large enough, $\Phi[\phi](\xi) := \xi\,|\xi|^{-2}\,\mf R[\phi](\xi)$ is of class $C^p$, where $p$ is the order of $\what{T}$ on the compact set $\Supp(\chi)$, so that the bracket
$\langle \what{T}, \Phi \rangle$ is well-defined.

Let us focus for a while on the action of $\what T$ on the polynomial part of $\what\phi$: we have
\begin{align}
\left\langle \what{T} (\xi)\, ,\, \chi(\xi) \sum_{|\alpha| \leq N} \frac{\xi^\alpha}{\alpha !} \partial^\alpha \what{\phi}(0) \right\rangle\,&=\,
\left\langle \sum_{|\alpha| \leq N}  \frac{1}{\alpha !} \left\langle \what{T}(\xi)\, ,\, \xi^\alpha\, \chi(\xi)\,\right\rangle\; \partial^\alpha \delta_0(\xi)\,,\,\what{\phi}(\xi) \right\rangle
\label{eq:IntegrabilityProofEQ1} \\
& := \left\langle \sum_{|\alpha| \leq N} \gamma_\alpha  \, \partial^\alpha \delta_0(\xi)\, , \,\what{\phi}(\xi) \right\rangle\,, \nonumber
\end{align}
where $\de_0$ is the Dirac mass centred at the origin.
As a linear combination of $\delta_0$ and its derivatives is the Fourier transform of a polynomial, we can fix $Q \in \mathbb{R}[x]$ such that the brackets
\eqref{eq:IntegrabilityProofEQ1} are equal to $\langle Q \,,\, \phi \rangle $. Hence, we have shown that
\begin{equation*} 
\langle T\,,\, \chi(D)\,\phi \rangle\, =\, \left\langle \what{T}\,,\,\mf R[\phi] \right\rangle\, +\, \langle Q \,, \,\phi \rangle\,.
\end{equation*}
Performing the same computations with the distribution $\curl(T)\equiv0$ and a matrix-valued Schwartz function $\Theta$, we see that we have
\begin{equation*}
\langle \curl(T)\,,\, \Theta \rangle \,=\, \left\langle \mc F\big(\curl(T)\big)(\xi)\,,\, \mf R[\Theta](\xi)\right\rangle\, +\, \langle \curl(Q)\,,\, \Theta \rangle\, =\, \langle \curl(Q)\,,\, \Theta \rangle\,,
\end{equation*}
whence we deduce that $\curl(Q) = 0$ in $\mathcal{S}'$. This means that there exists a polynomial $S_2 \in \mathbb{R}[x]$ such that $Q = \nabla S_2$.

Let us look now at the bracket $\langle \what{T}, \mf R[\phi] \rangle$. By virtue of hypothesis \eqref{hyp:T} on $T$, arguing exactly as for the high-frequency part, we get
\begin{equation*}
\left\langle \what{T}(\xi)\,,\,\mf R[\phi](\xi) \right\rangle\,=\, 
\left\langle \what{T}(\xi)\,, \,\xi\, \frac{\xi\cdot\mf R[\phi](\xi)}{|\xi|^2}\right\rangle = \left\langle i\,\xi\, \frac{-i\xi\cdot\what T(\xi)}{|\xi|^2}\,,\, \mf R[\phi](\xi) \right\rangle\,.
\end{equation*}
Therefore, by setting
\begin{equation*}
\langle S_3 \,,\, \vphi \rangle\,:=\, \left\langle  \frac{-i\,\xi\cdot \what T(\xi)}{|\xi|^2}\, ,\,\mf R[\vphi](\xi) \right\rangle\,=\,\left\lan -i\,\what T(\xi)\,,\,\Phi[\vphi](\xi) \right\ran
\qquad\qquad\forall\,\vphi\in\mc S(\R^d;\R)\,,
\end{equation*}
we define a bounded linear functional on the space $\mathcal{S}(\mathbb{R}^d)$, so that $S_3$ is a tempered distribution. In addition, by the previous computations we gather
\begin{equation*}
\left\langle \nabla S_3\, ,\, \phi \right\rangle\, =\, \left\langle \what{T}(\xi)\,,\,\mf R[\phi](\xi)  \right\rangle\,.
\end{equation*}
for all $\phi\in\mc S(\R^d;\R^d)$.

In the end, we have shown that, for any $\phi\in \mc S(\R^d;\R^d)$, the following series of equalities holds:
\begin{align*}
\lan T\,,\,\phi\ran\,=\,\left\lan T\,,\,\big(1-\chi(D)\big)\,\phi \right\ran\,+\,\left\lan T\,,\,\chi(D)\,\phi \right\ran\,=\,\left\lan \nabla S_1\,,\,\phi \right\ran\,+\,
\left\lan \nabla \big(S_2+S_3\big)\,,\,\phi \right\ran\,.
\end{align*}
This implies that $T\,=\,\nabla\big(S_1+S_2+S_3\big)$, and the lemma is proved.
\end{proof}

\section{Reformulation using Els\"asser variables} \label{s:Elsasser}

Let us introduce the so-called \emph{Els\"asser variables} $(\alpha,\beta)$, defined
by the transformation
\[ 
\alpha = u + b \qquad \text{ and } \qquad \beta = u - b.
\] 
In the new set of unknowns $(R,\alpha,\beta)$, the quasi-homogeneous ideal MHD system \eqref{i_eq:MHD-I} can be recasted in the following form:
\begin{equation}\label{eq:MHDab}
\begin{cases}
\partial_t R + \D \left( \dfrac{1}{2} R(\al + \bt) \right) = 0 \\[1ex]
\partial_t \alpha + (\beta \cdot \nabla) \alpha + \dfrac{1}{2}  R \mathfrak{C} (\al + \bt) + \nabla \pi_1 = 0\\[1ex]
\partial_t \beta + (\alpha \cdot \nabla) \beta + \dfrac{1}{2}  R \mathfrak{C} (\al + \bt) + \nabla \pi_2  = 0\\[1ex]
\D(\alpha) = \D(\beta) = 0\,,
\end{cases}
\end{equation}
where $\pi_1$ and $\pi_2$ are (possibly distinct) scalar ``pressure'' functions. In fact, as we will see in a while, in our framework we must have $\nabla \pi_1 = \nabla \pi_2$.

The main goal of this section is to establish equivalence of the two formulations \eqref{i_eq:MHD-I} and \eqref{eq:MHDab} for a large class of data and solutions. Namely, we will do this
in the framework of \emph{weak solutions}, as defined in Subsection \ref{ss:def-weak} below. 

\subsection{Weak solutions} \label{ss:def-weak}

In this subsection, we define the notions of \emph{weak solutions} of both the original ideal MHD system \eqref{i_eq:MHD-I} and the new system \eqref{eq:MHDab},
which are of relevance for us.

Specifically, we require that all terms appearing in the equations, apart from the pressure terms, possess some (at least local) integrability on $[0,T[\,\times \R^d$, where $T>0$. The fact of considering \emph{semi open} time
intervals $[0, T[\,$ allows us to include the case of solutions which may potentially blow up for $t \rightarrow T^-$.

\medskip

We start by defining weak solutions of the original quasi-homogeneous ideal MHD system \eqref{i_eq:MHD-I}.

\begin{defi}\label{d:MHDIdeal}
Let $T > 0$. Consider a set of initial data $(R_0, u_0, b_0) \in L^1_{\rm loc}(\R^d)\times L^1_{\rm loc}(\R^d;\R^d)\times L^1_{\rm loc}(\R^d;\R^d)$,
with $\div(u_0)=\D(b_0) = 0$ in the sense of $\mc D'\big([0,T[\,\times\R^d\big)$.
We say that a triplet $(R, u, b)$ of tempered distributions on $[0,T[\,\times\R^d$ is a weak solution to system \eqref{i_eq:MHD-I}, related to the initial datum $(R_0,u_0,b_0)$,
if the following conditions are satisfied:
\begin{enumerate}[(i)]
\item $(R,u,b)$ belongs to $L^1_{\rm loc}\big([0,T[\,\times\R^d\big)\times L^1_{\rm loc}\big([0,T[\,\times\R^d;\R^d\big)\times L^1_{\rm loc}\big([0,T[\,\times\R^d;\R^d\big)$;
\item the non-linear terms $u\otimes u\,-\,b\otimes b$ and $u\otimes b\,-\,b\otimes u$ belong to $L^1_{\rm loc}\big([0,T[\,\times\R^d;\mc M_d(\R)\big)$ and $R\,u$ belongs to
$L^1_{\rm loc}\big([0,T[\,\times\R^d;\R^d\big)$;
\item the equations are satisfied in the weak sense: for all $\phi \in \mathcal{D}\big([0, T[\, \times \mathbb{R}^d\big)$, we have
\begin{equation} \label{eq:weak-R}
\int_0^t \int_{\R^d} \Big\{ R \partial_t \phi + R u \cdot \nabla \phi \Big\} \dx \dt + \int_{\R^d} R_0 \phi(0) \dx = 0\,;
\end{equation}
for all $\psi\in\mc D\big([0,T[\,\times\R^d;\R^d\big)$ such that $\div\psi=0$,
we have
\begin{equation}\label{eq:WeakMomentum}
\int_0^T \int_{\R^d} \Big\{ u \cdot \partial_t \psi + \big( u \otimes u - b \otimes b \big) : \nabla \psi - R \mathfrak{C} u \cdot \psi \Big\} \dx \dt + \int_{\R^d} u_0 \cdot \psi(0) \dx = 0\,;
\end{equation}
for all $\z \in \mathcal{D}\big( [0, T[\, \times \mathbb{R}^d;\R^d\big)$, we have
\begin{equation}\label{eq:WeakMagnetic}
\int_0^T \int_{\R^d} \Big\{ b  \cdot \partial_t \z + \big( u \otimes b - b \otimes u \big) : \nabla \z \Big\} \dx \dt + \int_{\R^d} b_0 \cdot \z(0) \dx = 0\,;
\end{equation}
\item the divergence-free condition holds in the sense of distributions:
\begin{equation*}
\D (u) = 0 \qquad\qquad \text{ in }\qquad \mathcal{D}'\big(\,]0, T[\,\times \mathbb{R}^d\big)\,.
\end{equation*}
\end{enumerate}
\end{defi}

Note that, because the momentum equation is only tested with divergence-free functions, we cannot deduce from \eqref{eq:WeakMomentum} any form of continuity of the solutions with respect to time,
even in the $\mc D'$ topology. This means that the initial datum is to be understood only in the weak sense.

As we will see below, time continuity follows from extra integrability assumptions on the weak solutions $(R, u, b)$. However, in their absence, it is fairly easy to construct weak bounded
solutions which are discontinuous with respect to time (see Subsection \ref{ss:integrab} for more details).

\medskip

Similarly to Definition \ref{d:MHDIdeal}, we can define weak solutions of the symmetrised system \eqref{eq:MHDab}.

\begin{defi}\label{d:MHDab}
Let $T > 0$. Consider a set of initial data $(R_0, \al_0, \bt_0) \in L^1_{\rm loc}(\R^d)\times L^1_{\rm loc}(\R^d;\R^d)\times L^1_{\rm loc}(\R^d;\R^d)$,
with $\div(\al_0)=\D(\bt_0) = 0$ in the sense of $\mc D'\big([0,T[\,\times\R^d\big)$.
We say that a triplet $(R, \al, \bt)$ of tempered distributions on $[0,T[\,\times\R^d$ is a weak solution to system \eqref{eq:MHDab}, related to the initial datum $(R_0,\al_0,\bt_0)$,
if the following conditions are satisfied:
\begin{enumerate}[(i)]
\item $(R,\al,\bt)$ belongs to $L^1_{\rm loc}\big([0,T[\,\times\R^d\big)\times L^1_{\rm loc}\big([0,T[\,\times\R^d;\R^d\big)\times L^1_{\rm loc}\big([0,T[\,\times\R^d;\R^d\big)$;
\item the tensor product $\al\otimes\bt$ belongs to $L^1_{\rm loc}\big([0,T[\,\times\R^d;\mc M_d(\R)\big)$ and the vector
$R\,(\al+\bt)$ belongs to $L^1_{\rm loc}\big([0,T[\,\times\R^d;\R^d\big)$;
\item the equation for $R$ is satisfied in the weak sense, \tsl{i.e.} equation \eqref{eq:weak-R} where we replace $u$ by $(\al+\bt)/2$,
holds for all $\phi \in \mathcal{D}\big([0, T[\, \times \mathbb{R}^d\big)$;
\item both evolution equations for $\al$ and $\bt$ are satisfied in the weak sense: for any test function $\psi \in \mathcal{D} \big( [0, T[\, \times \mathbb{R}^d;\R^d \big)$ such that
$\D (\psi) = 0$, we have
\begin{equation*}
\int_0^T \int_{\R^d} \Big\{ \al \cdot \partial_t \psi + \big( \bt \otimes \al \big) : \nabla \psi - \frac{1}{2} R \mathfrak{C} (\al + \bt ) \cdot \psi \Big\} \dx \dt +
\int_{\R^d} \al_0 \cdot \psi(0) \dx = 0\,,
\end{equation*}
and likewise for the equation on $\bt$;
\item the divergence-free conditions for $\al$ and $\bt$ hold in the sense of distributions:
\begin{equation*}
\D (\al) = \D (\bt) = 0 \qquad\qquad \text{ in }\qquad \mathcal{D}'\big(\,]0, T[\, \times \mathbb{R}^d\big)\,.
\end{equation*}
\end{enumerate}
\end{defi}




To conclude this part, let us notice that Definitions \ref{d:MHDIdeal} and \ref{d:MHDab} are consistent with systems \eqref{i_eq:MHD-I} and \eqref{eq:MHDab}, even though
they do not state the existence of ``pressure'' fields $\pi$, $\pi_1$ and $\pi_2$. In fact, this is a consequence of Lemma \ref{l:Integrability}.
%
%

%




\subsection{Equivalence of the Els\"asser formulation}\label{ss:symm}

In this section, we are concerned with the way the definitions stated above interact. We are going to prove that both systems \eqref{i_eq:MHD-I} and \eqref{eq:MHDab} are equivalent
for a large class of weak solutions, provided they satisfy some global integrability properties.


\begin{thm}\label{th:symm}
Let $T > 0$. 
%
\begin{enumerate}[(i)]
 \item Consider a set of initial data $(R_0,u_0,b_0)$ enjoying the properties stated in Definition \ref{d:MHDIdeal},
 and let $(R,u,b)$ 
be a weak solution (in the sense of Definition \ref{d:MHDIdeal})
to \eqref{i_eq:MHD-I} related to that initial datum and to some hydrodyamic pressure function $\Pi$. \\ 
Then, after defining $(\al_0, \bt_0) = (u_0+b_0,  u_0-b_0)$ and $(\alpha,\beta)=(u+b,u-b)$, the triplet
$(R,\alpha,\beta)$ 
is a weak solution (in the sense of Definition \ref{d:MHDab}) to system \eqref{eq:MHDab}, with the initial datum $(R_0,\alpha_0,\beta_0)$
and for suitable ``pressure'' functions $\pi_1$ and $\pi_2$ such that $\nabla\pi=\nabla\pi_1=\nabla\pi_2$, where $\pi$ is the MHD pressure defined in \eqref{eq:MHD-p}.
\item Conversely, consider $(R_0,\alpha_0,\beta_0)$ 
as in Definition \ref{d:MHDab} above, and let
$(R,\alpha,\beta)$
be a corresponding weak solution to system \eqref{eq:MHDab}, for suitable ``pressures'' $\pi_1$ and $\pi_2$. Assume that there exist $1\leq p, q<+\infty$ such that $\al_0 - \bt_0 \in L^p$,
$\al-\bt\in L^1_{\rm loc}\big([0,T[\,; L^p(\R^d;\R^d)\big)$ and $\al^j\,\bt^k\in L^1_{\rm loc}\big([0,T[\,;L^q(\R^d)\big)$ for all $1\leq j,k\leq d$. \\
Then, one has $\nabla\pi_1=\nabla\pi_2$. Moreover, after defining  $(u_0,b_0)=\big(\frac{\alpha_0+\beta_0}{2},\frac{\alpha_0-\beta_0}{2}\big)$ and
$(u,b)=\big(\frac{\alpha+\beta}{2},\frac{\alpha-\beta}{2}\big)$, the triplet $(R,u,b)$ 
is a weak solution (in the sense of Definition \ref{d:MHDIdeal}) to system \eqref{i_eq:MHD-I} related to the initial datum $(R_0,u_0,b_0)$,
for a suitable hydrodynamic pressure function $\Pi$ such that $\nabla\Pi=\nabla\big(\pi_1-|b|^2/2\big)$.
\end{enumerate}
\end{thm}

Before proving the previous theorem, some remarks are in order. 

\begin{rmk} \label{r:integr_1}
Here, we comment about the integrability hypotheses of the previous statement. We will discuss their sharpness in Subsection \ref{ss:integrab} below.
\begin{enumerate}[(a)]
 \item The integrability assumption $\al^j\,\bt^k\in L^1_{\rm loc}\big([0,T[\,;L^q(\R^d)\big)$ for all $1\leq j,k\leq d$ has been formulated for simplicity of presentation.
As it will appear clear from our proof, the result still holds true if one replaces it by the following
weaker condition: for all $1\leq j,k\leq d$, there exists Lebesgue exponents $q_{jk}\in[1,+\infty[\,$ such that $\al^j\,\bt^k\in L^1_{\rm loc}\big([0,T[\,;L^{q_{jk}}(\R^d)\big)$.

\item The space integrability on $\al-\bt$ comes for free if one works with initial data possessing some global integrability conditions. As long as $\al$ and $\bt$
remain Lipschitz continuous, those integrability properties are ``automatically'' propagated in time, thanks to the transport structure underlying 
system \eqref{eq:MHDab}.

\item Owing to the (formal) energy conservation law \eqref{eq:BasicEN} for the quasi-homogeneous ideal MHD system (which is also shared by the Els\"asser counterpart),
it is natural to work with $p=2$. This will be also the case in the present paper.
\end{enumerate}
\end{rmk}

\begin{rmk} \label{r:equiv}
The equivalence between systems \eqref{i_eq:MHD-I} and \eqref{eq:MHDab}, stated in Theorem \ref{th:symm}, covers two interesting cases: firstly weak solutions in the energy space
$(R, u, b) \in L^\infty_T(L^\infty \times L^2 \times L^2)$, in which case we may take $p =  2$ and $q = 1$,
and secondly solutions which lie in the Besov spaces $L^\infty_T(B^{s}_{p,r})$, provided that $1\leq p<+\infty$ and that $B^s_{p,r}\hookrightarrow W^{1, \infty}$
(\tsl{i.e.} under condition \eqref{i_eq:Lip2} above), this situation falling into the scope of the previous statement with $p=q$.

As we will see later on, to cover the case of the Lebesgue exponent $p = +\infty$, we will work with solutions that are both in the Besov space $L^\infty_T(\B)$ \text{and} in the energy space.
\end{rmk}

We can now prove the previous theorem.

\begin{proof}[Proof of Theorem \ref{th:symm}]
First of all, assume that $(R, u, b)$ is a weak solution of \eqref{i_eq:MHD-I}, according to Definition \ref{d:MHDIdeal}.
We check that $(R,\al, \bt)$ satisfies Definition \ref{d:MHDab}.

Writing the tensor product differences that appear in \eqref{i_eq:MHD-I} as functions of $\alpha$ and $\beta$, we get
\begin{align}
u \otimes u - b \otimes b & = (u + b) \otimes (u - b) + b \otimes u  - u \otimes b = \alpha \otimes \beta - \big( u \otimes b  - b \otimes u \big) \label{eq:bilin} \\
& = (u - b) \otimes (u + b) - b \otimes u + u \otimes b = \beta \otimes \alpha + \big( u \otimes b - b \otimes u \big) \,. \nonumber
\end{align}
Taking the sum and the difference of the weak forms \eqref{eq:WeakMomentum} and \eqref{eq:WeakMagnetic} of the momentum and the magnetic field equations in Definition \ref{d:MHDIdeal}, we immediately see that $(R,\al, \bt)$ satisfies Definition \ref{d:MHDab}, with the appropriate initial datum. The local integrability of $\al \otimes \bt$ follows from \eqref{eq:bilin}.

\medbreak

Conversely, suppose that $(R, \al, \bt)$ is a weak solution of system \eqref{eq:MHDab}, as in Definition \ref{d:MHDab}.
As we wish to manipulate the ``pressure'' terms in \eqref{eq:MHDab}, we use Lemma \ref{l:Integrability}. Thus, we may fix two tempered distributions
$\pi_1, \pi_2 \in \mathcal{S}'\big([0, T[\, \times \mathbb{R}^d\big)$ such that 
\begin{equation} \label{eq:a-b}
\left\{\begin{array}{l}
        \partial_t \alpha + (\beta \cdot \nabla) \alpha + \dfrac{1}{2}  R \mathfrak{C} (\al + \bt) + \nabla \pi_1 = 0 \\[1ex]
	 \partial_t \beta + (\alpha \cdot \nabla) \beta + \dfrac{1}{2}  R \mathfrak{C} (\al + \bt) + \nabla \pi_2  = 0\,.
       \end{array}
       \right.
\end{equation}
Taking the sum and the difference of these two relations, we recover the equations for the quantities $u=(\al+\bt)/2$ and $b=(\al-\bt)/2$:
\begin{equation} \label{eq:u-b}
\left\{\begin{array}{l}
\partial_t u + \D \big(u \otimes u - b \otimes b \big) + R \mathfrak{C} u + \dfrac{1}{2} \nabla \big( \pi_1 + \pi_2 \big) = 0 \\[1ex]
\partial_t b + \D \big( u \otimes b - b \otimes u \big) = \dfrac{1}{2} \nabla \big( \pi_2 - \pi_1 \big)\,.
       \end{array}
       \right.
\end{equation}
Notice that the integrability of the non-linear terms appearing in \eqref{eq:u-b} immediately follows from taking the sum and the difference of the two equations in \eqref{eq:bilin}.
Thus, in order to prove that $(R,u,b)$ is a weak solution of the original system \eqref{i_eq:MHD-I} in the sense of Definition \ref{d:MHDIdeal}, we 
must show that $\nabla \pi_1 = \nabla \pi_2$. From this property it will also follow that we can take $\nabla\pi=\nabla\pi_1$ in the first equation of \eqref{eq:u-b},
and then define the hydrodynamic pressure $\Pi$ according to \eqref{eq:MHD-p}.

To prove that $\nabla \pi_1 = \nabla \pi_2$, we start by remarking that $\pi_2 - \pi_1$ is in fact a harmonic function: taking the divergence of the magnetic field equation, we get the Laplace equation
\begin{equation*}
- \Delta \big( \pi_1 - \pi_2 \big) = 2\sum_{i,j}\partial_i \partial_j \big( u_i b_j - b_i u_j \big) = 0\,.
\end{equation*}
Since both $\pi_1$ and $\pi_2$ are tempered distributions, this means that the distributions $Q(t)\,:=\,\pi_1(t) - \pi_2(t)$ must be harmonic polynomials $Q(t) \in \mathbb{R}[x]$.
However, taking the difference of the equations in \eqref{eq:a-b} yields
\[
\nabla Q\,=\,-\div\big(\bt\otimes\al\,-\,\al\otimes\bt\big)\,-\,\d_t(\al-\bt)\, ,
\]
so, in view of the integrability assumptions on $\alpha-\beta$ and $\alpha\otimes\beta$, we retrieve integrability also for $Q=\pi_1-\pi_2$:
we discover that, for all $t\in[0,T[\,$, one has
\begin{equation*}
\nabla Q\, \in\, W^{-1,1}\big([0,t];L^p \big)\, +\, L^1\big([0,t];W^{-1,q} \big) \hookrightarrow W^{-1, 1} \big(  [0, t[ ; L^p + W^{-1, q} \big) \,.
\end{equation*}
Now, the heart of the argument consists in proving that the only polynomial lying in the space $L^p  + W^{-1, q}$ is zero: functions belonging to this space are
(in some loose sense) small at infinity, since \emph{both Lebesgue exponents are finite}, that is to say $p, q < +\infty$. 

We implement this argument. To avoid dealing with functions of low time regularity, we take a mollification kernel
$\big( K_\veps \big)_{\veps > 0}$, with each $K_\eps (t) = \frac{1}{\veps}\,K_1 \left( \frac{t}{\veps} \right)$ having compact support. Before convoluting $K_\epsilon$ with the momentum equation, we need to extend the solution $(R, u, b)$ to a function defined on all times $t \in \mathbb{R}$. As we wish to capture the behavior of $\nabla Q$ on the whole semi-closed interval $[0, T[$, that is including the initial time $t = 0$, we must be careful on how we make this extention.

By noting $\widetilde{f}(t)$ the extension to $t \in \mathbb{R}$ of a function $f(t)$ (of nonnegative times $t \geq 0$) defined by $f(t) = 0$ whenever $t < 0$,
we see from the weak form \eqref{eq:WeakMagnetic} of the magnetic field equation that $\big(\widetilde{u}, \widetilde{b}\big)$ solves a new equation
\begin{equation}\label{eq:tensorInitial}
\partial_t \widetilde{b} + \D \left( \widetilde{u} \otimes \widetilde{b} - \widetilde{b} \otimes \widetilde{u} \right) + \nabla \widetilde{Q} = \delta_{0}(t) \otimes b_0(x),
\end{equation}
where the initial datum condition is incorporated in the left-hand side of the equation. The tensor product $\delta_{ 0}(t) \otimes b_0(x)$ of these two distributions of the time and space variables is defined by
\begin{equation*}
\forall \phi \in \mc D \left( \mathbb{R} \times \mathbb{R}^d;\R^d \right), \qquad \big\langle \delta_{0}(t) \otimes b_0(x), \phi(t, x) \big\rangle_{\mc D' \times \mc D} :=
\int_{\mathbb{R}^d} b_0(x) \cdot \phi(0, x) \dx.
\end{equation*}
We may now convolute \eqref{eq:tensorInitial} by $K_\epsilon$. We find
\begin{align*}
\partial_t (K_\veps * \wtilde{b}) + K_\veps * \D \big( \wtilde{u} \otimes \wtilde{b} - \wtilde{b} \otimes \wtilde{u} \big) + \nabla (K_\veps * \wtilde{Q}) & = \big(K_\epsilon * \delta_{0}\big)(t) \otimes b_0(x) \\
&  = K_\epsilon(t) b_0(x),
\end{align*}
so that we have
\begin{equation*}
\nabla (K_\epsilon * \widetilde{Q}) \in C^\infty \big( \mathbb{R} ;  L^p + W^{-1, q} \big).
\end{equation*}
Assume for a while the condition $q\leq p$: then, thanks to Proposition \ref{p:embed}, we get the chain of embeddings \begin{equation*}
W^{-1, q} \hookrightarrow B^{-1}_{q, \infty} \hookrightarrow B^{-1 - d \left( 1/q - 1/p \right)}_{p, \infty} := B^{-k}_{p, \infty}\,.
\end{equation*}
Therefore, we have the inclusion $L^p + W^{-1, q} \hookrightarrow B^{-k}_{p, \infty}$. In the case $p<q$, instead, we get a similar inclusion, in the space $B^{-k'}_{q,\infty}$, for a new $k'>0$.
Now, assume that $q\leq p$ (the case $p < q$ is perfectly symmetric) and work with $B^{-k}_{p,\infty}$ during the rest of the proof.
Now, recall from Subsection \ref{ss:LP} the low frequency Littlewood-Paley block $\Delta_{-1} = \chi(D)$. In view of the embedding above, we have, for all times $t \in \,]-T, T[\,$,
\begin{equation*}
\Delta_{-1} \nabla \big( K_\veps * \wtilde{Q} \big)(t) \in L^p\,.
\end{equation*}
On the other hand, since the Fourier transform of a polynomial is merely a linear combination of the Dirac mass $\delta_0$ and its derivatives, the space of polynomials stays untouched
by $\Delta_{-1}$, which is a Fourier multiplier by a smooth function $\chi$ equal to unit value on a neighborhood of
$\xi = 0$. 
Therefore, we conclude that
\begin{equation*}
\forall\, t \in\, ]-T, T[\,, \qquad \nabla \big(K_\veps * \wtilde{Q}\big) (t) = \Delta_{-1} \nabla \big(K_\veps * \wtilde{Q}\big) (t) \,\in\, \mathbb{R}[x] \cap L^p = \{ 0 \}\,,
\end{equation*}

We have shown that, for all $\epsilon > 0$, we have
\begin{equation*}
\partial_t (K_\epsilon * \widetilde{b}) + K_\epsilon * \D \left( \widetilde{u} \otimes \widetilde{b} - \widetilde{b} \otimes \widetilde{u} \right) = K_\epsilon \, b_0.
\end{equation*}
We may take the limit $\epsilon \rightarrow 0^+$ in the equation, in the sense of distributions, to obtain, in the end
\begin{equation*}
\partial_t \widetilde{b} + \D \left( \widetilde{u} \otimes \widetilde{b} - \widetilde{b} \otimes \widetilde{u} \right) = \delta_{0}(t) \otimes b_0(x),
\end{equation*}
in the space $\mc D' (\mathbb{R} \times \mathbb{R}^d)$, which is equivalent to saying that $(R, u, b)$ is indeed a weak solution of \eqref{i_eq:MHD-I}, in the sense of Definition \ref{d:MHDIdeal}.
This completes the proof of the theorem.
\end{proof}

Next, we can prove that, under similar global integrability assumption for the Els\"asser variables $\big(\al,\bt\big)$, equations \eqref{eq:MHDab} can be projected onto the
space of divergence-free vector fields, finding a new
equivalent formulation of the system. For this, we are going to make use of the Leray projector $\P$, whose definition and basic properties have been recalled at the end of Subsection \ref{ss:LP}.

More precisely, we are going to show the following result.
\begin{thm} \label{t:equiv-P}
Consider a weak solution $(R,\al,\bt)$ to system \eqref{eq:MHDab} in the sense of Definition \ref{d:MHDab}, related to the initial datum $(R_0,\al_0,\bt_0)$, with
both $\alpha_0$ and $\beta_0$ being divergence-free. 
Assume that there exists a triplet of indices $(p_1,p_2,p_3)\in[1,+\infty[\,^3$ such that the following conditions are verified: $\al_0, \bt_0 \in L^{p_1}$, 
$\al,\bt\,\in\, L^1_{\rm loc}\big([0,T[\,;L^{p_1}(\R^d;\R^d)\big)$, $R\,(\al+\bt)\in L^1_{\rm loc}\big([0,T[\,;L^{p_2}(\R^d;\R^d)\big)$ and
$\al^j\,\bt^k\in L^1_{\rm loc}\big([0,T[\,;L^{p_3}(\R^d)\big)$ for all $1\leq j,k\leq d$.


Then, the Els\"asser variables $(R, \al, \bt)$ solve the following system, in the weak sense:
\begin{equation}\label{eq:EquivEq}
\begin{cases}
\partial_t R + \dfrac{1}{2}\D \big( R (\al + \bt) \big) = 0 \\[1ex]
\partial_t \al + \mathbb{P} \D (\bt \otimes \al) + \dfrac{1}{2} \mathbb{P} \big( R \mathfrak{C} (\al + \bt) \big) = 0\\[1ex]
\partial_t \bt + \mathbb{P} \D (\al \otimes \bt) + \dfrac{1}{2} \mathbb{P} \big( R \mathfrak{C} (\al + \bt) \big) = 0\,.
\end{cases}
\end{equation}
\end{thm}

Notice that the functions $\P\big(R\,\mf C(\al+\bt)\big)$ and $\mathbb{P} \D (\bt \otimes \al)$ are well-defined, since $R\,\mf C(\al+\bt)\in L^1_{\rm loc}([0,T[\,;L^{p_2})$ and
$\bt \otimes \al \in L^1_{\rm loc}([0,T[\,;L^{p_3})$, with $p_2,p_3 < +\infty$. So, each term in \eqref{eq:EquivEq} is well-defined. Also note that, as we will see later on,
Theorem \ref{t:equiv-P} is no longer true if $\al$ and $\bt$ are merely bounded.

Let us now present the proof of the previous theorem. The argument is similar to the one used for proving Theorem \ref{th:symm} above.

\begin{proof}[Proof of Theorem \ref{t:equiv-P}]
From the proof of Theorem \ref{th:symm}, we already know that, under our assumptions, the couple $(\al,\bt)$ solves equations
\eqref{eq:a-b} with $\nabla\pi_1=\nabla\pi_2\,=\,\nabla\pi$, for a suitable tempered distribution $\pi$ which is uniquely determined up to an additive constant.

At this point, we proceed more or less as before: taking the divergence of the equations and solving the elliptic equation thus produced, we obtain
\begin{equation}\label{eq:QPolynomial}
\nabla \pi = \nabla (- \Delta)^{-1} \left(\sum_{i,j} \partial_i \partial_j \big( \al_i \bt_j\big)  +  \frac{1}{2} \D \big(R \mathfrak{C} (\al + \bt) \big) \right) + \nabla Q\,,
\end{equation}
where the functions $Q(t) \in \mathbb{R}[x]$ are harmonic polynomials. Incidentally, note that the inverted Laplacian, defined as a Fourier multiplier, poses no problem in \eqref{eq:QPolynomial} above,
as the operator $\nabla (-\Delta)^{-1}\D$ has a bounded symbol and as both functions $\al \otimes \bt$ and $R \mathfrak{C}(\al + \bt)$ belong to some $L^r$ space, with $r < +\infty$.
Inserting \eqref{eq:QPolynomial} in the equations, we see that 
\begin{equation*}
- \nabla Q = \partial_t \al + \big( I + \nabla (- \Delta)^{-1} \D  \big)   \Big( \D (\bt \otimes \al) + \frac{1}{2} R \mathfrak{C}(\al + \bt)  \Big)
\end{equation*}
so the polynomial function $\nabla Q$ must belong to the space $W^{-1,1}_t(L^{p_1}) + L^1_t(W^{-1,p_3}+L^{p_2})$, for all $t<T$. Assuming, without loss of generality, that $p_2,p_3\leq p_1$,
using embeddings as above we infer that 
$\nabla Q\in  W^{-1, 1}_{\rm loc}([0,T[\,;B^{-k}_{p_1, \infty})$, for some finite $k>0$.
Arguing exactly as before, by incorporating the initial value condition $\delta_{0}(t) \otimes \al_0(x)$ in the equation, noting $\nabla \wtilde{Q} (t)$ the extention by $0$ of $\nabla Q (t)$ to $t \in \mathbb{R}$ and convoluting by $K_\veps$, we see that
\begin{equation*}
\forall\, t \in \,]-T, T[\,, \qquad  \Delta_{-1} \nabla \big( K_\veps * \wtilde{Q}  \big)  (t)\, \in\, \mathbb{R}[x] \cap L^{p_1}\,,
\end{equation*}
which implies that $\nabla (K_\veps * \widetilde{Q}) = 0$ in the space $\mc D'(\mathbb{R} \times \mathbb{R}^d)$. Taking the limit $\veps \rightarrow 0^+$, we infer that 
\begin{equation*}
\partial_t \widetilde{\al} + \big( I + \nabla (- \Delta)^{-1} \D  \big)   \Big( \D \left( \widetilde{\bt} \otimes \widetilde{\al} \right) + 
\frac{1}{2} \widetilde{R} \mathfrak{C} \left( \widetilde{\al} + \widetilde{\bt} \right)  \Big) = \delta_{ 0}(t) \otimes \al_0(x)
\end{equation*}
in the space $\mc D' (\mathbb{R}\times \mathbb{R}^d)$, and likewise for $\bt$. These relations give us exactly \eqref{eq:EquivEq}.
\end{proof}

\begin{rmk} \label{r:commutator}
By introducing commutators of operators, we see that \eqref{eq:EquivEq} is in fact a system of transport equations:
\begin{equation*} 
\begin{cases}
\partial_t R + \dfrac{1}{2}\,\D \big( R(\al + \bt) \big) = 0 \\[1ex]
\partial_t \alpha + (\beta \cdot \nabla) \alpha + \dfrac{1}{2} \mathbb{P} \big( R \mathfrak{C} (\al + \bt) \big) = \big[ \beta \cdot \nabla, \, \mathbb{P} \big]\alpha \\[1ex]
\partial_t \beta + (\alpha \cdot \nabla) \beta + \dfrac{1}{2} \mathbb{P} \big( R \mathfrak{C} (\al + \bt) \big) = \big[ \alpha \cdot \nabla, \, \mathbb{P} \big]\beta\,.
\end{cases}
\end{equation*}
In the above system, the transport operators are to be understood in the weak sense: thanks to the divergence-free conditions, we can write
\begin{equation*}
(\bt \cdot \nabla) \al = \D (\bt \otimes \al) \qquad \text{and} \qquad \big[ \bt \cdot \nabla, \mathbb{P} \big] \al = (I - \mathbb{P}) \D (\bt \otimes \al)\,.
\end{equation*}
\end{rmk}

Formulation \eqref{eq:EquivEq} of the quasi-homogeneous ideal MHD system, or better the last formulation of Remark \ref{r:commutator}, has been broadly exploited
in \cite{Cobb-F_Rig}, for studying well-posedness in critical Besov spaces $B^s_{p,r}\hookrightarrow W^{1,\infty}$ when $1<p<+\infty$.
Here, for dealing with the case $p=+\infty$, we will not apply the Leray projector, and work instead with the vorticity formulation
of the equations.

\subsection{Counterexamples to the equivalence, and uniqueness issues} \label{ss:integrab}

Now, we make a few concluding remarks concerning the equivalence of formulations \eqref{i_eq:MHD-I}, \eqref{eq:MHDab} and \eqref{eq:EquivEq} of the quasi-homogeneous ideal MHD system.
This reveals to be a rather delicate issue, as one may guess looking at the statements of Theorems \ref{th:symm} and \ref{t:equiv-P}.

\subsubsection{About the integrability assumptions}

We start by observing the following fact: proving that \eqref{eq:EquivEq} is equivalent to its counterpart in the usual physical variables $(R, u, b)$, \tsl{i.e.}
\begin{equation}\label{eq:EquivUB}
\begin{cases}
\partial_t R + \D (R u) = 0 \\[1ex]
\partial_t u + \mathbb{P} \Big( \D \big( u \otimes u - b \otimes b \big) + R \mathfrak{C}u \Big) = 0 \\[1ex]
\partial_t b + \D \big( u \otimes b - b \otimes u \big) = 0\,,
\end{cases}
\end{equation}
is easy, as it simply relies on algebraic manipulations (notice that the magnetic field equation remains unchanged by the action of the Leray projector $\mathbb{P}$).

However, it is not obvious that weak solutions of the quasi-homogeneous ideal MHD equations \eqref{i_eq:MHD-I} necessarily solve \eqref{eq:EquivUB}, even if they are regular, just as it is unclear that all weak solutions of \eqref{eq:MHDab} are also solutions of \eqref{eq:EquivEq}.

The chief reason for this is that we cannot apply the Leray projector to \eqref{i_eq:MHD-I} or \eqref{eq:MHDab} without first ascertaining whether all terms are well defined. For instance, if we apply
the Leray projector to the momentum equation in \eqref{i_eq:MHD-I}, one must first check that the images $\mathbb{P} \, ( \nabla \pi )$ and $\mathbb{P} \, ( \partial_t u )$ have any
sense. As a matter of fact, it is shown in \cite{MY} (see equations (2.50) to (2.53) therein) that if $u$ and $b$ are regular and $L^\infty$,
then $\mathbb{P} \D (u \otimes u - b \otimes b)$ can always be defined.
On the contrary, this does not apply to the two functions $\partial_t u$ and $\nabla \pi$. 
In the $L^p$ framework (with $1 < p < +\infty$), it is, in fact, an \textsl{a posteriori} consequence of Theorems \ref{th:symm} and \ref{t:equiv-P} that both functions $\nabla \pi$ and $\partial_t u$ have enough integrability (\tsl{i.e.} no polynomial part) for the Leray projector to be applied to the momentum equation.

We have proven the equivalence of the three systems \eqref{i_eq:MHD-I}, \eqref{eq:MHDab} and \eqref{eq:EquivEq} for weak solutions that are small enough (in some sense)
at infinity. 
More precisely, the equivalence is proven under some \emph{global integrability} assumptions, which,
roughly speaking, 
act as a boundary condition for systems \eqref{i_eq:MHD-I} and \eqref{eq:MHDab}: the fluid is assumed to be at rest at infinity.
Such a boundary condition is instead implicit in both systems \eqref{eq:EquivEq} and \eqref{eq:EquivUB}, because the inverted Laplace operator $(- \Delta)^{-1}$ has its range
in the space of functions which have no harmonic part.

\subsubsection{Counterexamples to the equivalence}
Here, we discuss the the sharpness of the assumptions of Theorems \ref{th:symm} and \ref{t:equiv-P}. We will construct two counterexamples,
which prove the \emph{failure} of the equivalence between \eqref{i_eq:MHD-I} and \eqref{eq:EquivUB} the first one, between \eqref{i_eq:MHD-I} and \eqref{eq:MHDab} the second one,
in absence of any integrability condition on the solutions.

\subsubsection*{Failure of the equivalence between systems \eqref{i_eq:MHD-I} and \eqref{eq:EquivUB}}
For our first counterexample, we work in space dimension $d = 2$. Define the uniform flow by
\begin{equation}\label{eq:UnifFlow}
u(t, x) = \big( f(t), 0 \big)\,,\qquad\qquad \text{ with }\quad R = b = 0\,,
\end{equation}
where $f \in C^\infty(\mathbb{R})$ is a non-constant function of time. Observe that $\div(u\otimes u)=0$.
Then, $u$ thus defined solves the Euler equations \eqref{i_eq:MHD-I} with the associated pressure given by 
\begin{equation*}
\pi (t, x) = - f'(t) x_1 + C\,,
\end{equation*}
for some constant $C\in\R$.
By contrast, $u$ does \emph{not} solve system \eqref{eq:EquivUB}, because
\begin{equation}\label{eq:lastLabel}
\partial_t u + \mathbb{P} \, \D (u \otimes u)\, =\, \big(f'(t),0\big)\, \neq\, 0\,.
\end{equation}
Thus, regular bounded functions do not solve equivalently \eqref{i_eq:MHD-I} and \eqref{eq:EquivUB}, so these systems are not equivalent in a purely $L^\infty$ framework.
We refer to Paragraph \ref{sss:P} below for some additional details on this issue.

Note that this solution we give is not simply an artefact of the galilean invariance of the equations, as the uniform flow \eqref{eq:UnifFlow} is not always at rest in a given inertial reference frame. 
On the other hand, system \eqref{eq:EquivUB} naturally preserves certain boundary properties: if the fluid is initially at rest at infinity in some inertial reference frame,
then \eqref{eq:EquivUB} implies that it will always be so {in the same reference frame}. 


\medbreak
As a final comment, we remark that $f(t)$ need not be smooth for the uniform flow \eqref{eq:UnifFlow} to solve the ideal MHD problem \eqref{i_eq:MHD-I}.
By taking $f(t) = \mathds{1}_{[t_0,+\infty[}(t)$, for some $t_0>0$, we see that \eqref{i_eq:MHD-I} possesses bounded weak solutions that are \emph{not continuous} with respect to the time variable,
even in the $\mc D'$ topology.

\subsubsection*{Failure of the equivalence between systems \eqref{i_eq:MHD-I} and \eqref{eq:MHDab}}

In the same spirit as \eqref{eq:UnifFlow}, we define a solution of the Els\"asser system \eqref{eq:MHDab} by setting, in two dimensions of space $d = 2$,
\begin{equation*} 
\al(t, x) = \big(f(t), 0\big) = - \bt(t, x) \,, \qquad \qquad \text{ with } \quad R = 0\,.
\end{equation*}
This solution is associated to the following ``pressure'' functions:
\begin{equation*}
\pi_1(t, x) = - f'(t) x_1 + C \qquad\qquad \mbox{ and }\qquad\qquad  \pi_2(t,x) = - \pi_1 (t, x)\,.
\end{equation*}
However, when trying to recover the usual ``physical'' variables, we see that $u = (\al + \bt)/2 = 0$, whereas $b = (\al - \bt)/2 = \big(f(t), 0\big)$
does not solve the magnetic field equation in \eqref{i_eq:MHD-I}, since
\begin{equation*}
\partial_t b + (u \cdot \nabla)b - (b \cdot \nabla)u = \big(f'(t), 0\big) \neq 0\,.
\end{equation*}

This simple example shows that, for the same reasons invoked above, systems \eqref{i_eq:MHD-I} and \eqref{eq:MHDab} are not, \textsl{per se}, equivalent.
Some kind of boundary condition is needed alongside. Notice that, according to the assumptions of Theorem \ref{th:symm}, the velocity field $u$ constructed in this example is smooth and integrable,
whereas no integrability property is available for $b$, which is only bounded.

\subsubsection{Lack of uniqueness in a purely $L^\infty$ framework}

The two counterexamples constructed above highlight in fact a \emph{lack of uniqueness} of solutions to systems \eqref{i_eq:MHD-I} and \eqref{eq:MHDab} respectively,
in a purely $L^\infty$ framework. 


As a matter of fact, we know from \cite{MY} that (for $R\equiv0$) system \eqref{eq:EquivUB} is well-posed in $B^1_{\infty,1}$. Solving with the initial datum
$\wtilde u_0=u(0)=\big(f(0),0\big)$, $\wtilde b_0=0$ and $\wtilde R_0=0$, where $u$ is the uniform flow defined in \eqref{eq:UnifFlow}, we find a unique (constant) solution 
$\big(\wtilde u, \wtilde b, 0\big)$ to system \eqref{eq:EquivUB} in $L^\infty_T(B^{1}_{\infty,1})$, which then also solves \eqref{i_eq:MHD-I}, for a suitable pressure $\nabla\Pi$.
On the other hand, the triplet $(u,b,R)$ we have exhibited in our first example is another solution to the original problem \eqref{i_eq:MHD-I}, belonging to the same functional class
and related to the same initial datum\footnote{In other words, the uniqueness result of \cite{MY} for system \eqref{eq:EquivUB} in the $B^1_{\infty,1}$ functional framework does \emph{not} imply uniqueness 
for the original MHD system \eqref{i_eq:MHD-I} in that class, without any additional condition.}.

Observe that the issue is not regularity. Propagating higher order norms\footnote{Technically, this is not done in \cite{MY}, but it follows from the same analysis.} and repeating the argument above,
we can conclude that the quasi-homogeneous ideal MHD system \eqref{i_eq:MHD-I} is ill-posed in $C^\infty_b\,:=\,\bigcap_{m\in\N}C^m_b$,
with $C^m_b\,:=\,C^m\cap W^{m,\infty}$. Focusing on the case when $R\equiv0$,
we infer as well that the classical ideal MHD equations are ill-posed in the $C^\infty_b$ setting.

\medbreak
Even more simply, we may well choose our solution \eqref{eq:UnifFlow} so that $f(t)$ is compactly supported away from $t=0$, producing in this way another solution (apart from the zero one) starting with
zero initial datum.
Of course, the same argument applies also to our second counterexample. Hence we can conclude that the Els\"asser system \eqref{eq:MHDab} does not behave better
than the original MHD problem \eqref{i_eq:MHD-I} regarding well-posedness, inasmuch as it exhibits the same lack of uniqueness in a purely $L^\infty$ framework.

\subsubsection{Final remarks on the Leray projector} \label{sss:P}

We conclude this part with a few additional remarks on the Leray projector $\P$.

At first sight, it might seem rather suprising that applying $\P$ to the momentum equation in the usual MHD system \eqref{i_eq:MHD-I} does not yield equations \eqref{eq:EquivUB},
especially considering that, for any divergence-free function $w=w(x)$, we may \emph{formally} write
\begin{equation*}
\mathbb{P} w = w + \big[\nabla (- \Delta)^{-1}\big] \D (w) = w,
\end{equation*}
so that the image $\mathbb{P} w$ is (at least formally) well defined.

The fact is that several definitions of the Leray projector coexist (as a Fourier multiplier, or as a composition of operators as in the previous relation). Their equivalence 
must not be taken for granted, as it depends on the considered functional framework.
In particular, the Leray projector is an unbounded operator on spaces of bounded functions (such as $L^\infty$, $C^\infty_b$ or $B^1_{\infty, 1}$, for instance), 
on which it is \emph{not} densely defined. Therefore, its domain $\mf D (\mathbb{P})$ may depend on the precise definition of $\mathbb{P}$ we rely on.

To illustrate this point, let us make an example. Consider \tsl{e.g.} a constant function $w(x)\equiv\oline{w}\in\R^d$. Then $\P w$ has no meaning if we define $\P$ as a
Fourier multiplier (recall relation \eqref{eq:LerayFM} above),
because $\what w(\xi)=\oline{w}\,\de_0(\xi)$, while the symbol of $\P$ is not well-defined at $\xi=0$. On the other hand, if we define the projector as a composition of operators
by formula \eqref{eq:P-op}, with $(-\Delta)^{-1}$ being understood as (say) a convolution operator, and we perform computations precisely \emph{in that order} (in fact,
$\div$ and $(-\Delta)^{-1}$ do not commute anymore, in this case, because their domains and ranges are not compatible), then $\div w=0$, so $\P w$ makes sense, and we have $\P w=w$.

\medskip
A similar ambiguity of definition occurs also when trying to compute $\P(\nabla \theta)$, with $\theta$ being a harmonic function.
For instance, take the uniform flow \eqref{eq:UnifFlow}, which solves the Euler equations. Then, if we apply definition \eqref{eq:P-op} of $\P$, performing operations
in that precise order, we may formally compute the images $\mathbb{P}(\partial_t u)$ and $\mathbb{P}(\nabla \pi)$. In particular, since $\nabla\pi=\big(-f'(t),0\big)$, we have $\div \nabla\pi=0$,
which implies
\begin{equation}\label{eq:lastLabel2}
\mathbb{P}(\nabla \pi)\,=\,\nabla\pi\, =\, -\big(f'(t), 0\big)\neq 0\,.
\end{equation}
Note that this is exactly the missing term for obtaining the equality in \eqref{eq:lastLabel}. More precisely, by taking \eqref{eq:lastLabel2} into account, we see that
$$
\mathbb{P} \Big( \partial_t u + \D(u \otimes u) \Big)\,+\,\mathbb{P} ( \nabla \pi)\,=\,0\,,
$$
thus the uniform flow \eqref{eq:UnifFlow}, solution of the Euler equation \eqref{i_eq:MHD-I}, solves also the projected counterpart.
In particular, equations \eqref{eq:EquivUB} are not always the Leray projection of the ideal MHD problem \eqref{i_eq:MHD-I}: they are missing the $\P (\nabla \pi)$ term, which may be, as seen above,
non-zero.


\section{Well-posedness of the quasi-homogeneous ideal MHD system} \label{WP}

This section is devoted to the proof of Theorems \ref{th:BesovWP} and \ref{th:cont-crit}. In the first subsection, we prove uniqueness of solutions in the considered
functional framework.
In Subsection \ref{ss:existence} we prove existence of solutions, and exhibit a first lower bound (valid in any space dimension) for the lifespan of the solutions.
Finally, in Subsection \ref{ss:proof-cc} we show the proof of the continuation criterion.

\subsection{Uniqueness by an energy method} \label{ss:unique}

In this section, we focus on the uniqueness of solutions. As we have explained in Subsection \ref{ss:integrab}, a sufficient condition
for the quasi-homogeneous ideal MHD system \eqref{i_eq:MHD-I} to be well-posed consists in requiring the solutions to have some integrability property at infinity, whereas there is no hope of uniqueness for solutions that are solely bounded.
In our framework, this is guaranteed by the finite-energy condition on the initial data.

Uniqueness in our functional framework is a straightforward consequence of the following stability result, whose proof is based on an energy method.

\begin{thm} \label{th:w-s}
Let $(R_1, u_1, b_1)$ and $(R_2,u_2,b_2)$ be two solutions\footnote{To fix ideas, say weak solutions in the sense of Definition \ref{d:MHDIdeal}.}
to the quasi-homogeneous ideal MHD system \eqref{i_eq:MHD-I}. Assume that, for some $T>0$, one has the following properties:
\begin{enumerate}[(i)]
\item the three quantities $\de R\,:=\,R_1-R_2$, $\de u\,:=\,u_1-u_2$ and $\de b\,:=\,b_1-b_2$ all belong to the space $C^1\big([0,T];L^2(\R^d)\big)$;
\item $u_1 \in L^1\big([0,T];W^{1, \infty}(\R^d)\big)$ and $\nabla R_1, \nabla b_1 \in L^1\big([0,T];L^\infty(\R^d)\big)$;
\item $R_2\in L^1\big([0,T];L^{\infty}(\R^d)\big)$.
\end{enumerate}

Then, for all $t\in[0,T]$, we have the stability inequality 
\[
\big\| (\delta R, \delta u, \delta b)(t) \big\|_{L^2}\,\leq\,C\,\big\| (\delta R, \delta u, \delta b)(0) \big\|_{L^2}\,e^{CA(t)}\,, 
\]
for a universal constant $C>0$, where we have defined
$$
A(t)\,:=\,\int_0^t \Big\{ \| u_1(\t) \|_{W^{1, \infty}} + \| \nabla b_1(\t) \|_{L^\infty} + \| {R}_2(\t) \|_{L^\infty} + \| \nabla R_1(\t) \|_{L^\infty} \Big\} {\rm d}\t\,.
$$
\end{thm}

\begin{rmk}
In the case where the matrix $\mathfrak{C}$ is skew-symmetric, we may dispense with the norm $\| {R}_2 \|_{L^\infty}$ in the definition of $A$, 
as well as in assumption \textit{(iii)} of the statement.

If $\mathfrak{C} = 0$, then we may also replace $\| u_1 \|_{W^{1, \infty}}$ by $\| \nabla u_1 \|_{L^\infty}$.
\end{rmk}

\begin{rmk} \label{r:w-s}
Employing similar arguments as the ones used in \cite{Cobb-F_Rig} (see the proof to Theorem 4.3 therein), it would be enough to assume $C^0_T(L^2)$ regularity for
$\de R$, $\de \al$ and $\de\bt$.
In that case, the previous theorem would become a full-fledged weak-strong uniqueness result.

For the sake of simplicity, we do not pursue that issue here, and we assume that $\de R$, $\de \al$ and $\de\bt$ belong to $C^1_T(L^2)$.
\end{rmk}

\begin{proof}
The claimed bound is simply based on energy estimates for the difference of the two solutions. System \eqref{i_eq:MHD-I} is symmetric, so one could implement that strategy directly
on the $(R,u,b)$-formulation of the equations. However, in order to avoid unpleasant derivatives on those differences (which we do not know to be smooth enough),
it is better to work in Els\"asser variables.

Therefore, with obvious notations, let us introduce the Els\"asser variables $(R_1, \al_1, \bt_1)$ and $(R_2,{\al}_2, {\bt}_2)$,
which solve system \eqref{eq:MHDab}. This is possible thanks to item (i) in Theorem \ref{th:symm}; notice that this step is based only on algebraic manipulations of the equations, and requires
no special integrability conditions. Set
$$
\de\al\,:=\,\al_1\,-\,\al_2\qquad\qquad\mbox{ and }\qquad\qquad \de\bt\,:=\,\bt_1\,-\,\bt_2\,.
$$
We take the difference of the two systems solved by $(R_1, \al_1, \bt_1)$ and $(R_2, {\al}_2, \bt_2)$ to obtain
\begin{equation} \label{eq:syst-diff}
\begin{cases}
\partial_t (\delta R) + \dfrac{1}{2} (\al_2 + \bt_2) \cdot \nabla \delta R = - \dfrac{1}{2} \big( \delta \al + \delta \bt \big) \cdot \nabla R_1 \\[1ex]
\partial_t (\delta \al) + ({\bt}_2 \cdot \nabla ) \delta \al + (\delta \bt \cdot \nabla ) \al_1 + \dfrac{1}{2} R_2 \mathfrak{C}\big( \delta \al + \delta \bt \big) +
\dfrac{1}{2} \delta R \mathfrak{C} \big( {\al}_1 + {\bt}_1 \big) + \nabla \de\pi_1 = 0 \\[1ex]
\partial_t (\delta \bt) + ({\al}_2 \cdot \nabla ) \delta \bt + (\delta \al \cdot \nabla ) \bt_1 + \dfrac{1}{2}  R_2 \mathfrak{C}\big( \delta \al + \delta \bt \big) +
\dfrac{1}{2} \delta R \mathfrak{C} \big( {\al}_1 + {\bt}_1 \big) + \nabla \de\pi_2 = 0\\[1ex]
\D(\delta \al)\, =\, \D (\delta \bt)\, =\, 0\,,
\end{cases}
\end{equation}
where we have denoted $\de \pi_1$ and $\de\pi_2$ the difference of the two pressure terms appearing in system \eqref{eq:MHDab} and related to the triplets $(R_1,\al_1,\bt_1)$ and 
$(R_2,\al_2,\bt_2)$.

We start by testing the first equation against $\de R$: we gather
\begin{equation*}
\frac{1}{2} \frac{\rm d}{\dt} \left\|\delta R\right\|_{L^2}^2 = - \frac{1}{2} \int (\delta \al + \delta \bt) \cdot \nabla R_1\,\delta R  \dx\, \leq\,
\|\nabla R_1 \|_{L^\infty} \big\| (\delta R, \delta \al, \delta \bt) \big\|_{L^2}^2\,.
\end{equation*}

Next, testing the second equation on $\delta \al$, owing to the divergence-free conditions on $\de\al$ and $\de\bt$, we obtain
\begin{equation*}
\frac{1}{2} \frac{\rm d}{\dt}\left\| \delta \al \right\|_{L^2}^2 = - \int (\delta \bt \cdot \nabla ) \al_1 \cdot \delta \al \dx -
\frac{1}{2} \int {R}_2 \mathfrak{C} \big( \delta \al + \delta \bt \big)\cdot\de\al \dx - \int \delta R \mathfrak{C} (\al_1 + \bt_1) \cdot \delta \al \dx\,.
\end{equation*}
Bounding the three integrals on the right-hand side of the previous equality is fairly easy: after using the Cauchy-Schwarz and Young inequalities, we get
\begin{equation*}
\frac{1}{2} \frac{\rm d}{\dt}\left\| \delta \al \right\|_{L^2}^2\,\leq\,C\,
\Big( \| \nabla \al_1 \|_{L^\infty} + \| {R}_2 \|_{L^\infty} + \| \al_1 + \bt_1 \|_{L^\infty} \Big)\, \big\| (\delta R, \delta \al, \delta \bt) \big\|_{L^2}^2\,,
\end{equation*}
for a universal constant $C>0$ depending only on the coefficients of $\mf C$.

Performing, \textsl{mutatis mutandi}, the same computations with the second equation, we find an analogous inequality:
\begin{equation*}
\frac{1}{2} \frac{\rm d}{\dt}\left\| \delta \bt \right\|_{L^2}^2\,\leq\,C\,
\Big( \| \nabla \bt_1 \|_{L^\infty} + \| {R}_2 \|_{L^\infty} + \| \al_1 + \bt_1 \|_{L^\infty} \Big)\, \big\| (\delta R, \delta \al, \delta \bt) \big\|_{L^2}^2\,.
\end{equation*}

Putting all the three inequalities together, we finally deduce
\begin{equation*}
\frac{1}{2} \frac{\rm d}{\dt} \big\| ( \delta R, \delta \al, \delta \bt ) \big\|_{L^2}^2\,\leq\,C\,\Big\{ \| u_1 \|_{W^{1, \infty}} + \| \nabla b_1 \|_{L^\infty} +
\| {R}_2 \|_{L^\infty} + \| \nabla R_1 \|_{L^\infty} \Big\}\, \big\| ( \delta R, \delta \al, \delta \bt ) \big\|_{L^2}^2\,.
\end{equation*}
An application of Gr\"onwall's lemma ends the proof.
\end{proof}

From the previous result, it is possible to deduce uniqueness of solutions in the considered functional framework.
\begin{proof}[Proof of uniqueness in Theorem \ref{th:BesovWP}]
Let us consider an initial datum $(R_0,u_0,b_0)$ satisfying the assumptions of Theorem \ref{th:BesovWP}. Let $(R_1,u_1,b_1)$ and $(R_2,u_2,b_2)$ be two solutions to system \eqref{i_eq:MHD-I}
related to that initial datum, and fulfilling the conditions stated in the same theorem.

It is not hard to see that all the assumptions made in Theorem \ref{th:w-s} are matched by those solutions. The only point which deserves some explanation is Condition (i):
let us give some details. We focus only on the regularity of the quantity $\de R$, the proof being similar for $\delta u$ and $\delta b$.

First of all, we notice that, under our hypotheses, the equivalence of systems \eqref{i_eq:MHD-I} and \eqref{eq:MHDab} is fully justified by Theorem \ref{th:symm}.
So, we can pass to Els\"asser variables and recover again that the quantities $(\de R,\de\al,\de\bt)$, defined as in the previous proof, satisfy system \eqref{eq:syst-diff}.

Let us focus on the first equation of \eqref{eq:syst-diff}: $\de R$ takes the value $\de R_{|t=0}=0$ at initial time and is transported by a divergence-free vector field, under the action
of the ``external force'' $f\,:=\,-(\de\al+\de\bt)\cdot\nabla R_1/2$. By the regularity properties stated in Theorem \ref{th:BesovWP} and embeddings, we know that
$$
u\,,\,b\;\in\,C^0\big([0,T];L^2\big)\qquad\qquad\mbox{ and }\qquad\qquad R_1\,\in\,C^0\big([0,T];W^{1,\infty}\big)\,.
$$
From this, we infer $f\,\in\,C^0\big([0,T];L^2\big)$. Therefore, by transport we get $\de R\in C^0\big([0,T];L^2\big)$. By the same token, we also see that $\d_t\de R$ belongs to
the same space, so finally $\de R\in C^1\big([0,T];L^2\big)$, as claimed.

The needed regularities of $\de u$ and $\de b$ follow from a similar argument, using again the equations in \eqref{eq:syst-diff} and the Leray projector. In the end, we can apply Theorem \ref{th:w-s}
to deduce that $\big\|(\de R,\de u,\de b)\big\|_{L^\infty_T(L^2)}=0$. By both time and space continuity, we infer that $\big(R_1,u_1,b_1\big)(t,x)=\big(R_2,u_2,b_2\big)(t,x)$
for \emph{all} $(t,x)\in[0,T]\times\R^d$. This means exactly the sought uniqueness.
\end{proof}

\subsection{Existence of solutions}\label{ss:existence}

In this subsection, we show existence of solutions. For this, we implement a nowadays classical scheme. First of all, in Paragraph \ref{sss:bounds} we will show \tsl{a priori} estimates
for smooth solutions in the relevant norms. From those estimates, we will also deduce a first lower bound (valid in any space dimension) on the lifespan of the solutions.
After that, in Paragraph \ref{sss:approx} we will give the explicit construction of smooth solutions to approximate problems,
and show their convergence to a ``true'' solution of the original equations.

\begin{rmk} \label{r:Els}
As a consequence of Theorem \ref{th:symm}, the two systems \eqref{i_eq:MHD-I} and \eqref{eq:MHDab} are equivalent in the functional framework considered in Theorem \ref{th:BesovWP}.
Hence, throughout we will use the two formulations equivalently, depending on which one is more convenient for our scopes.
\end{rmk}

\subsubsection{\tsl{A priori} estimates} \label{sss:bounds}

In order to avoid the use of the Leray projector $\P$ (broadly employed in the analysis of \cite{Cobb-F_Rig}), the basic idea for deriving \tsl{a priori} estimates in spaces $\B$
is to resort to the \emph{vorticity formulation} of the Els\"asser system \eqref{eq:MHDab}.

For want of better notations, we call $X$ and $Y$ the vorticity matrices, defined according to \eqref{def:vort}: more precisely,
\begin{equation*}
\big[ \curl(\al)\big]_{ij} = X_{ij} = \partial_j \al_i - \partial_i \al_j \qquad\quad \text{ and } \quad\qquad \big[ \curl(\bt)\big]_{ij} = Y_{ij} = \partial_j \bt_i - \partial_i \bt_j\,.
\end{equation*}
In dimension $d = 2$, the $\curl$ can be identified with the scalar function $X = \partial_1 \al_2 - \partial_2 \al_1$, and, when $d = 3$, with the vector field
$X = \nabla \times \al$. However, we work in any dimension $d \geq 2$ of space.

Applying the $\curl$ to the second and third equations of system \eqref{eq:MHDab}, we get
\begin{equation}\label{eq:MHDVorticity}
\begin{cases}
\partial_t  R + \dfrac{1}{2}(\al + \bt) \cdot \nabla R = 0 \\[1ex]
\partial_t X + (\bt \cdot \nabla) X  = \L \big( \nabla \al, \nabla \bt \big) - \dfrac{1}{2} \curl \big( R \mathfrak{C} (\alpha + \bt) \big)\\[1ex]
\partial_t Y + (\al \cdot \nabla) Y = \L \big( \nabla \bt, \nabla \al \big) - \dfrac{1}{2} \curl \big( R \mathfrak{C} (\alpha + \bt) \big)\,,
\end{cases}
\end{equation}
where $\mc L$ denotes the bilinear matrix-valued operator defined by
\begin{equation}\label{eq:LOperator}
\big[ \L(\nabla \al, \nabla \bt) \big]_{ij}\, =\, \sum_{k=1}^d\left(\partial_j \bt_k\, \partial_k \al_i \,-\, \partial_i \bt_k\, \partial_k \al_j\right)\,,
\end{equation}
or, in other words, $\L (\nabla \al, \nabla \bt) = {}^{t} (\nabla \bt \, \nabla \al) - \nabla \bt \, \nabla \al$. 

The main result of this part is stated in the next result, which contains basic \tsl{a priori} bounds for smooth solutions to system \eqref{eq:MHDab}.

\begin{prop} \label{p:a-priori}
Let $(s, r)$ be such that the Lipschitz condition \eqref{i_eq:Lip} is satisfied.
Let $(R,\al, \bt)$ be regular solutions to the symmetrised system \eqref{eq:MHDab}, related to regular initial data $(R_0,\al_0, \bt_0)$, with $\al_0$ and $\bt_0$ being divergence-free.

Then, there exist a constant $C>0$, which depends on the dimension $d$ and $(s, r)$, as well as a time $T^*>0$, which depends on the above and $\big\|(R_0,\al_0, \bt_0)\big\|_{\B}$, such that
\begin{equation*}
\big\|R(t)\big\|_{\B}\,+\, \big\|\big(\al(t), \bt(t)\big)\big\|_{B^{s}_{\infty, r}\cap L^2} \,\leq\, C\,
e^{C\,T^*\,\|R_0\|_{L^\infty}}\,\Big(\big\|R_0\big\|_{\B}\,+\,\big\|\big(\al_0, \bt_0\big)\big\|_{B^{s}_{\infty, r}\cap L^2}\Big)
\end{equation*}
for all $t\in[0,T^*]$.
Moreover, we have the inequality
\begin{equation*} 
T^*\, \geq\, \frac{C}{\|R_0\|_{L^\infty}}\; {\rm argsinh}\! \left(  \frac{ C\, \|R_0\|_{L^\infty}}{\big\|R_0\big\|_{\B}\,+\,\big\|\big(\al_0, \bt_0\big)\big\|_{B^{s}_{\infty, r}\cap L^2}}  \right)\,.
\end{equation*}
\end{prop}

The rest of this paragraph is devoted to the proof of Proposition \ref{p:a-priori}. Let us begin with a simple lemma, which allows to bound the forcing terms appearing
in equations \eqref{eq:MHDVorticity}.

\begin{lemma}\label{l:CommBPInfinite}
Let $(R,\al, \bt)$ be a triplet of functions in $B^s_{\infty, r}$, with $R$ being scalar and the vector fields $\al$ and $\bt$ being divergence-free.
Then we have the following inequalities: 
\begin{align*}
\left\| \curl \big( R\, \mathfrak{C}(\al + \bt) \big) \right\|_{B^{s-1}_{\infty, r}}\, &\lesssim\,\| R \|_{L^\infty}\, \|\al + \bt\|_{B^{s}_{\infty, r}}\, +\,
\| \al + \bt \|_{L^\infty}\, \| R \|_{B^s_{\infty, r}} \\
\big\| \L(\nabla \al, \nabla \bt) \big\|_{B^{s-1}_{\infty, r}}\, &\lesssim\,\| \nabla \al \|_{L^\infty}\, \| \bt \|_{B^s_{\infty, r}}\, +\, \| \nabla \bt \|_{L^\infty}\, \| \al \|_{B^s_{\infty, r}}\,.
\end{align*}
\end{lemma}

\begin{proof}[Proof of Lemma \ref{l:CommBPInfinite}]
The first estimate is a simple consequence of the tame estimates for the Banach algebra $B^{s}_{\infty, r}$, see Corollary \ref{c:tame} above.
Likewise, the second estimate also follows from the tame estimates if $s>1$. The main difficulty is proving this last inequality for the endpoint case $s=r=1$, namely for $\nabla \al$ and $\nabla \bt$ lying in $B^0_{\infty, 1}$, which is not an algebra (keep in mind Remark \ref{r:tame}).

To overcome the problem of working with a $0$ regularity index, we use the fact that $\al$ and $\bt$ are divergence-free to rewrite things in the following way:
\begin{equation}\label{eq:FTrick}
\big[ \L(\nabla \al, \nabla \bt) \big]_{ij}\,  =\, \sum_{k=1}^d\Big(\partial_k \big(\al_i\, \partial_j \bt_k\big)\, -\, \partial_k \big(\al_j\, \partial_i \bt_k\big)\Big)\,.
\end{equation}

Now, making use of the Bony decomposition of a product, we get
\begin{equation*}
\L(\nabla \al, \nabla \bt)\, =\,  
\L_{\mc T}(\nabla\al,\nabla\bt)\,+\,\L_{\mc R}(\nabla\al,\nabla\bt)\,,
\end{equation*}
where we have defined
\begin{align*}
\big[\L_{\mc T}(\nabla\al,\nabla\bt)\big]_{ij}\,&:=\,\sum_{k=1}^d\Big(\mc T_{\d_k\al_i} (\d_j\bt_k) \,+\,\mc T_{\d_j\bt_k} (\d_k\al_i) \,-\,\mc T_{\d_k\al_j} (\d_i\bt_k) \,-\,\mc T_{\d_i\bt_k}(\d_k\al_j) \Big) \\
\big[ \L_{\mathcal{R}}(\nabla \al, \nabla \bt) \big]_{ij}\,&:=\,\sum_{k=1}^d \Big(\mathcal{R} (\d_k\al_i, \partial_j \bt_k)\, -\, \mathcal{R}(\d_k\al_j, \partial_i \bt_k)\Big)\,.
\end{align*}
On the one hand, thanks to Proposition \ref{p:op}, we can easily estimate the paraproducts: with a little abuse of notation, we may write
\begin{equation*}
\big\| \mathcal{T}_{\nabla \al} (\nabla \bt) \big\|_{B^0_{\infty, 1}} + \big\| \mathcal{T}_{\nabla \bt} (\nabla \al) \big\|_{B^0_{\infty, 1}}\,\lesssim\,
\| \nabla \al \|_{L^\infty}\, \| \nabla \bt \|_{B^0_{\infty, 1}}\, +\, \| \nabla \bt \|_{L^\infty}\, \| \nabla \al \|_{B^0_{\infty, 1}}\,.
\end{equation*}
On the other hand, using equation \eqref{eq:FTrick}, we can write the remainder terms in the following form:
\begin{equation*}
\big[ \L_{\mathcal{R}}(\nabla \al, \nabla \bt) \big]_{ij}\, =\,\sum_{k=1}^d\Big( \partial_k \mathcal{R} (\al_i, \partial_j \bt_k)\, -\, \partial_k \mathcal{R}(\al_j, \partial_i \bt_k)\Big)\,.
\end{equation*}
Now, each of the summands can be bounded thanks to Proposition \ref{p:op}. For instance, the first one is bounded by
\begin{equation*}
\big\| \partial_k \mathcal{R} (\al_j, \partial_i \bt_k) \big\|_{B^0_{\infty, 1}} \leq \big\| \mathcal{R} (\al_j, \partial_i \bt_k) \big\|_{B^1_{\infty, 1}} \lesssim \| \nabla \al \|_{B^0_{\infty, \infty}} \| \bt \|_{B^1_{\infty, 1}} \lesssim \| \nabla \al \|_{L^\infty} \| \bt \|_{B^1_{\infty, 1}}\,,
\end{equation*}
where the last inequality is due to the embedding $L^\infty \hookrightarrow B^0_{\infty, \infty}$. The other summand can be dealt with in a symmetric way. Putting all this together, we finally get the
sought bound for $\L(\nabla\al,\nabla\bt)$ in the space $B^0_{\infty,1}$.
The lemma is thus proved.
\end{proof}

With the estimates of Lemma \ref{l:CommBPInfinite} at hand, we can tackle the proof of the proposition.

\begin{proof}[Proof of Proposition \ref{p:a-priori}]
We start by bounding the $L^p$ norms of the solutions. First of all, since $R$ is merely transported by a divergence-free vector field, we get
\begin{equation} \label{est:R_inf}
\forall\,t\geq0\,,\qquad\qquad\|R(t)\|_{L^\infty}\,=\,\|R_0\|_{L^\infty}\,\leq\,\|R_0\|_{\B}\,.
\end{equation}
On the other hand, a simple energy estimate for the equations for $\al$ and $\bt$ in \eqref{eq:MHDab} yields, for some constant $c>0$ depending only on the coefficients of $\mf C$,
the inequality
\begin{equation} \label{est:en-ab}
\big\| \big(\al(t), \bt(t)\big) \big\|_{L^2}\, \leq\, \big\| \big(\al_0, \bt_0\big) \big\|_{L^2}\, e^{c\,t\, \| R_0 \|_{L^\infty}}\,,
\end{equation}
where we also used \eqref{est:R_inf} above and Gr\"onwall's lemma.

Next, 
assume the function $f \in B^s_{\infty, r} \cap L^2$ to be divergence-free, and denote by $\curl(f)$ its ``vorticity matrix'', defined according to \eqref{def:vort}:
$$
\curl(f)\,:=\,Df\,-\,\nabla f\,,\qquad\qquad\mbox{ so that }\qquad \big[\curl(f)\big]_{ij}\,:=\,\d_jf_i\,-\,\d_if_j\,.
$$
Then, using the divergence-free condition on $f$, we have the \emph{Biot-Savart law}
$$
\forall\,j\in[1,d]\,,\qquad\qquad f_j\,=\,(-\Delta)^{-1}\sum_{i=1}^d\d_i\big[\curl(f)\big]_{ij}\,.
$$
From the previous equality, by separating low and high frequencies, we deduce
\begin{align*}
\| f \|_{B^s_{\infty, r}} &\sim \sum_{j=1}^d \left\| \Delta_{-1} (- \Delta)^{-1}\sum_i\d_i\big[\curl(f)\big]_{ij} \right\|_{L^\infty} \\
&\qquad\qquad +\,\left\| \mds{1}_{\left\{\nu \geq 0\right\}}\, 2^{sj}\, \Big\| \Delta_\nu (- \Delta)^{-1}\sum_i\d_i\big[\curl(f)\big]_{ij} \Big\|_{L^\infty} \right\|_{\ell^r(\nu \geq 0)}\,.
\end{align*}
On the one hand, if $\nu \geq 0$, we know that $\Delta_\nu \big[\curl(f)\big]_{ij}$ is spectrally supported in an annulus, on which the symbol of the order $-1$ Fourier multiplier
$(-\Delta)^{-1} \d_i$ is smooth. Hence, by using the Bernstein inequalities of Lemma \ref{l:bern}, we get 
\begin{equation*}
2^{s\nu}\, \Big\| \Delta_\nu (- \Delta)^{-1}\sum_i\d_i\big[\curl(f)\big]_{ij} \Big\|_{L^\infty}\,\lesssim\, 2^{(s-1)\nu}\, \big\| \Delta_\nu\curl(f) \big\|_{L^\infty}\,.
\end{equation*}
On the other hand, using the fact that the symbol of $(-\Delta)^{-1}\nabla\curl$ is homogeneous of degree $0$ and bounded on the unit sphere $|\xi| = 1$, thus $L^\infty$ in a neighborhood of the origin,
Bernstein inequalities and Plancherel's theorem yield
\begin{equation*} 
\left\| \Delta_{-1} (- \Delta)^{-1}\sum_i\d_i\big[\curl(f)\big]_{ij} \right\|_{L^\infty}\,\lesssim\,\left\| \Delta_{-1} (- \Delta)^{-1}\,\nabla\curl(f)\right\|_{L^2}\,\lesssim\,\|f\|_{L^2}\,.
\end{equation*}
Therefore, in the end, we deduce
\begin{equation}\label{eq:abFromCurl}
\| f \|_{B^s_{\infty, r}}\, \lesssim\, \| f \|_{L^2}\, +\, \big\| \curl(f) \big\|_{B^{s-1}_{\infty, r}}\,.
\end{equation}
In view of \eqref{est:en-ab}, this bound tells us that, for bounding $(\al,\bt)$ in $\B$, we can focus on estimates for the $B^{s-1}_{\infty,r}$ norms of $X$ and $Y$, which solve system \eqref{eq:MHDVorticity}.

Let $\Delta_j$ be a dyadic block. Applying it to system \eqref{eq:MHDVorticity}, we find
\begin{equation*}
\begin{cases}
\partial_t(\Delta_j R) + \dfrac{1}{2} (\al + \bt)\cdot \nabla \Delta_j R = \dfrac{1}{2} \big[ (\al + \bt)\cdot \nabla , \Delta_j \big] R \\[1ex]
\partial_t (\Delta_j X) + (\bt \cdot \nabla) \Delta_j X = \big[ \bt \cdot \nabla, \Delta_j \big] X + \Delta_j \L (\nabla \al, \nabla \bt) - \dfrac{1}{2} \Delta_j \curl \big( R\, \mathfrak{C} (\al + \bt) \big)\\[1ex]
\partial_t (\Delta_j Y) + (\al \cdot \nabla) \Delta_j Y = \big[ \al \cdot \nabla, \Delta_j \big] Y +  \Delta_j \L (\nabla \bt, \nabla \al) - \dfrac{1}{2} \Delta_j \curl \big( R\, \mathfrak{C} (\al + \bt) \big)\,.
\end{cases}
\end{equation*}

Lemma \ref{l:CommBPInfinite} gives estimates in $B^{s-1}_{\infty, r}$ for all right-hand side terms, except the commutators:
\begin{equation}\label{eq:pinftyIN1}
\big\| \L (\nabla \al, \nabla \bt)  \big\|_{B^{s-1}_{\infty, r}} + \big\|  \curl \big( R\, \mathfrak{C} (\al + \bt) \big) \big\|_{B^{s-1}_{\infty, r}}\,
\lesssim\, \|\al\|_{B^s_{\infty, r}} \,\|\bt\|_{B^s_{\infty, r}}\, +\, \|\al + \bt\|_{B^s_{\infty, r}}\, \|R\|_{B^s_{\infty, r}}\,.
\end{equation}
For the commutator terms, we use instead Lemma \ref{l:CommBCD}. Note that, because we seek $B^s_{\infty, r}$ estimates on $(R,\al, \bt)$,
the vorticities $X$ and $Y$ will only be bounded in $B^{s-1}_{\infty, r}$, which does not contain the space $W^{1, \infty}$ of Lipschitz functions if $s< 2$, or $s=2$ and $r\neq 1$.
In those cases, we must use the second inequality of Lemma \ref{l:CommBCD}. In both situations, however, we get the bound
\begin{multline}\label{pinftyIN2}
2^{j(s-1)} \Big( \left\| \big[ \bt \cdot \nabla, \Delta_j \big] X \right\|_{L^\infty} + \left\| \big[ \al \cdot \nabla, \Delta_j \big] Y \right\|_{L^\infty} \Big) +
2^{js} \left\| \big[ (\al + \bt) \cdot \nabla, \Delta_j \big] R \right\|_{L^\infty}  \\
\lesssim c_j(t)\Big( \|\al\|_{B^s_{\infty, r}}\, \|\bt\|_{B^s_{\infty, r}}\, +\, \|\al + \bt\|_{B^s_{\infty, r}}\, \|R\|_{B^s_{\infty, r}} \Big)\,,
\end{multline}
for a suitable sequence $\big(c_j(t)\big)_{j\geq -1}$ belonging to the unit sphere of $\ell^r$. Using all this to write an $L^\infty$ estimate for $\Delta_j (R,X, Y)$, we get
\begin{align}\label{eq:ContCritLInftyEstimate}
&2^{j(s-1)} \big\| \Delta_j \big(X(t), Y(t)\big) \big\|_{L^\infty} + 2^{js} \| \Delta_j R(t) \|_{L^\infty}\,\lesssim\,
2^{j(s-1)} \big\| \Delta_j (X_0, Y_0) \big\|_{L^\infty} + 2^{js} \| \Delta_j R_0 \|_{L^\infty} \\
&\qquad\qquad\qquad\qquad\qquad\qquad\qquad
+\int_0^t c_j(\tau) \bigg\{ \|\al\|_{B^s_{\infty, r}} \|\bt\|_{B^s_{\infty, r}} + \|\al + \bt\|_{B^s_{\infty, r}} \|R\|_{B^s_{\infty, r}}  \bigg\} {\rm d} \tau\,, \nonumber
\end{align}

At this point, for all $t\geq0$, we set
\begin{equation*} 
E(t)\,:=\, \| R(t) \|_{B^s_{\infty_r}}\,+\, \big\| \big(\al(t), \bt(t) \big) \big\|_{L^2}\,+\,\big\| \big(X(t), Y(t)\big) \big\|_{B^{s-1}_{\infty, r}}\,.
\end{equation*}
Using the energy estimate \eqref{est:en-ab} and the previous inequality \eqref{eq:ContCritLInftyEstimate}, we get, thanks to \eqref{eq:abFromCurl} and the Minkowski inequality
(see Proposition 1.3 in \cite{BCD}), the bound
\begin{equation*}
E(t)\, \lesssim\, E(0)\;\exp\big(c\, t\, \|R_0\|_{L^\infty} \big)\, +\, \int_0^t E(\tau)^2\, {\rm d} \tau\,.
\end{equation*}
To end the proof, we define the time $T^*>0$ by
\begin{equation*}
T^*\, =\, \sup \left\{ T > 0\; \bigg|\quad \int_0^t E(\tau)^2\, {\rm d} \tau\, \leq\, E(0)\, e^{c\,t\, \| R_0 \|_{L^\infty}} \; \right\}\,.
\end{equation*}
Then we deduce $E(t) \leq CE(0)e^{ct \| R_0 \|_{L^\infty}}$ for all times $t \in [0, T^*]$ and for some positive constant $C = C(d,s,r)$. Therefore, for such times, the following
inequality holds true:
\begin{equation*}
\int_0^t E(\tau)^2\, {\rm d} \tau\, \leq\,  \frac{C\,E(0)^2}{2\,c\,\| R_0 \|_{L^\infty}} \left( e^{2ct\, \|R_0\|_{L^\infty} } -1  \right)\,.
\end{equation*}
By using the definition of $T^*$, we see that
\begin{equation*}
T^* \geq \frac{C_1}{\|R_0\|_{L^\infty}} {\rm argsinh} \left( C_2 \frac{\|R_0\|_{L^\infty}}{E(0)}  \right)\,,
\end{equation*}
for some suitable positive constants $C_1$ and $C_2$. This ends the proof of the proposition.
\end{proof}

\subsubsection{Proof of existence} \label{sss:approx}

In the previous paragraph, we have shown \tsl{a priori} bounds, in the relevant norms, for smooth solutions to system \eqref{i_eq:MHD-I}. Here, we present the proof of the existence
of solutions at the claimed level of regularity.

For this, we follow a standard procedure: first of all, we construct a sequence of smooth solutions
to approximate problems. Next, from the estimates of Paragraph \ref{sss:bounds} we deduce uniform bounds for that sequence of approximate solutions.
Finally, by use of those uniform bounds and an energy argument, we are able to show strong convergence properties for suitable quantities, which in turn allow us
to take the limit in the approximation parameter and gather the existence of a solution to the original problem.

Throughout this paragraph, we will exploit the equivalence of equations \eqref{i_eq:MHD-I}  with the Els\"asser formulation \eqref{eq:MHDab}, as established
by Theorem \ref{th:symm}. Also, for simplicity we are going to assume $r<+\infty$: the case $r=+\infty$ can be handled with minor modifications.

\paragraph{Construction of smooth approximate solutions.}
For any $n\in\N$, let us define
$$
\big(R^n_0\,,\,\al^n_0\,,\,\bt^n_0\big)\,:=\,\big(S_nR_0\,,\,S_n\al_0\,,\,S_n\bt_0\big)\,,
$$
where $S_n$ is the low frequency cut-off operator introduced in \eqref{eq:S_j}. By the finite energy assumption $\al_0, \bt_0 \in L^2$, one has, for any $n\in\N$, $\al^n_0,\bt^n_0\,\in H^\infty:=\bigcap_{\s\in\R}H^\s$, which is obviously embedded (for a suitable topology on $H^\infty$) in the space $C^\infty_b$ of $C^\infty$ functions which are globally bounded together with all their derivatives.
Analogously, for the density functions we have $R^n_0\in\bigcap_{\s\in\R}B^\s_{\infty,r}\hookrightarrow C^\infty_b$.
In addition, we have
\begin{equation} \label{conv:in-data}
R^n_0\,\tend_{n\ra+\infty}\,R_0\quad\mbox{ in }\ \B\qquad\mbox{ and }\qquad \big(\al^n_0\,,\,\bt^n_0\big)\,\tend_{n\ra+\infty}\,\big(\al_0\,,\,\bt_0\big)\quad \mbox{ in }\ L^2\cap\B\,.
\end{equation}

This having been done, we are going to define a sequence of approximate solutions to system \eqref{eq:MHDab} by induction. First of all, we set $\big(R^0\,,\,\al^0\,,\,\bt^0\big)\,:=\,\big(R^0_0\,,\,\al^0_0\,,\,\bt^0_0\big)$.
Obviously, for all $\s\in\R$, we have that $R^0\,\in\, C^0\big(\R_+;B^\s_{\infty,r}\big)$ and $\al^0\,,\,\bt^0\;\in\, C^0\big(\R_+;H^\s\big)$,
with $\div(\al^0)=\div(\bt^0)=0$.

Next, assume that the triplet $\big(R^n,\al^n,\bt^n\big)$ is given, with, for all $\s\in\R$, the properties
$$
R^n\,\in\, C^0\big(\R_+;B^\s_{\infty,r}\big)\,,\qquad \al^n\,,\,\bt^n\;\in\, C^0\big(\R_+;H^\s\big)\qquad\mbox{ and }\qquad \div(\al^n)\,=\,\div(\bt^n)\,=\,0\,.
$$
First of all, we define $R^{n+1}$ as the unique solution to the linear transport equation
\begin{equation} \label{eq:R^n}
\d_tR^{n+1}\,+\,\frac{1}{2}\,\left(\al^n\,+\,\bt^n\right)\cdot\nabla R^{n+1}\,=\,0\,,\qquad\mbox{ with }\quad R^{n+1}_{|t=0}\,=\,R^{n+1}_0\,.
\end{equation}
Since, by inductive hypothesis and embeddings, the transport field $\al^n+\bt^n$ is divergence-free, smooth and uniformly bounded with all its derivatives, we can apply Theorem \ref{th:transport}
to propagate all the $B^\s_{\infty,r}$ norms of the initial datum. We deduce that $R^{n+1}\in C^0\big(\R_+;B^\s_{\infty,r}\big)$, for any $\s\in\R$.

Next, we solve the two (linear) transport equations with divergence-free constraints
\begin{equation} \label{eq:al-bt^n}
\begin{cases}
\partial_t\,\al^{n+1}\, +\, (\bt^n \cdot \nabla) \al^{n+1}\,+\,\nabla\pi_1^{n+1}\, =\,-\,\dfrac{1}{2}\,R^{n+1}\,\mf C\left(\al^n\,+\,\bt^n\right)  \\[1ex]
\partial_t\,\bt^{n+1}\, +\, (\al^n \cdot \nabla) \bt^{n+1}\,+\,\nabla\pi_2^{n+1}\, =\,-\,\dfrac{1}{2}\,R^{n+1}\,\mf C\left(\al^n\,+\,\bt^n\right)  \\[1ex]
\div\left(\al^{n+1}\right)\,=\,\div\left(\bt^{n+1}\right)\,=\,0\,, 
\end{cases}
\end{equation}
with initial data $\al^{n+1}_{|t = 0}=\al^{n+1}_0$ and $\bt^{n+1}_{|t = 0}\, =\bt^{n+1}_0$,  to define the vector fields $\al^{n+1}$ and $\bt^{n+1}$. Notice that the right-hand side
of the previous equations belongs to $L^1_{\rm loc}\big(\R_+;H^\s\big)$ for any $\s\in\R$, thanks to the inductive assumption and the previous regularity property for $R^{n+1}$.
It is not hard to solve the previous linear probem by energy methods; see also Propositions 3.2 and 3.4 of \cite{D} in this respect. We thus find unique solutions $\al^{n+1}$ and $\bt^{n+1}$, belonging to the space $C^0\big(\R_+;H^\s\big)$ for all $\s\in\R$.

We omit here the analysis of the pressure gradients $\nabla\pi_1^{n+1}$ and $\nabla\pi_2^{n+1}$ (which are present to restore the divergence-free conditions on $\al^{n+1}$ and $\bt^{n+1}$),
since they are not needed in the rest of the present proof. However, this analysis can be performed following the argument we will use in the last paragraph of this section,
in order to establish the regularity of the (limit) pressure functions $\nabla\pi_1$ and $\nabla\pi_2$.

\paragraph{Uniform bounds for the approximate solutions.}
We now have to show uniform bounds for the sequence $\big(R^n,\al^n,\bt^n\big)_{n\in\N}$ we have constructed above. We argue by induction, and prove that there exists a time $T>0$ such that
the following property holds true: for all $t\in[0,T]$ and all $n\in\N$, one has
\begin{align}
&\left\|R^{n}(t)\right\|_{L^\infty}\,\leq\,C\,\|R_0\|_{L^\infty}
\qquad\mbox{ and }\qquad
\left\|\big(\al^n(t)\,,\,\bt^n(t)\big)\right\|_{L^2}\,\leq\,C\,\left\|\left(\al_0\,,\,\bt_0\right)\right\|_{L^2}\,e^{c\|R_0\|_{L^\infty}t}\,, \label{ind:energy} \\
&\big\| R^{n}(t) \big\|_{B^s_{\infty,r}}\,+\, \big\| \big(\al^{n}(t), \bt^{n}(t) \big) \big\|_{L^2\cap \B}\,\leq\,C\,K_0\,e^{C\,K_0\,t}\,, \label{ind:Besov}
\end{align}
where $c\,:=\,\sup_{j,k}|\mf C_{jk}|$, where the constant $C>0$ does not depend on the data or the solutions, and therefore not on $n$, and where we have set
$$
K_0\,:=\,\big\| R_0 \big\|_{B^s_{\infty,r}}\,+\, \big\| \big(\al_0, \bt_0 \big) \big\|_{L^2\cap \B}\,.
$$

It is obvious that the initial triplet $\big(R^0,\al^0,\bt^0\big)$ satisfies the previous requirements. Assume now that, for some $n\in\N$, the triplet $\big(R^n,\al^n,\bt^n\big)$
verifies the same conditions on some time interval $[0,T]$. We want to prove that, in the same time interval, $\big(R^{n+1},\al^{n+1},\bt^{n+1}\big)$
also enjoys the same properties.

First of all, we consider 
the propagation of $L^2$ and $L^\infty$ norms. By equation \eqref{eq:R^n} and the divergence-free condition on $\al^n$ and $\bt^n$, we immediately deduce that, for any $t\geq0$, one has
$$
\left\|R^{n+1}(t)\right\|_{L^\infty}\,=\,\left\|R^{n+1}_0\right\|_{L^\infty}\,\leq\,C\,\left\|R_0\right\|_{L^\infty}\,\leq\,C\,\|R_0\|_{\B}\,,
$$
so that the first property of \eqref{ind:energy} is also verified by $R^{n+1}$. 
Next, a simple energy estimate for equations \eqref{eq:al-bt^n} yields, for any $t\geq0$, the inequality
\begin{align*}
\left\|\big(\al^n(t),\bt^n(t)\big)\right\|_{L^2}\,&\leq\,\left\|\big(\al^n_0,\bt^n_0\big)\right\|_{L^2}\,+\,\int^t_0\left\|R^{n+1}\,\mf C(\al^n+\bt^n)\right\|_{L^2}\,{\rm d}\t \\
&\leq\,C\,\left\|\big(\al_0,\bt_0\big)\right\|_{L^2}\,+\,c\,\|R^{n+1}_0\|_{L^\infty}\int^t_0\left\|\big(\al^n,\bt^n\big)\right\|_{L^2}\,{\rm d}\t\,.
\end{align*}
Using the induction hypothesis, we also get the second inequality in \eqref{ind:energy} at step $n+1$, that is for $\big(\al^{n+1},\bt^{n+1}\big)$.

In order to get bounds for the Besov norms, we resort to the vorticity formulation of \eqref{eq:al-bt^n}: applying the $\curl$ operator to that system leads us to
\begin{equation}\label{eq:Vort^n}
\begin{cases}
\partial_t X^{n+1} + (\bt^n \cdot \nabla) X^{n+1}  = \L \big( \nabla \al^{n+1}, \nabla \bt^n \big) - \dfrac{1}{2} \curl \big( R^{n+1} \mathfrak{C} (\alpha^n + \bt^n) \big)\\[1ex]
\partial_t Y^{n+1} + (\al^n \cdot \nabla) Y^{n+1} = \L \big( \nabla \bt^{n+1}, \nabla \al^n \big) - \dfrac{1}{2} \curl \big( R^{n+1} \mathfrak{C} (\alpha^n + \bt^n) \big)\,.
\end{cases}
\end{equation}
Proceeding exactly as in the proof of Proposition \ref{p:a-priori}, but for the equations \eqref{eq:R^n} and \eqref{eq:Vort^n}, we find an estimate analogous to \eqref{eq:ContCritLInftyEstimate}:
for any $t\geq0$ one has
\begin{align}\label{est:Besov-n}
&2^{j(s-1)} \big\| \Delta_j \big(X^{n+1}(t), Y^{n+1}(t)\big) \big\|_{L^\infty} + 2^{js} \left\| \Delta_j R^{n+1}(t) \right\|_{L^\infty}\\
&\qquad\qquad\qquad\lesssim\,2^{j(s-1)} \big\| \Delta_j \big(X^{n+1}_0, Y^{n+1}_0\big) \big\|_{L^\infty} + 2^{js} \left\| \Delta_j R^{n+1}_0 \right\|_{L^\infty} \nonumber \\
&\qquad\qquad\qquad\qquad\qquad+\int_0^t c_j(\tau) \left\|\big(R^{n+1},\al^{n+1},\bt^{n+1}\big)\right\|_{B^s_{\infty, r}}\,\left\|\big(\al^n,\bt^n\big)\right\|_{B^s_{\infty, r}}
{\rm d} \tau\,, \nonumber
\end{align}
where, as usual, the sequence $\big(c_j(t)\big)_j$ belongs to the unit sphere of $\ell^r$.

At this point, for all $t\geq0$ we set
\begin{equation*} 
E^{n+1}(t)\,:=\, \big\| R^{n+1}(t) \big\|_{B^s_{\infty,r}}\,+\, \big\| \big(\al^{n+1}(t), \bt^{n+1}(t) \big) \big\|_{L^2\cap \B}\,. 
\end{equation*}
Recall that, in view of \eqref{eq:abFromCurl}, one has
\begin{equation} \label{est:equiv-norm}
\|f\|_{L^2\cap\B}\,\sim\,\|f\|_{L^2}\,+\,\|\curl(f)\|_{B^{s-1}_{\infty,r}}\,.
\end{equation}
Thus, taking the $\ell^r$ norm in \eqref{est:Besov-n} and
using the energy inequality \eqref{ind:energy} at level $n+1$, we obtain
$$
E^{n+1}(t)\,\leq\,C\,\left(E^{n+1}(0)\,e^{c\|R_0\|_{L^\infty}t}\,+\,\int^t_0E^{n+1}(\t)\,\left\|\big(\al^n(\t),\bt^n(\t)\big)\right\|_{B^s_{\infty, r}}\,{\rm d}\t\right)\,.
$$
An application of Gr\"onwall's lemma and the fact that $E^{n+1}(0)\lesssim K_0$ finally gives
\begin{equation} \label{est:E^n}
E^{n+1}(t)\,\leq\,C\,K_0\,\exp\left(C\int^t_0\left\|\big(\al^n(\t),\bt^n(\t)\big)\right\|_{B^s_{\infty, r}}\,{\rm d}\t\,+\,c\,\|R_0\|_{L^\infty}\,t\right)\,.
\end{equation}
Before going on, we remark that, from Theorem \ref{th:transport}, we could have obtained a more precise inequality for the Besov norm of the density $R^{n+1}$, \tsl{i.e.}
$$
\left\|R^{n+1}(t)\right\|_{\B}\,\leq\,C\,\left\|R^{n+1}_0\right\|_{\B}\,\exp\left(\int^t_0\left\|\big(\al^n(\t),\bt^n(\t)\big)\right\|_{\B}\,{\rm d}\t\right)\,.
$$
However, this inequality does not really simplify the inductive argument. Therefore, let us resume from estimate \eqref{est:E^n}: from the inductive assumption \eqref{ind:Besov},
we get
$$
\int^t_0\left\|\big(\al^n(\t),\bt^n(\t)\big)\right\|_{\B}\,{\rm d}\t\,\leq\,\frac{C}{c}\,\left(e^{c\,K_0\,t}\,-\,1\right)\,.
$$
Observe that, for $0\leq x\leq1$, one has $e^x-1\,\leq\,x+x^2\,\leq\,2x$. So, if $T>0$ is chosen so small that $cK_0T\leq1$, from the previous bound and \eqref{est:E^n} we finally deduce
$$
E^{n+1}(t)\,\leq\,C\,K_0\,\exp\left(\frac{C}{c}\,\left(e^{c\,K_0\,t}\,-\,1\right)\,+\,c\,\|R_0\|_{L^\infty}\,t\right)\,\leq\,C\,K_0\,e^{(2C+c)K_0\,t}\,,
$$
completing in this way the proof of \eqref{ind:Besov} at the level $n+1$.

\paragraph{Convergence.}
It remains us to show convergence of the sequence $\big(R^n,\al^n,\bt^n\big)_n$ towards a solution $(R,\al,\bt)$ of the original problem \eqref{eq:MHDab}: this is our next goal.

To begin with, we introduce the quantity $\rho^n\,:=\,R^n-R^n_0$, which satisfies the transport equation
$$
\d_t\rho^n\,+\,\frac{1}{2}\,\big(\al^{n-1}\,+\,\bt^{n-1}\big)\cdot\nabla\rho^n\,=\,-\,\frac{1}{2}\,\big(\al^{n-1}\,+\,\bt^{n-1}\big)\cdot\nabla R^n_0\,,\qquad\qquad\rho^n_{|t=0}\,=\,0\,.
$$
Thus, it is easy to check that the sequence $\big(\rho^n\big)_n$ is uniformly bounded (with respect to $n$) in the space
$C^0\big([0,T];L^2\big)$.

By an energy method, similar to the one used for proving uniqueness, we are going to show that $\big(\rho^n\big)_n$, $\big(\al^n\big)_n$ and $\big(\bt^n\big)_n$ are Cauchy sequences in the previous space
$C^0\big([0,T];L^2\big)$. For this, we introduce the following notation: for any couple $(n,p)\in\N^2$, we define the quantities
\begin{align*}
\de\rho^{n,p}\,:=\,\rho^{n+p}-\rho^n\,,\qquad
\de\al^{n,p}\,:=\,\al^{n+p}-\al^n\qquad\mbox{ and }\qquad
\de\bt^{n,p}\,:=\,\bt^{n+p}-\bt^n\,.
\end{align*}
Of course, $\div\de\al^{n,p}\,=\,\div\de\bt^{n,p}\,=\,0$ for any $(n,p)\in\N^2$. In addition, after setting $\de\pi_j^{n,p}:=\pi_j^{n+p}-\pi_j^n$ for $j=1,2$
and 
$\de R^{n,p}\,=\,R^{n+p}-R^n$,
simple computations yield the system of equations
\begin{equation} \label{eq:Cauchy}
\begin{cases}
\partial_t\delta\rho^{n,p} + \dfrac{1}{2} (\al^{n+p-1} + \bt^{n+p-1}) \cdot \nabla \delta\rho^{n,p} \\[1ex]
\qquad\qquad = - \dfrac{1}{2} \big( \delta \al^{n-1,p} + \delta \bt^{n-1,p} \big) \cdot \nabla R^n-\dfrac{1}{2}(\al^{n+p-1}+\bt^{n+p-1})\cdot\nabla\de R^{n,p}(0) \\[1ex]
\partial_t\delta \al^{n,p} + ({\bt}^{n+p-1} \cdot \nabla ) \delta\al^{n,p} + \nabla\de\pi^{n,p}_1 \\[1ex]
\qquad\qquad =-\de\bt^{n-1,p}\cdot\nabla\al^n- \dfrac{1}{2} R^{n+p}\mathfrak{C}\big( \delta \al^{n-1,p} + \delta \bt^{n-1,p} \big) -
\dfrac{1}{2} \delta R^{n,p} \mathfrak{C} \big( {\al}^{n-1} + {\bt}^{n-1} \big) \\[1ex]
\partial_t\delta \bt^{n,p} + ({\al}^{n+p-1} \cdot \nabla ) \delta\bt^{n,p} + \nabla\de\pi^{n,p}_2 \\[1ex]
\qquad\qquad =-\de\al^{n-1,p}\cdot\nabla\bt^n- \dfrac{1}{2} R^{n+p}\mathfrak{C}\big( \delta \al^{n-1,p} + \delta \bt^{n-1,p}\big) -
\dfrac{1}{2} \delta R^{n,p} \mathfrak{C} \big( {\al}^{n-1} + {\bt}^{n-1} \big)\,,
\end{cases}
\end{equation}
supplemented with initial data $\big(\de\rho^{n,p},\de\al^{n,p},\de\bt^{n,p}\big)_{|t=0}\,=\,\big(0,\de\al^{n,p}(0),\de\bt^{n,p}(0)\big)$.

An energy estimate for the first equation in \eqref{eq:Cauchy} gives
\begin{align*}
\left\|\de\rho^{n,p}(t)\right\|_{L^2}\,\lesssim\,\int^t_0
\left(\left\|\big(\de\al,\de\bt\big)^{n-1,p}\right\|_{L^2}\,\left\|\nabla R^n\right\|_{L^\infty}\,+\,\left\|\al^{n+p-1}+\bt^{n+p-1}\right\|_{L^2}\,\left\|\de R^{n,p}(0)\right\|_{L^\infty}\right)\,
{\rm d}\t\,.
\end{align*}
Similarly, from the second and third equations in \eqref{eq:Cauchy}, we get
\begin{align*}
\left\|\big(\de\al,\de\bt\big)^{n,p}(t)\right\|_{L^2}\,&\leq\,\left\|\big(\de\al,\de\bt\big)^{n,p}(0)\right\|_{L^2} \\
&\qquad +\,C\int^t_0\Big(\left\|\big(\de\al,\de\bt\big)^{n-1,p}\right\|_{L^2}\,\left\|\big(\nabla\al^n,\nabla\bt^n,R^{n+p}\big)\right\|_{L^\infty} \\
&\qquad\qquad\qquad +\,
\left(\left\|\de\rho^{n,p}\right\|_{L^2}\,+\,\left\|\de R^{n,p}(0)\right\|_{L^\infty}\right)\,\left\|\big(\al^{n-1},\bt^{n-1}\big)\right\|_{L^2\cap L^\infty}\Big)\,{\rm d}\t\,,
\end{align*}
where we have also used the fact that $\de R^{n,p}\,=\,\de\rho^{n,p}\,-\,\de R^{n,p}(0)$.

By using \eqref{ind:energy}, \eqref{ind:Besov} and Lipschitz embeddings, we know that
\begin{align*}
&\sup_{t\in[0,T]}\Big(\left\|\nabla\big(R^n,\al^n,\bt^n\big)(t)\right\|_{L^\infty}\,+\,\left\|R^{n+p}(t)\right\|_{L^\infty}\Big) \\
&\qquad\qquad
+\int^T_0\Big(\left\|\al^{n+p-1},\bt^{n+p-1}\right\|_{L^2}\,+\,\left\|\big(\al^{n-1},\bt^{n-1}\big)\right\|_{L^2\cap L^\infty}\Big)\dt\,\leq\,C_T\,,
\end{align*}
for a constant $C_T$ depending on $T$, but uniform with respect to $n$ and $p$.
Therefore, from the previous inequalities and Gr\"onwall's lemma, we deduce
\begin{multline*} 
\sup_{[0,t]}\left\|\big(\de\rho,\de\al,\de\bt\big)^{n,p}\right\|_{L^2}\,\leq\,C_T\,\bigg(\left\|\de R^{n,p}(0)\right\|_{L^\infty}\,+\,\left\|\big(\de\al,\de\bt\big)^{n,p}(0)\right\|_{L^2} \\
+\,\int^t_0 \sup_{[0,\t]}\left\|\big(\de\rho,\de\al,\de\bt\big)^{n-1,p}\right\|_{L^2}{\rm d}\t\bigg)\,.
\end{multline*}
After setting
$$
F^n(t)\,:=\,\sup_{p\geq0}\sup_{[0,t]}\left\|\big(\de\rho,\de\al,\de\bt\big)^{n,p}\right\|_{L^2}\ \mbox{ and }\
D^n_0\,:=\,\sup_{p\geq0}\left(\left\|\de R^{n,p}(0)\right\|_{L^\infty}\,+\,\left\|\big(\de\al,\de\bt\big)^{n,p}(0)\right\|_{L^2}\right)\,,
$$
the previous estimate 
implies that, for all $t\in[0,T]$, one has
\begin{equation} \label{est:Cauchy2}
F^n(t)\,\leq\,C_T\,D^n_0\,+\,C_T\int^t_0F^{n-1}(\t)\,{\rm d}\t\,.
\end{equation}
A simple induction argument yields, for all $t\in[0,T]$, the bound
\[
F^n(t)\,\leq\,C_T\sum_{k=0}^{n-1}\left(\frac{\left(C_T\,T\right)^k}{k!}\,D^{n-k}_0\right)\,+\,\frac{\left(C_T\,T\right)^n}{n!}\,F^0(t)\,.
\]
This having been established, we notice that, owing to \eqref{conv:in-data}, 
we have that
\[
\lim_{n\ra+\infty}\sup_{p\geq0}\left(\left\|\de R^{n,p}(0)\right\|_{L^\infty}\,+\,\left\|\big(\de\al,\de\bt\big)^{n,p}(0)\right\|_{L^2}\right)\,=\,0\,.
\]
Hence, using dominated convergence, we can take the limit for $n\ra+\infty$ in \eqref{est:Cauchy2} and conclude, thanks to Gr\"onwall's lemma, that
$$
\lim_{n\ra+\infty}\,\sup_{p\geq0}\,\sup_{t\in[0,T]}\left\|\big(\de\rho,\de\al,\de\bt\big)^{n,p}(t)\right\|_{L^2}\,=\,0\,.
$$
This property implies that $\big(\rho^n\big)_n$, $\big(\al^n\big)_n$ and $\big(\bt^n\big)_n$ are Cauchy sequences in $C^0\big([0,T];L^2\big)$, thus they converge respectively to
some $\rho$, $\al$ and $\bt$ in that space. Define $R\,:=\,\rho\,-\,R_0$.

Observe that, owing to the embedding $L^2\hookrightarrow B^{-d/2}_{\infty,2}$, to uniform bounds and to interpolation, the sequences $\big(\al^n\big)_n$ and $\big(\bt^n\big)_n$ also strongly converge in any intermediate space $L^\infty_T(B^\s_{\infty,r})$, with $\s<s$, and in particular in $L^\infty([0,T]\times\R^d)$.
On the other hand, we have that $R^n\,=\,\rho^n\,-\,R^n_0$ strongly converges to $R$ in $L^\infty_T(L^2_{\rm loc})$.
Thus, it is easy to pass to the limit in the weak formulation of equations \eqref{eq:R^n} and \eqref{eq:al-bt^n}, finding that the triplet $(R,\al,\bt)$ is a weak solution to
the original problem \eqref{eq:MHDab}, for suitable pressure gradients $\nabla\pi_1$ and $\nabla\pi_2$.
Space regularity for $\big(R,\al,\bt\big)$ in $\B$ follows by uniform bounds and Fatou's property in Besov spaces. By the analysis preformed in the proof of Theorem \ref{th:symm},
we also know that $\nabla\pi_1=\nabla\pi_2$.

\paragraph{Regularity of the pressure terms, and final checks.} Let us now devote some attention to the study of the regularity of $\nabla\pi_1$. First of all, similar computations as the ones
leading to \eqref{eq:abFromCurl} give the bound
\begin{equation} \label{est:press_1}
\left\|\nabla\pi_1\right\|_{L^2\cap\B}\,\lesssim\,\left\|\nabla\pi_1\right\|_{L^2}\,+\,\left\|\Delta\pi_1\right\|_{B^{s-1}_{\infty,r}}\,.
\end{equation}
Now, applying the $\div$ operator to the first equation in \eqref{eq:MHDab}, we deduce that $\pi_1$ satisfies the elliptic equation
\begin{equation} \label{eq:ell}
-\Delta\pi_1\,=\,\div F\,,\qquad\mbox{ where }\qquad F\,:=\,(\bt\cdot\nabla)\al\,+\,\frac{1}{2}\,R\,\mf C(\al+\bt)\,.
\end{equation}
On the one hand, an application of the Lax-Milgram theorem implies that
\[
\left\|\nabla\pi_1\right\|_{L^2}\,\lesssim\,\left\|F\right\|_{L^2}\,\lesssim\,\|\bt\|_{L^2}\,\|\nabla\al\|_{L^\infty}\,+\,\|R\|_{L^\infty}\,\|\al+\bt\|_{L^2}\,,
\]
so that $\nabla\pi_1\in L^\infty_T(L^2)$. On the other hand, we observe that, owing to the divergence-free condition on $\al$ and $\bt$, one has
$\div\big((\bt\cdot\nabla)\al\big)\,=\,\nabla\bt:\nabla\al\,=\,\sum_{j,k}\d_j\bt^k\,\d_k\al^j$. Therefore,
\begin{align*}
\left\|\Delta\pi_1\right\|_{B^{s-1}_{\infty,r}}\,&\lesssim\,\left\|\nabla\bt:\nabla\al\right\|_{B^{s-1}_{\infty,r}}\,+\,\left\|\div\big(R\,\mf C(\al+\bt)\big)\right\|_{B^{s-1}_{\infty,r}} \\
&\lesssim\,\|\bt\|_{\B}\,\|\al\|_{\B}\,+\,\|R\|_{\B}\,\|\al+\bt\|_{\B}\,.
\end{align*}
Notice that, when $s=1$, the estimate $\left\|\nabla\bt:\nabla\al\right\|_{B^{s-1}_{\infty,r}}\,\lesssim\,\|\bt\|_{\B}\,\|\al\|_{\B}$
still holds true. For proving this, one has to argue as in the proof of Lemma \ref{l:CommBPInfinite}, and use the divergence-free condition on $\al$ (or $\bt$) in order to bound the remainders
appearing in the Bony decomposition of the previous product $\nabla\bt:\nabla\al$.
In the end, we deduce that $\Delta\pi_1$ belongs to $L^\infty_T(B^{s-1}_{\infty,r})$.
Thus, from \eqref{est:press_1} and those two pieces of information, we conclude that $\nabla\pi_1\in L^\infty_T(L^2\cap\B)$. 

This having been established, we can use classical results on solutions to transport equations in Besov spaces (recall Theorem \ref{th:transport} above)
to infer the claimed time continuity of $R$, $\al$ and $\bt$ with values in $\B$, and of $\al$ and $\bt$ with values in $L^2$. Combining these properties with equation
\eqref{eq:ell}, we discover that also $\nabla\pi_1$ belongs to $C^0_T(L^2\cap \B)$. Since $b=(\al-\bt)/2$ belongs to the same space as well, from that property and Theorem \ref{th:symm}
it is easy to recover that $\nabla\Pi\in C^0_T(L^2\cap \B)$. Finally, the claimed regularity properties for the time derivatives $\d_tR$, $\d_tu$ and $\d_tb$
follow from an inspection of the equations in \eqref{i_eq:MHD-I}.

The proof of the existence is now completed.

\subsection{The continuation criterion in Els\"asser variables} \label{ss:proof-cc}

In this section, we seek to prove the continuation criterion of Theorem \ref{th:cont-crit} for the solutions of \eqref{i_eq:MHD-I} in $B^s_{\infty, r}$,
where the couple $(s, r)$ satisfies the Lipschitz condition \eqref{i_eq:Lip}. Thanks to the equivalence stated in Theorem \ref{th:symm}, it is enough to prove an analogous
continuation criterion for the Els\"asser variables $(R,\al,\bt)$ solving system \eqref{eq:MHDab}.

\begin{prop}\label{p:ContInfinite}
Let $\big(R_0,\al_0, \bt_0\big)\in \left(B^s_{\infty, r}\right)^3$, with $\D(\al_0) = \D(\bt_0) = 0$ and $\al_0,\bt_0\in L^2$.
Given a time $T > 0$, let $(R,\al, \bt)$ be a solution of \eqref{eq:MHDab} on $[0,T[\,$, related
to that initial datum and belonging to the space $L^\infty_t(\B)\times L^\infty_t(L^2\cap \B)\times L^\infty_t(L^2\cap \B)$ for any $0\leq t<T$.
Assume moreover that 
\begin{equation}\label{eq:ContCritPInfinite}
\int_0^T \Big( \|\nabla \al \|_{L^\infty} + \|\nabla \bt \|_{L^\infty} \Big) \dt\, <\, +\infty\,.
\end{equation}

Then $(R,\al, \bt)$ can be continued beyond $T$ into a solution of \eqref{eq:MHDab} with the same regularity.
\end{prop}


\begin{proof} 
The previous statement is somewhat classical for quasi-linear hyperbolic problems. A standard continuation argument, which is based on uniqueness of solutions (and which we omit here for the sake of
conciseness), allows us to reduce the proof to showing that, under assumption \eqref{eq:ContCritPInfinite}, the solution $(R,\al,\bt)$ remains bounded in
$L^\infty_T(\B)\times L^\infty_T(L^2\cap\B)\times L^\infty_T(L^2\cap\B)$.

For obtaining this latter property, the main point is to exhibit more precise estimates for the non-linear terms than the ones we used in \eqref{eq:ContCritLInftyEstimate}.
To begin with, we make full use of the inequalities of Lemma \ref{l:CommBPInfinite} to get, instead of \eqref{eq:pinftyIN1}, the following bounds:
\begin{multline*}
\big\| \L (\nabla \al, \nabla \bt)  \big\|_{B^{s-1}_{\infty, r}} + \big\|  \curl \big( R \mathfrak{C} (\al + \bt) \big) \big\|_{B^{s-1}_{\infty, r}} \\ 
\lesssim \Big( \|\nabla \al \|_{L^\infty} + \|\nabla \bt \|_{L^\infty} + \| \al+\bt \|_{L^\infty} + \| R \|_{L^\infty} \Big)\, \big\| (R,\al, \bt) \big\|_{B^s_{\infty, r}}\,.
\end{multline*}
We do the same with the commutator terms, employing Lemma  \ref{l:CommBCD} above: having the precaution of distinguishing among the cases $s>2$, $s<2$ and $s=2$ and $r\neq 1$, like we have done
for obtaining \eqref{pinftyIN2}, we can replace that inequality by
\begin{multline*}
2^{j(s-1)} \Big( \left\| \big[ \bt \cdot \nabla, \Delta_j \big] X \right\|_{L^\infty} + \left\| \big[ \al \cdot \nabla, \Delta_j \big] Y \right\|_{L^\infty} \Big) +
2^{js}\left\| \big[ (\al + \bt) \cdot \nabla, \Delta_j \big] R \right\|_{L^\infty} \\
\lesssim c_j(t)\,\Big( \|\nabla \al \|_{L^\infty} + \|\nabla \bt \|_{L^\infty} + \| \nabla R \|_{L^\infty} \Big)\, \big\| (R\,\al, \bt) \big\|_{B^s_{\infty, r}}\,,
\end{multline*}
where we have also used the trivial fact that, by definition, $\|X\|_{L^\infty}\,\leq\,2\,\|\nabla\al\|_{L^\infty}$ and $\|X\|_{B^{s-1}_{\infty,r}}\,\leq\,2\,\|\al\|_{\B}$
(and the same relations hold for $Y$ and $\beta$).
Here above, as usual, the sequence $\big(c_j(t)\big)_j$ belongs to the unit ball of $\ell^r$.

Thanks to the previous bounds, we get the following estimate: for all $t\in[0,T[\,$, one has
\begin{multline} \label{est:cont-crit}
\left\|R(t)\right\|_{\B}\,+\,\big\| \big(X(t), Y(t)\big) \big\|_{B^{s-1}_{\infty, r}}\,\lesssim\,\left\|R_0\right\|_{\B}\,+\,\big\| \big(X_0, Y_0\big) \big\|_{B^{s-1}_{\infty, r}} \\
+ \int_0^t  \Big( \| \nabla \al \|_{L^\infty} + \| \nabla \bt \|_{L^\infty} + \|R\|_{W^{1,\infty}} + \| \al + \bt \|_{L^\infty} \Big)\,
\big\|\big(R,\al, \bt\big) \big\|_{B^s_{\infty, r}} {\rm d} \tau\,.
\end{multline}
It remains to find estimates for $\| \al + \bt \|_{L^\infty}$ and $\|R\|_{W^{1,\infty}}$. Let us start by dealing with the first term: separating low and high frequencies yields
\begin{equation*}
\|\al + \bt\|_{L^\infty} \leq \|\Delta_{-1} (\al + \bt)\|_{L^\infty} + \sum_{j \geq 0} \|\Delta_j (\al + \bt) \|_{L^\infty}.
\end{equation*}
Using the first Bernstein inequality (see Lemma \ref{l:bern} above), we have $\|\Delta_{-1} (\al + \bt)\|_{L^\infty} \leq \|\al + \bt\|_{L^2}$, for which we can employ \eqref{est:en-ab}
to deduce that this quantity is bounded with respect to time. As for the high frequency terms, we write, thanks to the second Bernstein inequality, 
\begin{equation*}
\sum_{j \geq 0} \|\Delta_j (\al + \bt) \|_{L^\infty} \lesssim \sum_{j \geq 0} 2^{-j} \|\Delta_j \nabla (\al + \bt) \|_{L^\infty} \lesssim \|\nabla (\al + \bt)\|_{L^\infty}\,.
\end{equation*}
Therefore, we have
\begin{equation} \label{est:a-b_inf}
\| \al + \bt \|_{L^\infty}\, \lesssim\, \big\|\big(\al_0, \bt_0\big)\big\|_{L^2}\,e^{c\,t\,\| R_0 \|_{L^\infty}}\, +\, \| \nabla (\al + \bt)\|_{L^\infty}\,.
\end{equation}
Concerning the second term, \tsl{i.e.} $\|R\|_{W^{1,\infty}}$, we notice that estimate \eqref{est:R_inf} still holds true. On the other hand, combining
a simple differentiation of the equation for $R$ and an $L^\infty$ estimate for the resulting transport equation, we are led to
\begin{equation} \label{est:DR}
\big\|\nabla R(t)\big\|_{L^\infty_T(L^{\infty})}\, \leq\, \big\|\nabla R_0 \big\|_{L^{\infty}}\, \exp\left( C \int_0^T \big\|\nabla(\al + \bt)\big\|_{L^\infty}\, \dt \right)\,.
\end{equation}

Thus, under assumption \eqref{eq:ContCritPInfinite}, we deduce that
\begin{align*}
\sup_{t\in[0,T[}\left\|R(t)\right\|_{W^{1,\infty}}\,&\lesssim\,\left\|R_0\right\|_{W^{1,\infty}} \\
\| \al(t) + \bt(t) \|_{L^2\cap L^\infty}\, &\lesssim\, \big\|\big(\al_0, \bt_0\big)\big\|_{L^2}\,e^{C\,T}\, +\, \big\| \nabla\big(\al + \bt\big)(t)\big\|_{L^\infty}\,,
\end{align*}
where the second inequality holds true for every $t\in[0,T[\,$.
Inserting those bounds into \eqref{est:cont-crit} and using the fact that
$$
\big\|\big(\al(t),\bt(t)\big)\big\|_{L^2}+\big\| \big(X(t), Y(t)\big) \big\|_{B^{s-1}_{\infty, r}}\,\sim\,\big\|\big(\al(t),\bt(t)\big)\big\|_{L^2\cap\B}\,,
$$
which holds in view of \eqref{est:equiv-norm}, we finally find
\begin{align*} 
&\left\|R(t)\right\|_{\B}+\big\|\big(\al(t),\bt(t)\big)\big\|_{L^2\cap\B}\,\lesssim\,e^{CT}\bigg(\left\|R_0\right\|_{\B}+\big\|\big(\al_0,\bt_0\big)\big\|_{L^2\cap\B} \\
&\quad 
+ \int_0^t  \Big( \| \nabla \al \|_{L^\infty} + \| \nabla \bt \|_{L^\infty} + \|R_0\|_{W^{1,\infty}} + \big\|\big(\al_0, \bt_0\big)\big\|_{L^2}\Big)\,
\left(\|R\|_{\B}\,+\,\big\|\big(\al, \bt\big) \big\|_{L^2\cap\B}\right) {\rm d} \tau\bigg).
\end{align*}
An application of Gr\"onwall's lemma, together with assumption \eqref{eq:ContCritPInfinite}, finally yields
$$
\sup_{t\in[0,T[}\left(\left\|R(t)\right\|_{\B}+\big\|\big(\al(t),\bt(t)\big)\big\|_{L^2\cap\B}\right)\,<\,+\,\infty\,,
$$
and this bound, as already claimed at the beginning of the proof, implies the result.
\end{proof}


\section{Improved lifespan in the two-dimensional case} \label{s:lifespan}

In this section we prove Theorem \ref{th:lifespan}, namely an improved lower bound on the lifespan of the solutions in the two-dimensional case.


\subsection{The structure of the bilinear term} \label{ss:structure}
In all that follows, we consider the case of space dimension $d=2$. Thus, the $\curl$ of a function $f : \mathbb{R}^2 \longrightarrow \mathbb{R}^2$ can be identified with the scalar function
\begin{equation*}
\curl(f) = \partial_1 f_2 - \partial_2 f_1\,.
\end{equation*}
As a consequence, the bilinear term $\L (\nabla \al, \nabla \bt)$ introduced in \eqref{eq:LOperator} reads
\begin{equation}\label{eq:2DLOperator}
\L(\nabla \al,\nabla \bt) = \partial_1 \al_1\, \left( \partial_1 \bt_2 + \partial_2 \bt_1 \right)\, +\, \partial_2 \bt_2\, \left( \partial_1 \al_2 + \partial_2 \al_1 \right)\,.
\end{equation}
By noting $\mf T_2$ the linear space of traceless $2 \times 2$ matrices, we may see $\L$ as a bilinear operator $\L : \mf T_2 \times \mf T_2 \tend \mathbb{R}$ which is, by virtue of \eqref{eq:2DLOperator}, and in the special case $d = 2$, skew-symmetric. Now, by writing $\al$ and $\bt$ as functions of $u$ and $b$, and by noting that, by skew-symmetry, one has $\L(\nabla f, \nabla f) = 0$ for any divergence-free vector field $f$, we get
\begin{equation}\label{eq:LRewriteUB}
\L(\nabla \al, \nabla \bt)\, =\, \L \big( \nabla (u + b), \nabla (u - b) \big)\,=\, -\,2\, \L (\nabla u, \nabla b)\,.
\end{equation}

\begin{rmk} \label{r:life-b}
Owing to the bilinearity of $\L$, a result similar to Theorem \ref{th:lifespan} holds true for solutions to the Els\"asser system \eqref{eq:MHDab} with respect to small values of $\bt$,
\emph{regardless the space dimension}.
Moreover, as we have a better equation for $\bt$ in \eqref{eq:MHDab} then we have for $b$ in \eqref{i_eq:MHD-I}, we have no need of initial data in the higher regularity space
$B^2_{\infty, 1}$. Specifically, assuming $R\equiv0$ for simplicity, we have, in all dimensions, 
\begin{equation*}
T^*\, \geq\, \frac{C}{\big\|\big(\al_0, \bt_0\big)\big\|_{B^1_{\infty, 1} \cap L^2 }}\;
\bigg[\log\big(1\,+\,C\,\cdot\,\big)\bigg]^{\bigcirc2}\left( \frac{\big\|\big(\al_0, \bt_0\big) \big\|_{B^1_{\infty, 1} \cap L^2 } }{\| \bt_0 \|_{B^1_{\infty, 1}}} \right)\,.
\end{equation*}
However, the regime of small $\beta$ is not of great interest, since in that case the ``true'' MHD system \eqref{i_eq:MHD-I} degenerates.
For this reason, we focus on the result of Theorem \ref{th:lifespan}.
\end{rmk}

\subsection{Proof of Theorem \ref{th:lifespan}} \label{ss:life}

We now present the proof of Theorem \ref{th:lifespan}.
As usual in our approach, we work in Els\"asser variables, thanks to the equivalence provided by Theorem \ref{th:symm} in our functional framework.

As was already the case in \cite{DF} and \cite{F-L} in dealing with different systems, the main idea of the proof is to take advantage
of the linear estimates of Theorem \ref{th:AnnInnLinTV} for the transport equations in Besov spaces with $0$ regularity index.
For this, it is fundamental to work with the vorticity formulation \eqref{eq:MHDVorticity} of the equations, since $X,Y\in B^{0}_{\infty,1}$.

We divide the proof of Theorem \ref{th:lifespan} into three main steps.

\subsubsection*{Step 1: an estimate with loss of derivatives.}
Although the data possess additional regularity, by virtue of Theorem \ref{th:cont-crit}, it is enough to bound the lifespan in the lower regularity space $B^1_{\infty,1}$.
For this, we define
\begin{equation*}
\E(t)\,:=\, \big\| \big(\al(t), \bt(t)\big) \|_{L^2}\, +\, \big\| \big(X(t), Y(t)\big) \big\|_{B^0_{\infty, 1}}\,\sim\,\big\| \big(\al(t), \bt(t)\big) \|_{L^2\cap B^1_{\infty,1}} \,.
\end{equation*}

Let us focus for a while on bounding the Besov norm of the vorticities $X$ and $Y$. Since they solve system \eqref{eq:MHDVorticity}, Theorem \ref{th:AnnInnLinTV} provides us with the estimate
\begin{multline*}
\big\| \big(X(t), Y(t)\big) \big\|_{B^0_{\infty, 1}}\,\lesssim\, \left( 1 + \int_0^t \|(\nabla \al, \nabla \bt) \|_{L^\infty}\,{\rm d}\t \right) \\
\times\,\left\{ \big\| \big(X_0, Y_0\big) \big\|_{B^0_{\infty, 1}} + \int_0^t \Big( \|\L(\nabla \al, \nabla \bt) \|_{B^0_{\infty, 1}} +
\left\|\curl\big(R\,\mf C(\al + \bt)\big)\right\|_{B^0_{\infty, 1}} \Big) {\rm d} \tau \right\}\,.
\end{multline*}
The key point of our proof is that, owing to \eqref{eq:LRewriteUB} above,
the integral terms in the second line are in fact linear with respect to both $R$ and $\nabla b$.
By use of Lemma \ref{l:CommBPInfinite} and relation \eqref{eq:LRewriteUB}, we gather
\begin{align*}
\left\| \curl \big( R\, \mathfrak{C}(\al + \bt) \big) \right\|_{B^{0}_{\infty, 1}}\, &\lesssim\,\| R \|_{L^\infty}\, \|\al + \bt\|_{B^{1}_{\infty, 1}}\, +\,
\| \al + \bt \|_{L^\infty}\, \| R \|_{B^1_{\infty, 1}} \\
&\lesssim\,\|R\|_{B^1_{\infty,1}}\,\E(t) \\
\big\| \L(\nabla \al, \nabla \bt) \big\|_{B^{0}_{\infty, 1}}\,=\,2\,\big\| \L(\nabla u, \nabla b) \big\|_{B^{0}_{\infty,1}}\,&\lesssim\,\| \nabla u \|_{L^\infty}\, \|b \|_{B^1_{\infty, 1}}\, +\,
\| \nabla b \|_{L^\infty}\, \| u \|_{B^1_{\infty, 1}} \\
&\lesssim\,\|b\|_{B^1_{\infty,1}}\,\E(t) \,.
\end{align*}
Thanks to those inequalities and the energy estimate \eqref{est:en-ab}, we deduce
\begin{equation}\label{eq:2DIntIneq}
\E(t)\,\lesssim\,\left( 1 + \int_0^t \E(\tau)\, {\rm d} \tau \right)\,
\left\{ \E(0)\,e^{c\|R_0\|_{L^\infty}t} \,+\, \int_0^t \E(\tau)\,\Big( \| b \|_{B^1_{\infty, 1}} + \| R \|_{B^1_{\infty, 1}} \Big)\, {\rm d} \tau \right\}\,.
\end{equation}

In order for this to be practical, we must finds bounds on $b$ and $R$ which depend only on the initial data. First of all, 
Theorem \ref{th:transport} implies that
\begin{equation}\label{eq:REstimate}
\| R \|_{B^1_{\infty, 1}}\, \leq\, \| R_0 \|_{B^1_{\infty, 1}}\, \exp \left( C \int_0^t \E(\tau)\, {\rm d} \tau \right)\,.
\end{equation}
Next, applying the dyadic block $\Delta_j$ to the transport equation solved by $b$, see system \eqref{i_eq:MHD-I}, we get
\begin{equation*}
\big( \partial_t + u\cdot \nabla \big) \Delta_j b = \Delta_j\big( (b \cdot \nabla) u\big) + \big[ u \cdot \nabla, \Delta_j  \big]b\,.
\end{equation*}
Using Corollary \ref{c:tame} and Lemma \ref{l:CommBCD}, we can estimate the right-hand side by
\begin{equation*}
2^j \left\| \Delta_j\big( (b \cdot \nabla) u\big) \right \|_{L^\infty} + 2^j \left\| \big[ u \cdot \nabla, \Delta_j  \big]b \right\|_{L^\infty}\,\lesssim\,c_j(t)\,\|b\|_{B^1_{\infty, 1}}\,
\|u\|_{B^2_{\infty, 1}}\,,
\end{equation*}
where, as usual, the sequence $\big(c_j(t)\big)_j$ belongs to $\ell^1$, with $\sum_{j\geq -1} c_j(t) = 1$. This gives rise to the differential inequality
\begin{equation*}
\big\| b(t) \big\|_{B^1_{\infty, 1}}\, \lesssim\, \big\| b_0 \big\|_{B^1_{\infty, 1}}\, +\, \int_0^t \|b(\t)\|_{B^1_{\infty, 1}}\, \|u(\t)\|_{B^2_{\infty, 1}}\,  {\rm d} \tau\,,
\end{equation*}
and an application of Gr\"onwall's lemma yields
\begin{equation}\label{eq:BEstimate}
\big\| b(t) \big\|_{B^1_{\infty, 1}}\, \lesssim\, \big\| b_0 \big\|_{B^1_{\infty, 1}} \,\exp\left( C \int_0^t \| u(\t) \|_{B^2_{\infty, 1}}\,{\rm d}\t  \right)\,.
\end{equation}
Inserting both estimates \eqref{eq:REstimate} and \eqref{eq:BEstimate} into inequality \eqref{eq:2DIntIneq} gives
\begin{multline}\label{eq:LRGronEstimate}
\E(t)\, \lesssim\,\left( 1 + \int_0 \E(\tau)\, {\rm d} \tau \right) \\
\times\left\{ \E(0)\,e^{ct\|R_0\|_{L^\infty}} \,+\,\big\|\big(R_0, b_0\big)\big\|_{B^1_{\infty, 1}}\int_0^t\E(\tau)\,
\exp \left( \int_0^\tau \big\| \big(\al(s), \bt(s)\big)\big\|_{B^2_{\infty, 1}}\, {\rm d}s \right)\,{\rm d} \tau \right\}\,.
\end{multline}

\subsubsection*{Step 2: bounding the higher order norms.}

Inequality \eqref{eq:LRGronEstimate} presents an apparent one derivative loss. The key point, now, is to find a way to write $B^2_{\infty, 1}$ bounds for $(\al, \bt)$ by using only the function $\E(t)$.

For simplicity of notation, let us introduce the quantity
$$
\H(t)\,:=\,\big\|R(t)\big\|_{B^2_{\infty,1}}\,+\,\big\|\big(\al(t),\bt(t)\big)\big\|_{L^2\cap B^2_{\infty,1}}
$$
and notice that, by virtue of \eqref{est:equiv-norm}, we have
\begin{equation} \label{est:H_first}
\H(t)\,\sim\,\|R(t)\|_{B^2_{\infty, 1}}\, +\, \big\|\big(\al(t), \bt(t)\big)\big\|_{L^2}\, +\, \big\|\big(X(t), Y(t)\big)\big\|_{B^1_{\infty, 1}}\,.
\end{equation}

The function $\H(t)$ is the higher regularity quantity which we wish to estimate by $\E(t)$. For this, the starting point is
inequality \eqref{est:cont-crit}, which was used to prove the continuation criterion, where we take $s=2$ and $r=1$. Inserting \eqref{est:a-b_inf}, \eqref{est:R_inf} and \eqref{est:DR}
directly therein, we are led to
\begin{multline*}
\left\|R(t)\right\|_{B^2_{\infty,1}}+\big\| \big(X(t), Y(t)\big) \big\|_{B^{1}_{\infty,1}}\,\lesssim\,\H(0)\,+\,
\int^t_0\Big(\| \nabla \al \|_{L^\infty} + \| \nabla \bt \|_{L^\infty} + \|R_0\|_{L^{\infty}} \\
+\|\nabla R_0\|_{L^{\infty}}\,e^{C\int^\t_0\|\nabla(\al,\bt)(s)\|_{L^\infty}{\rm d}s}\,+\,\big\|\al_0 + \bt_0 \|_{L^2}\,e^{c\t\|R_0\|_{L^\infty}} \Big)\,
\H(\t) {\rm d} \tau\,.
\end{multline*}
where we have used also \eqref{est:H_first} to control the initial data and the $B^2_{\infty,1}$ norm of $(R,\al,\bt)$ appearing inside the integral. 
From this inequality and \eqref{est:H_first} again, we deduce
\begin{equation*}
\H(t)\,\lesssim\,e^{ct\|R_0\|_{L^\infty}}\H(0) +\int_0^t \H(\tau) \left\{ \E(\tau) +
\left(\|R_0\|_{B^1_{\infty, 1}}+\big\|\al_0+\bt_0\big\|_{L^2}\right)\,e^{C \int_0^\tau \E(s) {\rm d}s}  \right\} {\rm d} \tau\,.
\end{equation*}
Notice that the first exponential factor multiplies only $\H(0)$, since it appears from the control of the $L^2$ norm of $(\al,\bt)$. For the exponential factor appearing inside the integral, instead, we have used the following trick: $t\|R_0\|_{L^\infty}\,=\,\int^t_0\|R(\t)\|_{L^\infty}{\rm d}\t\,\leq\,\int^t_0\E(\t){\rm d}\t$.

In the end,  applying Gr\"onwall's lemma to the previous estimate gives the upper bound
\begin{equation}\label{eq:HRGronEstimate}
\H(t)\,\lesssim\,e^{ct\|R_0\|_{L^\infty}}\,\H(0)\,\exp\left\{ C\int_0^t \left[ \E(\tau) +
\left(\|R_0\|_{B^1_{\infty, 1}}+\big\|\al_0+\bt_0\big\|_{L^2}\right)\,e^{C \int_0^\tau \E(s) {\rm d}s} \right]\, {\rm d} \tau  \right\}\,.
\end{equation}

\subsubsection*{Step 3: end of the proof.}
All that remains to do is to use inequalities \eqref{eq:LRGronEstimate} and \eqref{eq:HRGronEstimate} to find bounds on $\E(t)$ on a good time interval. With that idea in mind, define
\begin{equation*}
T^*\,:=\,\sup \left\{ T\,>\,0 \, \Big| \quad \big\| \big(R_0, b_0\big) \big\|_{B^1_{\infty, 1}}\, \int_0^T \E(\tau) \exp \left( \int_0^\tau \H(s) {\rm d}s \right)\, {\rm d} \tau\,\leq\,
\E(0)\,e^{c\H(0)T}\right\}     
\end{equation*}

To simplify notations in the next computations, we introduce the functions
$$
f(t)\,:=\,e^{c\H(0)t}
\qquad\mbox{ and }\qquad F(t)\,:=\,\int^t_0f(\t)\,{\rm d}\t\,=\,\frac{1}{c\,\H(0)}\left(e^{c\H(0)t}-1\right)\,.
$$
Using the definition of $T^*$, from \eqref{eq:LRGronEstimate} and Gr\"onwall's lemma, it is straightforward to deduce, for all times $t\in[0,T^*]$, the bounds
\begin{equation}\label{eq:EstimateForE}
\int_0^t \E(\tau)\, {\rm d} \tau \,\leq\, e^{C\E(0)F(t)}-1 \qquad \text{ and } \qquad \E(t) \,\leq\, C\, \E(0)\,f(t)\,e^{C\E(0)F(t)}\,.
\end{equation}
We now employ these estimates in \eqref{eq:HRGronEstimate}: for all $t\in[0,T^*]$, we infer
\begin{align*}
\H(t)\,&\leq\, C\,f(t)\,\H(0)\,\exp \bigg\{ C\,\left(e^{C\E(0)F(t)}-1\right) \\
&\qquad\qquad\qquad\qquad\qquad
+\,C\left(\|R_0\|_{B^1_{\infty, 1}}+\big\|\al_0+\bt_0\big\|_{L^2}\right)\int_0^t \exp\Big( e^{C\E(0)F(\t)} - 1 \Big){\rm d} \tau  \bigg\} \\
&\leq\, C\,f(t)\,\H(0)\,\exp \bigg\{ C\,\left(e^{C\H(0)F(t)}-1\right) \\
&\qquad\qquad\qquad\qquad\qquad
+\,C\left(\|R_0\|_{B^1_{\infty, 1}}+\big\|\al_0+\bt_0\big\|_{L^2}\right)\int_0^t \exp\Big( e^{C\H(0)F(\t)} - 1 \Big){\rm d} \tau  \bigg\}\,.
\end{align*}
After noticing that $f(t)\geq1$ and $F(t)\geq0$, we can compute
\begin{align*}
 \int_0^t \exp\Big( e^{C\H(0)F(\t)} - 1 \Big){\rm d} \tau\,&\leq\,\frac{1}{C\,\H(0)}\int^t_0C\,\H(0)\,f(\t)\,e^{C\H(0)F(\t)} \exp\Big( e^{C\H(0)F(\t)} - 1 \Big){\rm d} \tau \\
 &=\,\frac{1}{C\,\H(0)}\,\left(\exp\Big( e^{C\H(0)F(t)} - 1 \Big)\,-\,1\right)\,.
\end{align*}
Observing that $x\leq e^x-1$ for $x\geq0$ and that
$$
\frac{\|R_0\|_{B^1_{\infty, 1}}+\big\|\al_0+\bt_0\big\|_{L^2}}{\H(0)}\,\leq\,1\,,
$$
we finally find (take $x=e^{C\H(0)F(t)}-1$ in the previous estimate for $\H$) the inequality
\begin{equation*}
 \H(t)\,\leq\, C\,f(t)\,\H(0)\,\exp\bigg\{C\,\Big(\exp\left(e^{C\H(0)F(t)}-1\right)\,-\,1\Big)\bigg\}\,,
\end{equation*}
which holds true for all $t\in[0,T^*]$.

It remains to find a control on the integral of $\H(t)$. For this, we can use the same trick as above of introducing a total derivative in the inequalities.
Indeed, after observing that both $e^{C\H(0)F(t)}\geq1$ and $\exp\big(e^{C\H(0)F(t)}-1\big)\geq1$, we can compute
\begin{align*}
\int^t_0\H(\t)\,{\rm d}\t\,&\leq\,C\,\H(0)\int^t_0f(\t)\,\exp\bigg\{C\,\Big(\exp\left(e^{C\H(0)F(\t)}-1\right)\,-\,1\Big)\bigg\}\,{\rm d}\t \\
&\leq\,C\,\H(0)\int^t_0f(\t)\,e^{C\H(0)F(\t)}\,\exp\big(e^{C\H(0)F(\t)}-1\big) \\
&\qquad\qquad\qquad\qquad\qquad\qquad\qquad\qquad
\times\exp\bigg\{C\,\Big(\exp\left(e^{C\H(0)F(\t)}-1\right)\,-\,1\Big)\bigg\}\,{\rm d}\t \\
&=\,C\,\left(\exp\bigg\{C\,\Big(\exp\left(e^{C\H(0)F(t)}-1\right)\,-\,1\Big)\bigg\}\,-\,1\right)\,.
\end{align*}
Having established this estimate, we can now employ it, together with \eqref{eq:EstimateForE}, in the inequality which defines $T^*$:
using again the trick of making a total time derivative appear, we see that, for all times $t \in [0, T^*]$, one has
\begin{equation*}
\begin{split}
&\big\| \big(R_0, B_0\big) \big\|_{B^1_{\infty, 1}} \int_0^t \E(\tau)\,\exp \left( \int_0^\tau  \H(s) {\rm d}s \right)\, {\rm d} \tau \\
&\qquad\leq\, C\,\big\| \big(R_0, B_0\big) \big\|_{B^1_{\infty, 1}} \int_0^t \E(0)f(\t) e^{C\E(0)F(\t)} \\
&\qquad\qquad\qquad\qquad\qquad\qquad\qquad
\times \exp\bigg\{C\,\left(\exp\bigg\{C\,\Big(\exp\left(e^{C\H(0)F(\t)}-1\right)\,-\,1\Big)\bigg\}\,-\,1\right) \bigg\}  {\rm d} \tau \\
&\qquad\leq\, C\,\big\| \big(R_0, B_0\big) \big\|_{B^1_{\infty, 1}} \int_0^t \H(0)f(\t) e^{C\H(0)F(\t)} \\
&\qquad\qquad\qquad\qquad\qquad\qquad\qquad
\times \exp\bigg\{C\,\left(\exp\bigg\{C\,\Big(\exp\left(e^{C\H(0)F(\t)}-1\right)\,-\,1\Big)\bigg\}\,-\,1\right) \bigg\}  {\rm d} \tau \\
&\qquad\leq\,C\,\big\| \big(R_0, B_0\big) \big\|_{B^1_{\infty, 1}}\,\bigg[\exp\bigg\{C\,\left(\exp\bigg\{C\,\Big(\exp\left(e^{C\H(0)F(t)}-1\right)\,-\,1\Big)\bigg\}\,-\,1\right) \bigg\}\,-\,1\bigg]\,.
\end{split}
\end{equation*}
On the other hand, by definition of $T^*$ and continuity, at $t=T^*$ we must have
$$
\big\| \big(R_0, B_0\big) \big\|_{B^1_{\infty, 1}} \int_0^{T^*} \E(\tau)\,\exp \left( \int_0^\tau  \H(s) {\rm d}s \right)\, {\rm d} \tau\,=\,\E(0)\,e^{cT^*\|R_0\|_{L^\infty}}\,\geq\,\E(0)\,.
$$
Thus, putting those last bounds together, we deduce that
\begin{align*}
C\,\H(0)\,F(T^*)\,\geq\,\bigg[\log\big(1\,+\,C\,\cdot\,\big)\bigg]^{\bigcirc4}\left(\frac{\E(0)}{\big\| \big(R_0, B_0\big) \big\|_{B^1_{\infty, 1}}}\right)\,.
\end{align*}
Using the definition of $F$ finally gives the claimed lower bound on $T^*$, ending in this way the proof of Theorem \ref{th:lifespan}.


\begin{rmk} \label{r:life-C}
In the case when the matrix $\mf C$ is skew-symmetric, \tsl{i.e.} it satisfies (up to a constant factor) relation \eqref{eq:MatrixC},
%
the energy is conserved, \tsl{i.e.} \eqref{eq:BasicENantiSymmetric} holds true. Thus, the exponential factors in front of the initial data $\E(0)$
and $\H(0)$ disappear from \eqref{eq:LRGronEstimate} and \eqref{eq:HRGronEstimate}, as well as from the right-hand side of the inequality defining $T^*$.
So, we can take $f(t)\equiv1$ and $F(t)=t$. In the end, we get
$$
C\,\H(0)\,T^*\,\geq\,\bigg[\log\big(1\,+\,C\,\cdot\,\big)\bigg]^{\bigcirc4}\left(\frac{\E(0)}{\big\| \big(R_0, B_0\big) \big\|_{B^1_{\infty, 1}}}\right)\,,
$$
which proves the claim stated in Remark \ref{r:lifespan} after Theorem \ref{th:lifespan}.
\end{rmk}


\end{document}